\documentclass[fontsize=10pt,a4paper]{scrartcl}
\global\def\sc{\egroup\textsc\bgroup}

\usepackage{amsmath, amssymb, amsthm, mathabx}
\usepackage{graphicx}
\usepackage{xcolor}
\usepackage{comment}
\usepackage{fullpage}
\usepackage{hyperref}
\usepackage{float}
\usepackage{mdframed}
\usepackage[T1]{fontenc}
\usepackage[utf8]{inputenc}
\usepackage{XCharter}
\usepackage{tikz}
\usepackage{MnSymbol}

\usepackage[justification=justified,
   format=plain]{caption} 

\DeclareOldFontCommand{\bf}{\normalfont\bfseries}{\mathbf}

\newtheorem{theorem}{Theorem}
\newtheorem{lemma}[theorem]{Lemma}

\newtheorem{proposition}[theorem]{Proposition}
\newtheorem{observation}[theorem]{Observation}

\def\cA{\mathcal{A}}
\def\cB{\mathcal{B}}
\def\cR{\mathcal{R}}
\def\cG{\mathcal{G}}
\def\cW{\mathcal{W}}

\def\nm{n-1}

\def\ob{\overline{b}}

\newcommand{\rr}{{R}}
\newcommand{\rc}{\mathcal{R}}
\newcommand{\dc}{\mathcal{D}}

\newcommand{\btri}{\filledmedtriangleleft}

\graphicspath{{figures/}{figures/}}

\newcommand{\wma}{\raisebox{-2pt}{\includegraphics[scale=0.9]{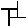}}}
\newcommand{\wmb}{\raisebox{-2pt}{\includegraphics[scale=0.9]{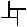}}}
\newcommand{\zwall}{\raisebox{-2pt}{\includegraphics[scale=0.9]{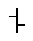}}}
\newcommand{\swall}{\raisebox{-2pt}{\includegraphics[scale=0.9]{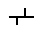}}}
\newcommand{\izwall}{\scalebox{1}[-1]{\raisebox{-10pt}{\includegraphics[scale=0.9]{zwall}}}}
\newcommand{\iswall}{\scalebox{1}[-1]{\raisebox{-6pt}{\includegraphics[scale=0.9]{swall}}}}
\newcommand{\cross}{\raisebox{-2pt}{\includegraphics[scale=0.65]{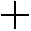}}}

\tikzstyle{wall}=[line width=0.6pt]

\newcommand{\td}{
\begin{tikzpicture}[baseline=0.1mm,scale=0.10]
\protect\draw[wall] (0,1.7)--(2,1.7);
\protect\draw[wall] (1,0)--(1,1.7);
\end{tikzpicture}
}

\newcommand{\tu}{
\begin{tikzpicture}[baseline=-0.1mm,scale=0.10]
\protect\draw[wall] (0,0)--(2,0);
\protect\draw[wall] (1,0)--(1,1.7);
\end{tikzpicture}
}

\newcommand{\tl}{
\begin{tikzpicture}[baseline=0.1mm,scale=0.10]
\protect\draw[wall] (1.7,0)--(1.7,2);
\protect\draw[wall] (0,1)--(1.7,1);
\end{tikzpicture}
}

\newcommand{\tr}{
\begin{tikzpicture}[baseline=0.1mm,scale=0.10]
\protect\draw[wall] (0,0)--(0,2);
\protect\draw[wall] (0,1)--(1.7,1);
\end{tikzpicture}
}

\newcommand{\pe}{
\begin{tikzpicture}[baseline=0.1mm,scale=0.10]
\protect\draw[wall] (0,2)--(2,2);
\protect\draw[wall] (2,2)--(2,0);
\end{tikzpicture}
}

\newcommand{\va}{
\begin{tikzpicture}[baseline=0.1mm,scale=0.10]
\protect\draw[wall] (0,2)--(0,0);
\protect\draw[wall] (0,0)--(2,0);
\end{tikzpicture}
}

\begin{document}
\title{Combinatorics of rec\-tan\-gu\-la\-ti\-ons:\\
  Old and new bijections}
\author{Andrei Asinowski\thanks{Alpen-Adria Universität Klagenfurt, Austria. 
E-mail: \href{andrei.asinowski@aau.at}{andrei.asinowski@aau.at}}, \ 
Jean Cardinal\thanks{Universit\'{e} libre de Bruxelles, Belgium. 
E-mail: \href{jean.cardinal@ulb.be}{jean.cardinal@ulb.be}}, \ 
Stefan Felsner\thanks{Technische Universität Berlin, Germany.
E-mail: \href{felsner@math.tu-berlin.de}{felsner@math.tu-berlin.de}}, \ 
\'{E}ric Fusy\thanks{LIGM, CNRS, Univ. Gustave Eiffel, Marne-la-Vallée, France.
E-mail: \href{eric.fusy@u-pem.fr}{eric.fusy@u-pem.fr}}}
\maketitle
\begin{abstract}
A rec\-tan\-gu\-la\-ti\-on is a decomposition of a rectangle into finitely many rectangles. 
Via na\-tu\-ral equivalence relations, rectangulations can be seen as combinatorial objects with a rich structure,
with links to lattice congruences, flip graphs, polytopes, lattice paths, Hopf algebras, etc. 
In this paper, we first revisit the structure of the respective equivalence classes: 
\emph{weak rec\-tan\-gu\-la\-ti\-ons} that preserve rectangle--segment adjacencies, 
and  \emph{strong rec\-tan\-gu\-la\-ti\-ons} that preserve rectangle--rectangle adjacencies.
We thoroughly investigate 
posets defined by adjacency in rec\-tan\-gu\-la\-ti\-ons 
of both kinds, and unify and simplify known	
bijections between rec\-tan\-gu\-la\-ti\-ons and permutation classes.
This yields a uniform treatment of mappings between permutations and rec\-tan\-gu\-la\-ti\-ons that unifies the results from earlier contributions,
and emphasizes parallelism and differences between the weak and the strong cases. 
Then, we consider the special case of \emph{guillotine rec\-tan\-gu\-la\-ti\-ons},
 and prove that they can be characterized 
--- under all known mappings between permutations and rec\-tan\-gu\-la\-ti\-ons  ---
by avoidance of two mesh patterns 
that correspond to ``windmills'' in rec\-tan\-gu\-la\-ti\-ons. 
This yields new permutation classes in bijection with weak guillotine rec\-tan\-gu\-la\-ti\-ons,
and the first known permutation class in bijection with strong guillotine rec\-tan\-gu\-la\-ti\-ons.
Finally, we address enumerative issues and prove asymptotic bounds for several
families of strong rec\-tan\-gu\-la\-ti\-ons.
\end{abstract}

\tableofcontents

\sloppy

\section{Introduction}

A~\textit{rec\-tan\-gu\-la\-ti\-on} is a decomposition of a rectangle into finitely many interior-disjoint rectangles. 
Rec\-tan\-gu\-la\-ti\-ons constitute a classical topic in mathematical tessellation theory. 
Among the earliest contributions on this topic one finds two papers by 
Abe from early 1930s~\cite{Abe30, Abe32},
and the papers by Brooks, Stone, Smith, and Tutte 
(collectively known as \emph{Blanche Descartes}) on ``squaring the square''~\cite{BlancheD40}, and on partitioning a square into equal-area rectangles~\cite{BlancheD71}. 
In the last decades, many results on rec\-tan\-gu\-la\-ti\-ons have been published in journals and conferences 
 on computational geometry as well as engineering and electronics, due to their being a basic model in 
\emph{floorplanning} --- an essential step in the design of very large scale integrated circuits (VLSI)~\cite{Sherwani, Sait, WCC09}.  
In floorplanning, functional blocks of a circuit, represented by rectangles, have to be packed on a small rectangular area.
The term \emph{floorplan} is therefore often used to designate a rec\-tan\-gu\-la\-ti\-on. 
Rec\-tan\-gu\-la\-ti\-ons have also applications 
in the analysis of geometric algorithms~\cite{Beaumont02, Calheiros03},
in visualization of scientific data 
(for instance treemaps~\cite{BSW02} and cartograms~\cite{KS07,BSV12}), 
in mathematical foundations of architectural design~\cite{Steadman83},
and also appear in visual art --- notably in the work of the Dutch art movement \emph{De Stijl}, 
see Figure~\ref{fig:doesburg}.

\begin{figure}[!h]
  \begin{center}
    \includegraphics[height=.25\textheight]{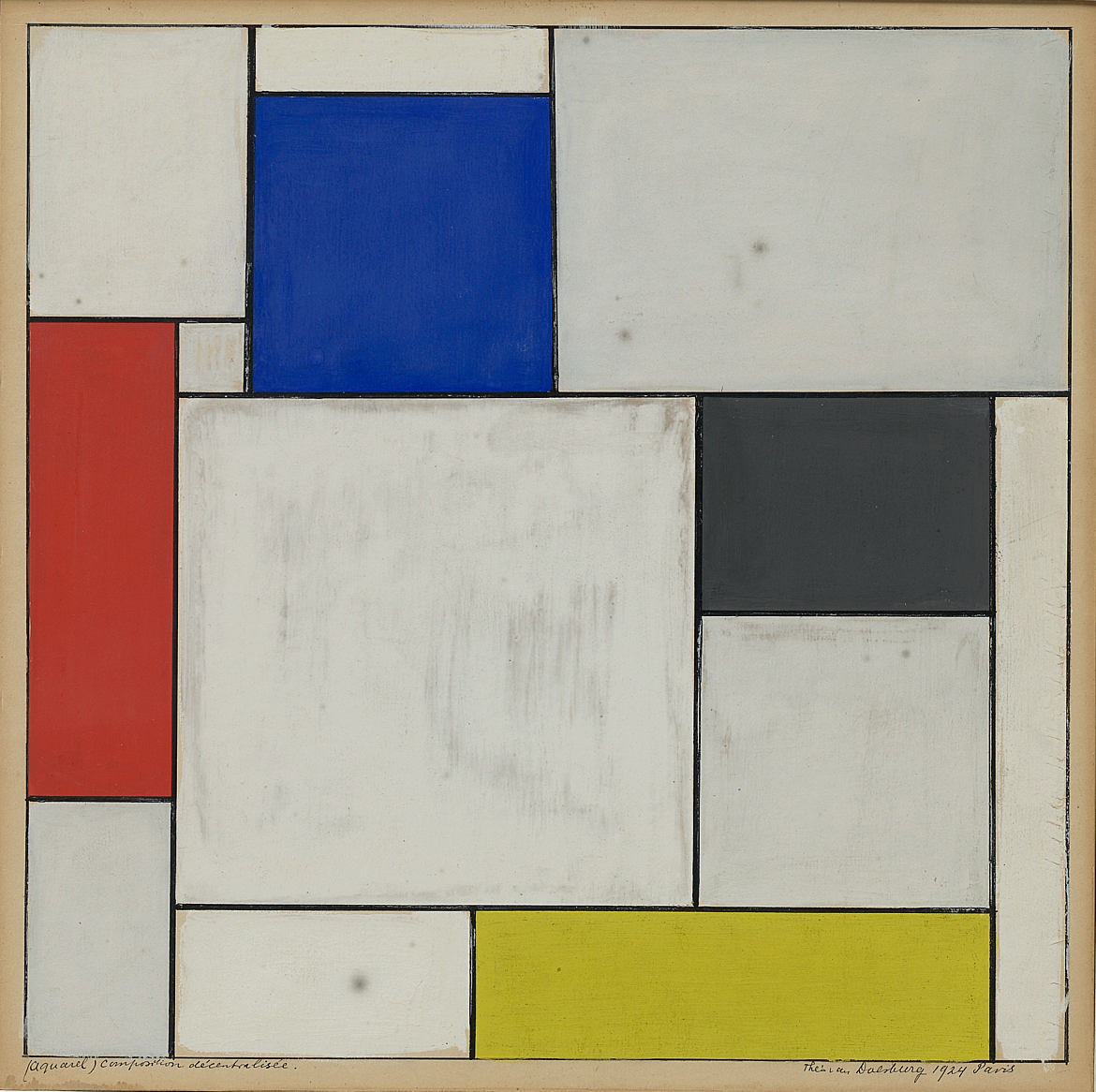}
  \end{center}
  \caption{Artwork \emph{Composition d\'{e}centralis\'{e}e} (1924) by Theo van Doesburg (Dutch, 1883--1931).
    Solomon R.~Guggenheim Museum, New York.}
\label{fig:doesburg}
\end{figure}

We investigate structural properties of rec\-tan\-gu\-la\-ti\-ons, which in
particular means that we are not interested in precise measures of rectangles
but rather in adjacencies between their elements --- rectangles and segments.
In order to treat rec\-tan\-gu\-la\-ti\-ons as combinatorial objects, one can
introduce an equivalence relation formalizing the idea of two
rec\-tan\-gu\-la\-ti\-ons being ``structurally identical''.  There are two
natural equivalence relations of this kind.  The \emph{weak equivalence}
preserves incidence and sidedness between segments and rectangles. The
\emph{strong equivalence} additionally preserves the adjacencies between
rectangles. Precise definitions are given in Section~\ref{sec:equivalences}. 

Many structural investigations of rec\-tan\-gu\-la\-ti\-ons are focused on their bijective representation 
by classes of permutations determined by pattern avoidance~\cite{ABP06,FT07,FFNO11,LR12,R12,ABBMP13,CSS18,M19a,M19b,MM23}. 
For example, \emph{Baxter permutations}, 
defined by avoidance of a certain pair of \emph{vincular patterns} of size $4$,
have been shown to be in a~(size-preserving) bijection with 
\emph{mosaic} or \emph{diagonal} rec\-tan\-gu\-la\-ti\-ons~\cite{ABP06} --- that is,
rec\-tan\-gu\-la\-ti\-ons considered up to the weak equivalence.
This bijection can be restricted to a bijection between \emph{separable permutations}, 
defined by avoidance of the patterns $2413$ and $3142$~\cite{BBL98},
and the so-called \textit{guillotine} rec\-tan\-gu\-la\-ti\-ons (also known as \textit{sliceable} rec\-tan\-gu\-la\-ti\-ons).
These results can be fruitfully compared to a basic result in Catalan combinatorics, namely the bijection between \emph{triangulations} of a~convex polygon and $231$-avoiding permutations~\cite{S15}.

The combinatorics of such families has been analyzed in the framework of \emph{congruences} of the \emph{weak Bruhat order}~\cite{R04}.
The weak Bruhat order is the ordering of the permutations of $[n]$ by inclusion of their set of inversions.
A congruence is an equivalence relation on the elements of a lattice that is consistent with the meet and join operations.
Catalan families and mosaic rec\-tan\-gu\-la\-ti\-ons are both examples of families that define congruences of the weak Bruhat order.
As a consequence, the corresponding families (of pattern-avoiding permutations or tessellations) are ordered by a lattice, defined as the \emph{quotient} of the weak Bruhat order by the congruence.
It was shown by Pilaud and Santos~\cite{PS19} that the cover graphs of these quotients are all skeletons of polytopes, that they called \emph{quotientopes}.
In the case of triangulations and other Catalan objects, the quotient lattice is the well-studied \emph{Tamari lattice}~\cite{ATLRS} and the quotientope is the ubiquitous \emph{associahedron}~\cite{S63,L04,CSZ15}.
The quotientope of mosaic rec\-tan\-gu\-la\-ti\-ons, on the other hand, is known to be a Minkowski sum of two associahedra~\cite{LR12}.

In 2012, Reading~\cite{R12} studied rec\-tan\-gu\-la\-ti\-ons considered up to the strong equivalence.
He showed that, similarly to the weak case, they are bijective to equivalence classes of permutations that form congruence classes and thus 
induce a~quotient of the weak Bruhat order, and also to so-called \emph{2-clumped permutations}.

Subsequently, Meehan~\cite{M19b} analyzed the cover relation in this quotient, yielding a nice \emph{flip graph} on generic rec\-tan\-gu\-la\-ti\-ons.
From Pilaud and Santos~\cite{PS19}, this flip graph is the skeleton of the quotientope of generic rec\-tan\-gu\-la\-ti\-ons.

The main goals that motivated the study presented in this paper were as follows:

\begin{quote}
\begin{enumerate}
\item To develop a uniform treatment of mappings between permutations and rec\-tan\-gu\-la\-ti\-ons that would unify the results from earlier contributions
and emphasize parallelism and differences between the weak and the strong cases.
\item To simplify the description of the bijection between generic rec\-tan\-gu\-la\-ti\-ons and 2-clumped permutations, and give a concise characterization of the corresponding congruence classes of the weak Bruhat order.
\item To find a permutation class in bijection with guillotine generic rec\-tan\-gu\-la\-ti\-ons.
  
\item To address the enumeration of guillotine generic rec\-tan\-gu\-la\-ti\-ons.
Under the weak equivalence, the generating function of all rec\-tan\-gu\-la\-ti\-ons is (non-algebraic) D-finite, while the generating function of guillotine rec\-tan\-gu\-la\-ti\-ons is algebraic. 
Under the strong equivalence, the generating function for all rec\-tan\-gu\-la\-ti\-ons is not D-finite, while the status of the generating function for guillotine rec\-tan\-gu\-la\-ti\-ons is yet to be determined.
  \end{enumerate}
\end{quote}

\paragraph{Our results.}

The first part of our contribution is on the strong equivalence relation and strong rec\-tan\-gu\-la\-ti\-ons.
We define the \emph{strong order} on rectangles of a strong rec\-tan\-gu\-la\-ti\-on, and prove that the linear extensions of this strong 
order form equivalence classes of permutations that are bijective with strong rec\-tan\-gu\-la\-ti\-ons, and are intervals in the weak Bruhat order. 
We naturally derive bijections between strong rec\-tan\-gu\-la\-ti\-ons and, 
respectively, 2-clumped and co-2-clumped permutations --- the minimum and the maximum of the equivalence classes.

The material in this part streamlines and simplifies a number of previous works.
On one hand, Reading~\cite{R12} (see also Meehan~\cite{M19b} and Merino and M\"utze~\cite{MM23}) considers the combinatorics of strong rec\-tan\-gu\-la\-ti\-ons and defines the same permutations-to-rec\-tan\-gu\-la\-ti\-ons mapping as ours.
We present simple incremental algorithms for the forward and backward directions of this mapping that allow for simpler and more direct proofs.
In particular, our forward algorithm yields a simple proof for the description of the flip graph studied by Meehan~\cite{M19b}.
Our definition of the strong poset for the strong equivalence relation between rec\-tan\-gu\-la\-ti\-ons is an analogue of the adjacency poset for weak equivalence defined by Meehan~\cite{M19a}.
Interestingly, it appears that the mapping defined by Reading has already been studied in the form of the ``FT-squeeze'' algorithm, devised by Fujimaki and Takahashi~\cite{FT07,TF08} for VLSI design.
The strong poset that we introduce is equivalent to the ``seagull order'' proposed by Fujimaki and Takahashi as a physical intuition for the FT-squeeze~\cite{FT07}.
It also appears in the guise of the elimination process devised by Takahashi, Fujimaki, and Inoue~\cite{ITF09,TFI09} for giving efficient counting and coding methods on strong rec\-tan\-gu\-la\-ti\-ons. 
Finally, our forward  algorithm is strongly related to a mapping defined by Fran\c con and Viennot~\cite{FV79} for the analysis of permutations parameterized by their number of peaks, valleys, double ascents, and double descents, and can be analyzed within the framework of quadrant walks. 
The connection between these numerous lines of work seems to have \mbox{gone unnoticed so far.}

The second part of our paper is dedicated to guillotine rec\-tan\-gu\-la\-ti\-ons.
We introduce two mesh patterns on permutations 
that can be used for ``encoding'' \textit{windmills} --- certain configurations of segments,
whose occurrence in a rec\-tan\-gu\-la\-ti\-on is equivalent to being non-guillotine.
Combining these mesh patterns with the forbidden patterns of Baxter permutations, we 
obtain new bijections for weak guillotine rec\-tan\-gu\-la\-ti\-ons. 
More interestingly, combining these two mesh patterns with the forbidden patterns for 2-clumped 
permutations, we obtain a bijection between \emph{strong guillotine rec\-tan\-gu\-la\-ti\-ons} (that is, strong equivalence classes of guillotine rec\-tan\-gu\-la\-ti\-ons)  
and a permutation class. This is the first known representation of this family of rec\-tan\-gu\-la\-ti\-ons by a permutation class.

\medskip

The plan of the paper is as follows. 
In Section~\ref{sec:bg}, we give precise definitions of the objects that we study and review basic results.
In particular, the equivalence classes of rec\-tan\-gu\-la\-ti\-ons of the weak and strong equivalence
will be called, respectively, weak and strong rec\-tan\-gu\-la\-ti\-ons.
In Section~\ref{sec:weak}, we review earlier results on weak rec\-tan\-gu\-la\-ti\-ons:
a mapping $\gamma_w$ from permutations to weak rec\-tan\-gu\-la\-ti\-ons, 
\textit{weak posets} as fibers of this mapping, 
the induced bijections between weak rec\-tan\-gu\-la\-ti\-ons and Baxter, twisted Baxter, and co-twisted Baxter permutations, 
as well as the structure of the corresponding weak Bruhat order congruence.
Then, Section~\ref{sec:strong} is devoted to an extensive study of the structure of strong rec\-tan\-gu\-la\-ti\-ons, while emphasizing its parallelism to the weak case: 
a mapping $\gamma_s$ from permutations to strong rec\-tan\-gu\-la\-ti\-ons, 
\textit{strong posets} as fibers of this mapping, 
and the induced bijections between strong rec\-tan\-gu\-la\-ti\-ons and 2-clumped (resp. co-2-clumped) permutations. Moreover, we identify the flip graph on rec\-tan\-gu\-la\-ti\-ons, and we show how to encode rec\-tan\-gu\-la\-ti\-ons (and subfamilies) by quadrant walks, allowing efficient counting. 
Finally, in Section~\ref{sec:guillotine}, we present two mesh patterns $p_1, p_2$
that ``encode'' windmills,
propose novel permutation classes in bijection with weak and strong guillotine rec\-tan\-gu\-la\-ti\-ons, show that the $n$ first terms of the enumerating sequence of strong guillotine rec\-tan\-gu\-la\-ti\-ons can be computed in polynomial time in $n$,  
and provide lower and upper bounds on the number of strong guillotine rec\-tan\-gu\-la\-ti\-ons of size $n$. 
The following table shows a summary of bijections 
between the considered classes of rec\-tan\-gu\-la\-ti\-ons and permutation classes,
along with the references to the sections where these are discussed.

\medskip

\renewcommand{\arraystretch}{1.2}
\hspace{-13pt}
  \begin{tabular}{|c|c|c|}
  \hline
              & Weak equivalence               & Strong equivalence \\
    \hline
    Arbitrary & \textbf{Weak rec\-tan\-gu\-la\-ti\-ons}  & \textbf{Strong rec\-tan\-gu\-la\-ti\-ons} \\
  & Baxter permutations & \\  
    & twisted Baxter permutations            & 2-clumped permutations \\
    & co-twisted Baxter permutations         & co-2-clumped permutations \\
    & \hspace*{\fill} $\longrightarrow$ Section~\ref{sec:baxter} & \hspace*{\fill} $\longrightarrow$ Section~\ref{sec:2clumped} \\
    \hline
    Guillotine & \textbf{Weak guillotine rec\-tan\-gu\-la\-ti\-ons} & \textbf{Strong guillotine rec\-tan\-gu\-la\-ti\-ons} \\
    & separable permutations            &  \\
    & $\{p_1,p_2\}$-avoiding twisted Baxter perm. & $\{p_1,p_2\}$-avoiding 2-clumped permutations \\
    & $\{p_1,p_2\}$-avoiding co-twisted Baxter perm. & $\{p_1,p_2\}$-avoiding co-2-clumped permutations \\
    & \hspace*{\fill} $\longrightarrow$ Sections~\ref{sec:guil} and~\ref{sec:windmill_bij} & \hspace*{\fill} $\longrightarrow$ Section~\ref{sec:windmill_bij} \\
    \hline
  \end{tabular}
\renewcommand{\arraystretch}{1}

\section{Definitions and basics} 
\label{sec:bg}

In this section we present basic notions and definitions used in the paper,
as well as some ``classical'' results.
In this exposition, we mainly follow the works by
Ackerman, Barequet, and Pinter~\cite{ABP06},
Law and Reading~\cite{LR12}, 
Reading~\cite{R12}, 
Cardinal, Sacrist{\'{a}}n, and Silveira~\cite{CSS18},
and 
Merino and M\"utze~\cite{MM23},
with some minor modifications for the sake of uniformity.

\subsection{Rectangulations and their elements}

Let $\rr$ be an axes-aligned rectangle in the plane.
A~\emph{rec\-tan\-gu\-la\-ti\-on} of $\rr$ is a~decomposition (or tiling) 
of~$\rr$ into finitely many interior-disjoint \textit{rectangles}.
The \emph{size} of a rec\-tan\-gu\-la\-ti\-on is the number of rectangles in the decomposition. 
The rec\-tan\-gu\-la\-ti\-on in Figure~\ref{fig:doesburg} is of size~$13$. 
Rec\-tan\-gu\-la\-ti\-ons will generically be denoted by $\rc$, and their size by $n$.

A~\emph{segment} of a rec\-tan\-gu\-la\-ti\-on is 
a maximal straight line segment that consists of one or several sides of 
some 
rectangles
of $\rc$, 
and is not included in one of the sides of~$\rr$. 
A rec\-tan\-gu\-la\-ti\-on is \textit{generic} if there is no point at which four 
rectangles meet.
\textbf{From now on, we assume that all rec\-tan\-gu\-la\-ti\-ons in this paper are generic. 
Thus, intersection of two segments of a rec\-tan\-gu\-la\-ti\-on can form a joint of one of the following shapes:
 $\tu$\hspace{1pt}, $\tr$, $\td$,  $\tl$\hspace{1pt},
but never $\cross$.}
It is easily shown that a~rec\-tan\-gu\-la\-ti\-on of size $n$ has precisely $n-1$ segments.  
The \emph{neighbors} of a segment~$s$ are the segments (perpendicular to~$s$)
 with an endpoint that lies on~$s$.\footnote{In some papers rec\-tan\-gu\-la\-ti\-ons are referred to as \textit{floorplans}, their rectangles as \emph{rooms}, and segments as \emph{walls}.}

We also label the corners of $\rr$, or of any rectangle of $\rc$, by the ordinal directions: NE for top-right, SE for bottom-right, SW for bottom-left, NW for top-left.

\subsection{Weak equivalence and strong equivalence}\label{sec:equivalences}

In order to deal with rec\-tan\-gu\-la\-ti\-ons as combinatorial objects, one has to consider some equivalence relation,
formalizing the idea of rec\-tan\-gu\-la\-ti\-ons ``having the same structure''.
There are two natural ways to do this: the \emph{weak equivalence} that preserves segment--rectangle incidences and sidedness, 
and the \emph{strong equivalence} that additionally preserves rectangle--rectangle adjacencies. 

To give precise definitions, we introduce \emph{left--right} and \emph{above--below} order relations between rectangles of a rec\-tan\-gu\-la\-ti\-on:
\begin{itemize}
\item Rectangle $r$ is \emph{on the left} of rectangle $r'$ 
(equivalently, $r'$ is \emph{on the right} of $r$)
if there is a sequence of rectangles, $r=r_1, r_2, \ldots, r_k=r'$
such that the right side of $r_i$ and the left side of $r_{i+1}$ lie in the same segment for $i=1, 2, \ldots, k-1$.
\item Rectangle $r$ is \emph{below} rectangle $r'$ 
(equivalently, $r'$ is \emph{above} $r$)
if there is a sequence of rectangles, $r=r_1, r_2, \ldots, r_k=r'$
such that the upper side of $r_i$ and the bottom side of $r_{i+1}$ lie in the same segment for $i=1, 2, \ldots, k-1$.
\end{itemize}

Given a rectangle $r$ of $\rc$, one can specify the regions that contain the rectangles 
which are above, below, on the left, or on the right of $r$, as follows.
 The \textit{NE alternating path associated with $r$} is the path
that starts at the NE (top-right) corner of $r$, 
goes first upwards if the NE corner of $r$ has the shape $\tl$\hspace{1pt} 
or rightwards if it has the shape $\td$\hspace{1pt},
and then alternatingly traverses
vertical segments upwards to their upper endpoint then turning rightwards,
and horizontal segments rightwards to their right endpoint and then turning upwards,
until it reaches the NE~corner of $\rc$. 
The NE alternating path associated with $r$ is shown by red in Figure~\ref{fig:abrl}.
One similarly defines SE, SW, and NW alternating paths. 
The four alternating paths split $\rr \setminus \{r\}$ into four regions (some of them can be empty). 
This leads to the following observation.

\begin{observation}
\label{obs:lrab}
Let $r$ be a rectangle in a rec\-tan\-gu\-la\-ti\-on $\rc$.
\begin{enumerate}
\item The rectangles of $\rc$ which lie above, below, on the left, or on the right of $r$
are contained in respective regions of $\rr \setminus \{r\}$ delimited by the alternating paths
(refer to Figure~\ref{fig:abrl}).
\item Every pair of distinct rectangles in a rec\-tan\-gu\-la\-ti\-on is comparable with precisely one of the order relations: either one of them is on the left of the other, or one of them is above the other.
\end{enumerate}
\end{observation}

\begin{figure}[!h]
\begin{center}
\includegraphics[scale=0.9]{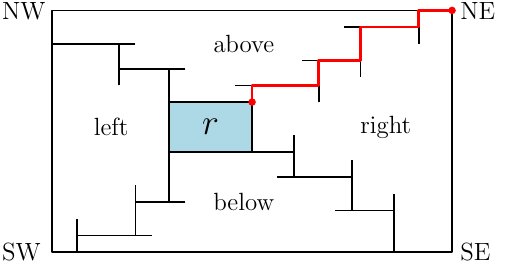} 
\end{center}
\caption{Illustration to Observation~\ref{obs:lrab}: four regions delimited by alternating paths.}
\label{fig:abrl}
\end{figure}

\textbf{Now the two kinds of equivalence of rec\-tan\-gu\-la\-ti\-ons are defined as follows.}
\begin{itemize}
\item Two rec\-tan\-gu\-la\-ti\-ons $\rc_1$ and $\rc_2$ are \emph{weakly equivalent} 
if there is a~(unique) bijection between their rectangles that preserves the left-right and the 
above-below orders. 
\item Two rec\-tan\-gu\-la\-ti\-ons $\rc_1$ and $\rc_2$ are \emph{strongly equivalent} if they are weakly equivalent, 
and the bijection that realizes the weak equivalence also preserves adjacencies between rectangles,
 that is, two rectangles in $\rc_1$ are adjacent if and only if the corresponding rectangles in $\rc_2$ are adjacent.
\end{itemize}

These equivalence relations can be also described in terms of local modifications of rec\-tan\-gu\-la\-ti\-ons.
Two rec\-tan\-gu\-la\-ti\-ons are weakly equivalent if they can be obtained from each other by a sequence
of moves, where each move is shifting some segment (or one of the sides of $\rr$) horizontally 
or vertically,
and accordingly extending or shortening its neighbors,
so that the adjacencies between segments do not change.
To obtain strong equivalence, we restrict the moves so that 
the adjacencies between rectangles are also preserved.

A~\emph{weak rec\-tan\-gu\-la\-ti\-on} is an equivalence class of rec\-tan\-gu\-la\-ti\-ons with respect to the weak equivalence,
and a~\emph{strong rec\-tan\-gu\-la\-ti\-on} is an equivalence class of rec\-tan\-gu\-la\-ti\-ons with respect to the strong equivalence.\footnote{In some earlier papers, for example~\cite{R12, MM23},
weak rec\-tan\-gu\-la\-ti\-ons are called \emph{mosaic rec\-tan\-gu\-la\-ti\-ons},
and strong rec\-tan\-gu\-la\-ti\-ons are referred to just as \emph{generic rec\-tan\-gu\-la\-ti\-ons}.
In~\cite{FT07, TF08},
weak rec\-tan\-gu\-la\-ti\-ons are called \textit{room-to-wall floorplans},
and strong rec\-tan\-gu\-la\-ti\-ons are called \textit{room-to-room floorplans}.}
In Figure~\ref{fig:ex0}, rec\-tan\-gu\-la\-ti\-ons $\mathcal{R}_1$, $\mathcal{R}_2$, and $\mathcal{R}_3$ are weakly equivalent, but only $\mathcal{R}_1$ and $\mathcal{R}_2$ are strongly equivalent.
In other words, here we have two weak rec\-tan\-gu\-la\-ti\-ons (one of them is given by three representatives: $\mathcal{R}_1$, $\mathcal{R}_2$, and $\mathcal{R}_3$),
and three strong rec\-tan\-gu\-la\-ti\-ons (one of them is given by two representatives: $\mathcal{R}_1$ and $\mathcal{R}_2$).
\textbf{Rec\-tan\-gu\-la\-ti\-on~$\mathcal{R}_1$ will be used throughout the paper as a~``running example''
for demonstrating various results.}
\begin{figure}[!h]
\begin{center}
\includegraphics[scale=0.9]{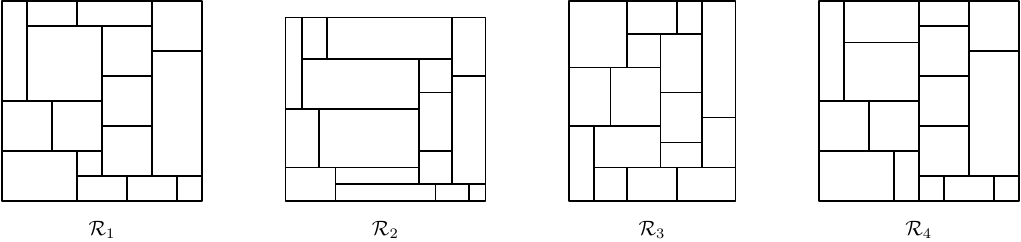} 
\end{center}
\caption{$\mathcal{R}_1$, $\mathcal{R}_2$, and $\mathcal{R}_3$ are weakly equivalent. 
$\mathcal{R}_1$ and $\mathcal{R}_2$ are strongly equivalent. $\mathcal{R}_4$ is guillotine.}
\label{fig:ex0}
\end{figure}

The strong equivalence refines the weak one, 
and thus every weak rec\-tan\-gu\-la\-ti\-on yields one or several strong rec\-tan\-gu\-la\-ti\-ons
by all possible \emph{shuffles} of the neighbors of its segments.
If a segment $s$ has $a$ neighbors on one side and $b$ neighbors on the other side,
then these can be shuffled in $\binom{a+b}{a}$ ways.

\subsection{NW--SE and SW--NE labelings}
\label{sec:labels}

Let $\rc$ be a rec\-tan\-gu\-la\-ti\-on. 
Since, by Observation~\ref{obs:lrab}, the transitive relations ``left'' and ``above'' yield a partition of the edges of the complete graph, their union ``left of, or above'' is a total order of the rectangles.
Hence, we can label the rectangles by the numbers from $1$ to $n$ according to this order.
The rectangle with label $j$ ($1 \leq j \leq n$) will be denoted by $r_j$.
Since $r_1$ contains the NW (top-left) corner of $\rr$ and $r_n$ contains the SE (bottom-right) corner of $\rr$,
we call this labeling \emph{the NW--SE labeling}.
If $j$ is fixed, then 
$r_i$ with $i<j$ are precisely the rectangles above or to the left of $r_j$,
and $r_i$ with $i>j$ are precisely the rectangles below or to the right of $r_j$:
see Figure~\ref{fig:ablr_diag} for a schematic depiction.

\begin{figure}[!h]
\begin{center}
\includegraphics[scale=0.9, page=2]{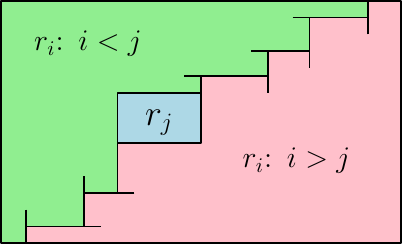} 
\end{center}
\caption{The rectangles $r_i$ with $i>j$ (respectively $i<j$) are located 
to the right or below $r_j$ (respectively to the left or above $r_j$).} 
\label{fig:ablr_diag}
\end{figure}

The NW--SE labeling can also be obtained by the following algorithm.

\begin{mdframed}
\textbf{Algorithm: NW--SE labeling}

\noindent Input: Rec\-tan\-gu\-la\-ti\-on $\rc$.

\noindent Output: The NW--SE labeling of the rectangles of $\rc$.
\begin{enumerate}
\item Label $r_1$ the rectangle that contains the top-left corner of $\rr$. 

\item For $j=2$ to $n$:

Consider the joint of segments at the bottom-right corner of $r_{j-1}$.
\begin{itemize}
\item If its shape is $\tl$, then label $r_{j}$ the leftmost rectangle whose upper side belongs to the same horizontal segment as the bottom side of $r_{j-1}$,
\item If its shape is $\tu$, then label $r_j$ the topmost rectangle whose left side belongs to the same vertical segment as the right side of $r_{j-1}$.
\end{itemize}
\end{enumerate}
\end{mdframed}

\medskip

Similarly, one can define the \textit{SW--NE labeling} in which $i<j$ 
if and only if~$r_i$ is to the left or below~$r_j$. 
It can be obtained by an obvious modification of the algorithm given above. 
Figure~\ref{fig:NWSEorder} shows the rec\-tan\-gu\-la\-ti\-on~$\mathcal{R}_1$ with 
the NW--SE (left) and the SW--NE (right) labelings of its rectangles.
(In this and other examples below, we label the rectangles just $j$ instead of $r_j$).

\begin{figure}[!h]
\begin{center}
\includegraphics[scale=0.82]{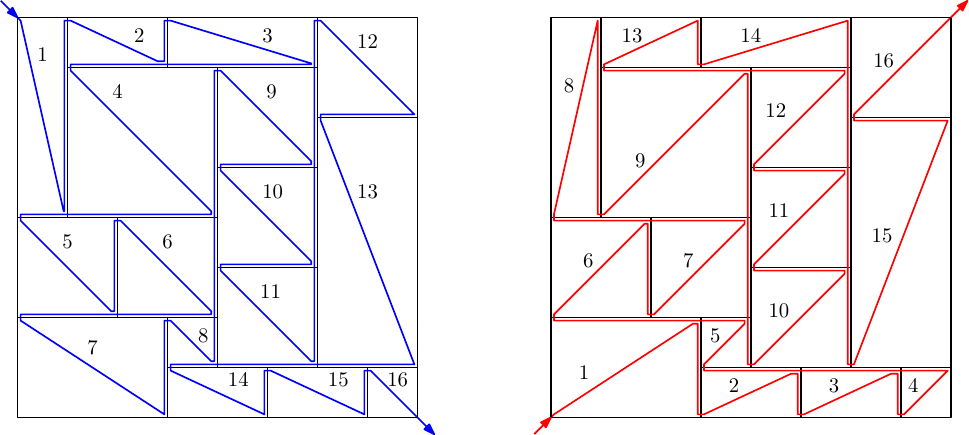} 
\end{center}
\caption{The NW--SE (left) and the SW--NE (right) orderings of the rectangles of $\rc_1$.}
\label{fig:NWSEorder}
\end{figure}

\subsection{Diagonal rectangulations}\label{sec:diag}
\label{sec:diagonal}

A~\emph{diagonal rec\-tan\-gu\-la\-ti\-on} of size $n$ is a rec\-tan\-gu\-la\-ti\-on $\dc$ of size $n$, drawn on an $n\times n$ grid square $S$, 
such that all the segments are drawn along grid lines, and every rectangle of $\dc$ intersects 
the NW--SE diagonal of~$S$.
Diagonal rec\-tan\-gu\-la\-ti\-ons have the following properties (see for example~\cite[Section~5]{LR12}).
\begin{proposition}\label{prop:diag}
\begin{enumerate}
\item[(a)] Every rec\-tan\-gu\-la\-ti\-on $\rc$ is weakly equivalent to a~\emph{unique} diagonal rec\-tan\-gu\-la\-ti\-on
$\dc$, which will be referred to as the \emph{diagonal representative} of $\rc$.
\item[(b)] In a diagonal rec\-tan\-gu\-la\-ti\-on we have the following. 
For every horizontal segment $s$,
all the above neighbors of $s$ occur from the left of all its below neighbors;
and 
for every vertical segment $t$,
all the left neighbors of $t$ occur above all its right neighbors.
\item[(c)] The order in which the NW--SE diagonal of $S$ meets the rectangles of a diagonal rec\-tan\-gu\-la\-ti\-on,
is the NW--SE order.
\end{enumerate}
\end{proposition}
Due to property (a), diagonal rec\-tan\-gu\-la\-ti\-ons are frequently considered as canonical representatives
of weak rec\-tan\-gu\-la\-ti\-ons (or sometimes even identified with them). 
Property (b) specifies the unique shuffling of the segments of $\rc$
that its diagonal representative can have.
In other words, it specifies the unique strong rec\-tan\-gu\-la\-ti\-on 
which is weakly equivalent to the given $\rc$
and strongly equivalent to the {diagonal representative} of $\rc$.
Due to property (c), the NW--SE labeling of a diagonal rec\-tan\-gu\-la\-ti\-on is also called
\emph{the diagonal labeling}.   

\medskip

One similarly defines \emph{anti-diagonal rec\-tan\-gu\-la\-ti\-ons} all of whose rectangles meet the
SW--NE diagonal (in the order determined by the SW--NE labeling).
In Figure~\ref{fig:diag} we show
the diagonal rec\-tan\-gu\-la\-ti\-on $\dc_1$ weakly equivalent to $\mathcal{R}_1$ 
along with its NW--SE labeling,
and
the anti-diagonal rec\-tan\-gu\-la\-ti\-on $\dc'_1$ weakly equivalent to $\mathcal{R}_1$ 
along with its SW--NE labeling.
\begin{figure}[!h]
\begin{center}
\includegraphics[scale=0.72]{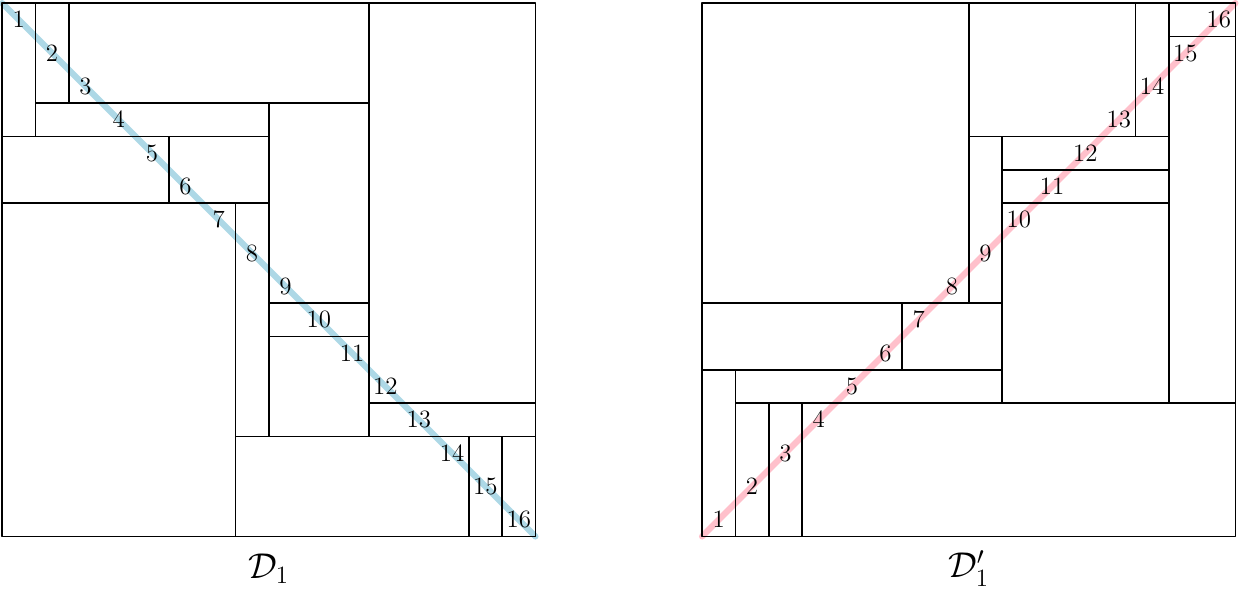} 
\end{center}
\caption{Left: The diagonal rec\-tan\-gu\-la\-ti\-on $\dc_1$ weakly equivalent to~$\mathcal{R}_1$. 
Right: The anti-diagonal rec\-tan\-gu\-la\-ti\-on $\dc'_1$ weakly equivalent to~$\mathcal{R}_1$.}
\label{fig:diag}
\end{figure}

\subsection{Guillotine rectangulations}
\label{sec:guil}

A~\emph{cut} of a rec\-tan\-gu\-la\-ti\-on $\rc$ is
a vertical segment that extends from the top side to the bottom side of~$\rr$,
or a horizontal segment that extends from the left side to the right side of~$\rr$.
If $\rc$ has several cuts, then they all have the same orientation.

A rec\-tan\-gu\-la\-ti\-on is \emph{guillotine} if it is either of size $1$, 
or it has a cut $s$ such that both sub-rec\-tan\-gu\-la\-ti\-ons separated by $s$ are guillotine.
In Figure~\ref{fig:ex0}, only rec\-tan\-gu\-la\-ti\-on $\mathcal{R}_4$ is guillotine.

A~\emph{windmill} in a rec\-tan\-gu\-la\-ti\-on is a quadruple of segments forming one of the following two shapes:
\wma\ or \wmb. (Windmills are also referred to as \emph{pin-wheels}~\cite{ABP06}.)
Note that segments that form a windmill can have arbitrarily positioned further neighbors,
also in the \textit{interior} --- the rectangular region that they bound.
Guillotine rec\-tan\-gu\-la\-ti\-ons have the following characterization (proven for instance in~\cite{ABP06}).

\begin{proposition}
  \label{prop:guillotinechar}
  A rec\-tan\-gu\-la\-ti\-on is guillotine if and only if it avoids the windmills \wma\ and \wmb.
\end{proposition}

The enumeration of weak guillotine rec\-tan\-gu\-la\-ti\-ons is nearly elementary.
Denote their generating function with respect to the size by $G(x)$.
We say that a guillotine rec\-tan\-gu\-la\-ti\-on of size $>1$ is \emph{horizontal}
or \emph{vertical} in accordance with the orientation of its cut(s). 
The rec\-tan\-gu\-la\-ti\-on of size $1$ is considered neither horizontal nor vertical.
Then the generating function of horizontal guillotine rec\-tan\-gu\-la\-ti\-ons,
and that of vertical guillotine rec\-tan\-gu\-la\-ti\-ons, is $H(x)=V(x)=(G(x)-x)/2$.
Every vertical guillotine rec\-tan\-gu\-la\-ti\-on
is split by its leftmost cut such that the left part is either vertical guillotine or of size $1$,
and the right part is arbitrary guillotine.
This decomposition is unique, and, 
hence, the generating functions introduced above satisfy the equation
\begin{equation}\label{eq:HV}
V(x) = \big(x + H(x)\big)G(x),
\end{equation}
which yields 
\[
H(x)=V(x) = \frac{1-3x-\sqrt{1-6x+x^2}}{4}=x^2+3x^3+11x^4+45x^5+\ldots, 
\]
\begin{equation}\label{eq:schr}
G(x) = x+H(x)+V(x)  = \frac{1-x-\sqrt{1-6x+x^2}}{2}=x+2x^2+6x^3+22x^4+90x^5+\ldots.
\end{equation}
Therefore, we have the following result 
(proven, for example, in~\cite{Yao03} via a bijection to \textit{v-h-trees},
and in~\cite{Shen03} via a bijection to \textit{skewed slicing trees}).

\begin{proposition}\label{prop:guillotineenum}
The number of weak guillotine rec\-tan\-gu\-la\-ti\-ons of size $n$ is the $(n-1)$th \emph{Schröder number} (\href{https://oeis.org/A006318}{OEIS A006318}).
\end{proposition}

Multidimensional generalizations of guillotine rec\-tan\-gu\-la\-ti\-ons were considered in~\cite{ABPR06} and~\cite{AM10}.

\subsection{Permutation patterns}

In this section we briefly review the basic definitions and notation from the field of permutation patterns (see also the summary by David Bevan~\cite{bevan15}).  
To specify a permutation, we use the linear notation:
that is, $\pi=a_1 a_2 \ldots a_n$ is the permutation of $[n]$ that maps $i$ to $a_i$
for $i=1, 2, \ldots, n$.
It is convenient to describe such a permutation by a~\emph{plot} --- the point set $\{(i, a_i)\colon \, i \in [n]\}$.

\smallskip\noindent\textbf{Classical patterns.}
Let $\pi=\pi_1 \pi_2 \ldots \pi_n$ be a permutation of $[n]$, and let $\tau$ be a~``pattern'' --- a fixed permutation of $[k]$ .
An \emph{occurrence} of $\tau$ in $\pi$ is a~(not necessarily consecutive) subsequence $\pi_{s_1} \pi_{s_2} \ldots \pi_{s_k}$ of $\pi$, which is order-isomorphic to $\tau$.
If $\pi$ has an occurrence of $\tau$, we say that $\pi$ \emph{contains} $\tau$.
Otherwise, we say that $\pi$ \emph{avoids}~$\tau$.
For example, the permutation $\pi=32514$ contains the pattern $132$
(the subsequence $254$ of $\pi$ is an occurrence of $132$); and the permutation $\rho=43512$ avoids $132$.

\smallskip\noindent\textbf{Vincular patterns.}
A~\emph{vincular pattern} is a pair  $v=(\tau, \lambda)$, where $\tau$ is a permutation of $[k]$,
and $\lambda$ is a set of one or several pairwise disjoint strings in $\tau$, 
indicated by underlining (for example $ 36\underline{18}57\underline{942}$).
An occurrence of $v$ in $\pi$ is an occurrence of $\tau$ such that the letters that correspond to the same underlined string occur consecutively in $\pi$.
For example, the permutation $\pi=24513$ contains the pattern $2\underline{41}3$
(the subsequence $2513$ of $\pi$ is an occurrence of $2\underline{41}3$); 
and the permutation $\rho=25314$ avoids $2\underline{41}3$
(but contains the classical pattern $2413$).

\smallskip\noindent\textbf{Mesh patterns.}
A~\emph{mesh pattern} is a pair $m=(\tau, \mu)$, where $\tau$ is permutation of $[k]$, 
and $\mu$ is a subset of $\{0, 1, \ldots, k\}\times\{0, 1, \ldots, k\}$.
An occurrence of $m$ in $\pi \in [n]$ is an occurrence $a_{s_1} a_{s_2} \ldots a_{s_k}$ of $\tau$ that satisfies the condition:
for every $(i,j)\in \mu$, there is no $\ell$ such that $  s_i < \ell < s_{i+1}$ (with the convention $s_0=-\infty$ and $s_{k+1}=+\infty$) 
and 
$t_j<\pi_{\ell}<t_{j+1}$, where $t_1 < t_2 < \ldots < t_k$ are the (sorted) elements of $\{a_{s_1} a_{s_2} \ldots a_{s_k}\}$ (with the convention $t_0=-\infty$ and $t_{k+1}=+\infty$).

To illustrate this concept graphically, we draw the plot of $\tau$ and add the grid lines. 
They split the plane into $(k+1)^2$ regions,
which are naturally labeled by  $\{(i,j)\colon \, 0 \leq i \leq k, 0 \leq j \leq k\}$.
The regions $(i,j) \in \mu$ are then indicated by shading.
An occurrence of $m$ in the plot of $\pi$ is an occurrence of $\tau$ such 
that the interiors of shaded regions do not contain any points of $\pi$.
See~\cite{Branden11} for examples and basic results on mesh patterns.
In Section~\ref{sec:guillotine}, we will work with two mesh patterns, see Figure~\ref{fig:mesh}.

\smallskip
If $\tau_1, \tau_2, \ldots, \tau_p$ are some fixed patterns (of any kind), 
then we denote by $\mathsf{Av}(\tau_1, \tau_2, \ldots, \tau_p)$
the family of permutations that avoid all these patterns. 
A~\textit{permutation class} is any family of permutations that can be specified by avoidance of one or several patterns.

\subsection{Permutation classes}
\label{sec:families}

In this section we list some permutation classes which will play a role in our paper.

\medskip

\noindent \textbf{Separable permutations} are defined as the class $\mathsf{Av}(2413, 3142)$. 
Alternatively, they can be defined as permutations
that can be recursively constructed from the size-1 permutation
by taking direct and skew sums.
The equivalence of the two definitions was proven 
by Ehrenfeucht and Rozenberg~\cite{Ehrenfeucht93},
and the name ``separable permutations'' was coined 
by Bose, Buss, and Lubiw~\cite{BBL98}.
Separable permutations are enumerated by Schröder numbers,
as proven by West~\cite{West95} via generating trees.
For an alternative proof of this fact, note that
the definition of weak guillotine rec\-tan\-gu\-la\-ti\-ons 
and the (second) definition of separable permutations 
yield the same recurrence for their enumerating sequences.

Schröder numbers also enumerate various combinatorial structures, for example \emph{Schröder paths} ---
the lattice walks from $(0,0)$ to $(2n,0)$ that use steps $(1,1), (2,0), (1,-1)$ and stay (weakly) 
above the $x$-axis. Closely related to them are \emph{little Schröder numbers}
(\href{https://oeis.org/A001003}{OEIS A001003}): they were introduced by Ernst Schröder in the context of counting parenthesizations~\cite{schroeder}.
Remarkably, little Schröder numbers were supposedly
mentioned in Plutarch's \emph{Table Talk} (ca.~AD 100)
in the context of counting compound propositions~\cite{Stanley97}. 

The generating function of Schröder numbers is algebraic, it is given above in~\eqref{eq:schr}. Singularity analysis
readily implies their asymptotics
$S_n \sim \frac{(1 + \sqrt{2})^{2n  + 1 }}{2^{3/4} \, \sqrt{\pi n^3}}$
(see~\cite[note VII.19]{FlaSe} and~\cite[A001003, A006318]{oeis}).

\medskip

\noindent \textbf{Baxter permutations} are defined as the class $\mathsf{Av}(2\underline{41}3, 3\underline{14}2)$.
They were introduced by Baxter and Joichi \cite{Baxter64, BaxterJoichi63} in the context of commuting real functions.

Baxter permutations are enumerated by \emph{Baxter numbers} 
(\href{https://oeis.org/A001181}{OEIS A001181}) given by the explicit formula 
\[
B_n = \sum_{k=1}^{n} \frac{ \binom{n+1}{k-1} \binom{n+1}{k} \binom{n+1}{k+1}  }{\binom{n+1}{0}\binom{n+1}{1} \binom{n+1}{2}}.
\]
It was first obtained in 1978 by Chung, Graham, Hoggatt, and Kleiman~\cite{Chung78};
soon after that Mallows~\cite{mallows79} showed that the term corresponding to fixed $k$ is the number of Baxter permutations with precisely $k-1$ descents.
Another proof of this formula, via generating trees, was given by Bousquet-Mélou~\cite{MBM03}.
The generating function of Baxter numbers is D-finite but not algebraic, 
and their asymptotics is $B_n \sim$ 
$ \frac{2^{3n+5}}{n^4 \pi \sqrt{3}}$~\cite[``pointed out by A.~M.~Odlyzko'']{Chung78}.
We refer to Felsner, Fusy, Noy, and Orden~\cite{FFNO11} for a comprehensive survey on 
combinatorial families enumerated by Baxter numbers and bijections between them.

\medskip

\noindent \textbf{Twisted Baxter permutations} are defined as the class $\mathsf{Av}(2\underline{41}3, 3\underline{41}2)$, and \textbf{co-twisted Baxter permutations} as the class $\mathsf{Av}(2\underline{14}3, 3\underline{14}2)$. 
They are, respectively, the minimum and the maximum of the congruence classes associated to weak rec\-tan\-gu\-la\-ti\-ons~\cite{LR12}.

\medskip

\noindent\textit{Remark.} The four patterns $2\underline{41}3, 3\underline{41}2, 2\underline{14}3, 3\underline{14}2$ used in the definition of Baxter, twisted Baxter, and co-twisted Baxter permutations are known as \emph{Baxter-like patterns}. 
Bouvel, Guerrini, Rechnizter and Rinaldi~\cite{BouvelGRR18} and
Bouvel, Guerrini and Rinaldi~\cite{bouvel19} investigated the enumeration of permutation families
defined by avoidance of all possible combinations of these patterns, by means of generating trees.
Five (out of six possible) pairs of Baxter-like patterns yield
permutation classes enumerated by Baxter numbers;
the exceptional combination is $\{2\underline{14}3, 3\underline{41}2\}$.
Permutations that avoid this pair of patterns were studied by 
Asinowski, Barequet, Bousquet-M\'elou, Mansour, and Pinter~\cite{ABBMP13}:
they constitute the ``even part'' of the so-called ``complete Baxter permutations'',
and they are related to orders between \emph{segments} in rec\-tan\-gu\-la\-ti\-ons.

\medskip

\noindent \textbf{2-clumped permutations} are defined by $\mathsf{Av}(24\underline{51}3, 42\underline{51}3, 3\underline{51}24, 3\underline{51}42)$, and \textbf{co-2-clumped permutations} are $\mathsf{Av}(24\underline{15}3, 42\underline{15}3, 3\underline{15}24, 3\underline{15}42)$. 
They are, respectively, the minimum and the maximum of the congruence classes associated to strong rec\-tan\-gu\-la\-ti\-ons~\cite{R12}. 
The enumerating sequence of these classes is \href{https://oeis.org/A342141}{OEIS A342141}, 
and it was proven by Fusy, Narmanli, and Schaeffer~\cite{FNS21} that its generating function is not D-finite (via the enumeration of \emph{transversal structures}, which are dual to  strong rec\-tan\-gu\-la\-ti\-ons).

\medskip

In Sections~\ref{sec:weak} and~\ref{sec:strong}, 
we review and revisit the connection of these classes of permutations
with rec\-tan\-gu\-la\-ti\-ons, also providing a visual interpretation of the congruence classes associated to weak and strong rec\-tan\-gu\-la\-ti\-ons. 
Specifically, in Theorems~\ref{thm:weak_distinguished} and~\ref{thm:weak_bijections},
Baxter, twisted Baxter, and co-twisted Baxter permutations will be linked to weak rec\-tan\-gu\-la\-ti\-ons,
and separable permutations to weak guillotine rec\-tan\-gu\-la\-ti\-ons;
and, in Theorems~\ref{thm:strong_distinguished} and~\ref{thm:strong_bijections},
2-clumped and co-2-clumped permutations will be linked to strong rec\-tan\-gu\-la\-ti\-ons.

\section{Weak rec\-tan\-gu\-la\-ti\-ons}
\label{sec:weak}

In this section we deal with representation of weak rec\-tan\-gu\-la\-ti\-ons by posets and permutations.
It has an expository nature and does not contain any new results, therefore we will present the material, 
mainly from~\cite{ABP06, CSS18, LR12, M19a}, rather briefly and without proofs.
We include it in order to provide a~systematic summary of all relevant material from different contributions,
which makes the comparison with the case of strong rec\-tan\-gu\-la\-ti\-ons especially clear and transparent. 

As mentioned in Section~\ref{sec:diag}, diagonal rec\-tan\-gu\-la\-ti\-ons are considered as canonical representatives of weak rec\-tan\-gu\-la\-ti\-ons.
Therefore, posets and permutations associated with weak rec\-tan\-gu\-la\-ti\-ons will be defined via their diagonal representatives.

\subsection{The weak poset}\label{sec:weak_poset}

We first define the adjacency poset of a rec\-tan\-gu\-la\-ti\-on $\rc$.
Let the rectangles of $\rc$ be labeled with the NW--SE labeling. 
Then, for two rectangles $r_j$ and $r_k$ of $\rc$, 
we define $j \triangleleft k$ 
if $r_j$ and $r_k$ are adjacent, and~$r_j$ is on the left of \textit{or} below $r_k$.
In this case we also say that $r_j$ and $r_k$ \emph{block} each other:
$r_k$ blocks $r_j$ from the top or from the right,
and $r_j$ blocks $r_k$ from the bottom or from the left.
The \emph{adjacency poset} $P_a(\rc)$ is the poset on $[n]$ whose order relation is the 
transitive closure of $\triangleleft$.

Now, given a weak rec\-tan\-gu\-la\-ti\-on $\rc$, its \textit{weak poset} $P_w(\rc)$ 
is defined as the adjacency poset of its diagonal representative $\dc$.
This poset was introduced by Law and Reading in~\cite{LR12}
and thoroughly studied in~\cite{M19a} as a special case of \textit{Baxter posets}.

Note that the adjacency posets of distinct rec\-tan\-gu\-la\-ti\-ons weakly equivalent to $\rc$
may be different. However, all of them are extensions of the adjacency poset 
of the corresponding diagonal rec\-tan\-gu\-la\-ti\-on --- that is, extensions of $P_w(\rc)$.
Figure~\ref{fig:w_poset} shows the rec\-tan\-gu\-la\-ti\-on $\rc_1$ and its adjacency poset $P_a(\rc_1)$, 
as well as the diagonal representative $\dc_1$ and the weak poset $P_w(\rc_1) = P_a(\dc_1)$.
We draw Hasse diagrams of the weak poset via the natural embedding by duality, and,
therefore, the parents of every vertex occur in the increasing order from left to right. 
This representation also implies that the weak poset $P_w(\rc)$ is a planar two-dimensional lattice (compare with Proposition~\ref{prop:planar}).
Indeed, the planarity is inherited from the $\triangleleft$-relation which is an orientation of the dual map of $\rc$. 
The cover relations of $P_w(\rc)$ are a subset of the $\triangleleft$-relations. 
The bounded faces of the lattice correspond to segments of $\rc$ that have neighbors from both sides.

\begin{figure}[!h]
\begin{center}
\includegraphics[height=.243\textheight]{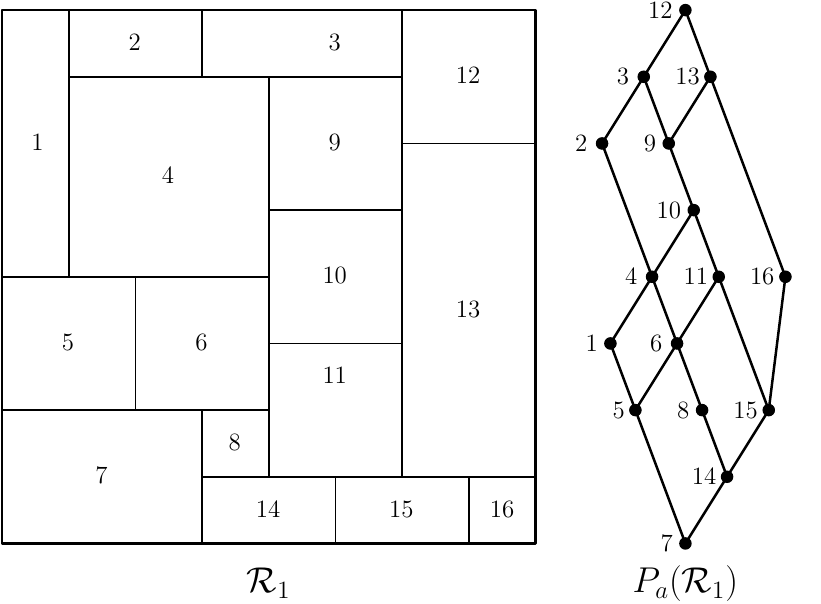}
\includegraphics[height=.243\textheight]{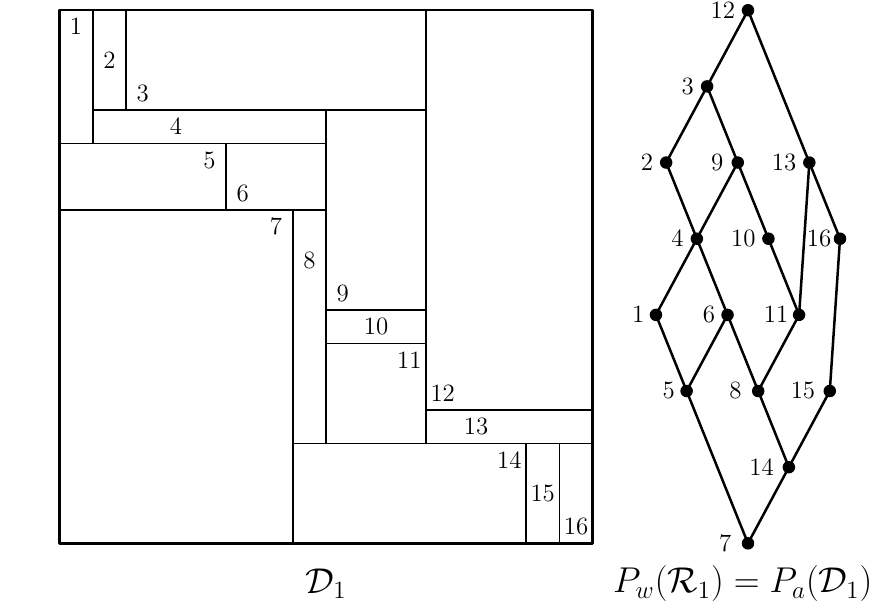} 
\end{center}
\caption{Left: Weak rec\-tan\-gu\-la\-ti\-on $\rc_1$ and its adjacency poset $P_a(\rc_1)$.
Right: The corresponding diagonal rec\-tan\-gu\-la\-ti\-on $\dc_1$ and its adjacency poset $P_a(\dc_1)$,
which is also, by definition, the weak poset $P_w(\rc_1)$.}
\label{fig:w_poset}
\end{figure}

\subsection{Mapping $\gamma_w$ from permutations to weak rectangulations}

Next we describe the fundamental mapping $\gamma_w$, introduced by Law and Reading~\cite{LR12}, 
from the set $S_n$ of permutations of size $n$ to the set $\mathsf{WR}_n$ of weak rec\-tan\-gu\-la\-ti\-ons of size $n$.

Let $\pi\in S_n$.
The corresponding weak rec\-tan\-gu\-la\-ti\-on $\gamma_w(\pi) \in \mathsf{WR}_n$ will be given by its diagonal representative.
It is constructed by the following \textit{forward algorithm} that takes an $n \times n$ grid square $\rr$
and inserts rectangles in the order prescribed by $\pi$ such that at the end 
a diagonal rec\-tan\-gu\-la\-ti\-on is obtained. 
At each step, a~\textit{partial rec\-tan\-gu\-la\-ti\-on} --- the union of already inserted rectangles ---
is bounded by a horizontal segment from the bottom, 
a vertical segment from the left, 
and a monotonically decreasing \emph{staircase} from the top-right. It is also convenient to adjoin two fictitious rectangles $r_0$ and $r_{n+1}$
that occupy respectively the column to the left of the grid, and the row below the grid. 
Accordingly the staircase is extended horizontally at its top-left end, and vertically at its bottom-right end.
The turning points of the staircase are referred to
as \emph{peaks} \pe\ and \emph{valleys}~\va.
Every peak is labeled according to the rectangle incident to it within the partial rec\-tan\-gu\-la\-ti\-on.

\begin{mdframed}

\textbf{Algorithm~WF (weak forward): Permutations to weak rec\-tan\-gu\-la\-ti\-ons.}\label{alg:wf}

\noindent Input: Permutation $\pi = \pi_1 \pi_2 \ldots \pi_n \in S_n$.

\noindent Output: Weak rec\-tan\-gu\-la\-ti\-on $\rc = \gamma_w(\pi)$.

\begin{enumerate}
\item Draw an $n \times n$ square grid $\rr$, and  label its diagonal cells by $1, 2, \ldots, n$ from the top-left to the bottom-right corner.
Amend them by an auxiliary rectangle $r_0$  
in the column to the left of the grid, and an auxiliary rectangle $r_{n+1}$  
in the row below the grid.
\item Initialize the \emph{staircase} to be the union of the left side and the bottom side of $\rr$, extended by a horizontal unit-segment at the beginning, and a vertical unit-segment at the end.
Initialize the set of its \emph{peaks} to be $P:=\{0, n+1\}$.
\item For $i$ from $1$ to $n$, with $j=\pi_i$:\\[1mm]
Insert rectangle $r_j$ according to the following rules. 
\begin{itemize}
\item The bottom-left corner of $r_{j}$ is the valley delimited by the two consecutive peaks of $P$ with labels $a$ and $b$ such that $a<{j}<b$.
\item If all rectangles 
$r_k$ with $a < k < {j}$
have already been inserted, 
then the top side of $r_{j}$ aligns with the top side of $r_{a}$. 
In this case, $a$ is deleted from $P$.
Otherwise, the top side of $r_{j}$ is contained in the horizontal grid line that
separates rows ${j}-1$ and~${j}$.
\item If all rectangles 
$r_k$ with ${j} < k < b$
have already been inserted, 
then the right side of $r_{j}$ aligns with the right side of $r_{b}$. 
In this case, $b$ is deleted from $P$.
Otherwise, the right side of $r_j$ is contained in the vertical grid line that
separates columns ${j}$ and ${j}+1$.

\item Update the staircase by replacing 
the union of left and bottom sides of $r_{j}$ 
with the union of its top and right sides.
Add ${j}$ to $P$.
\end{itemize}
\end{enumerate}
\end{mdframed}

An example of executing this algorithm is shown in Figure~\ref{fig:gamma_w}.

\begin{figure}[p]
\begin{center}
\includegraphics[width=\textwidth]{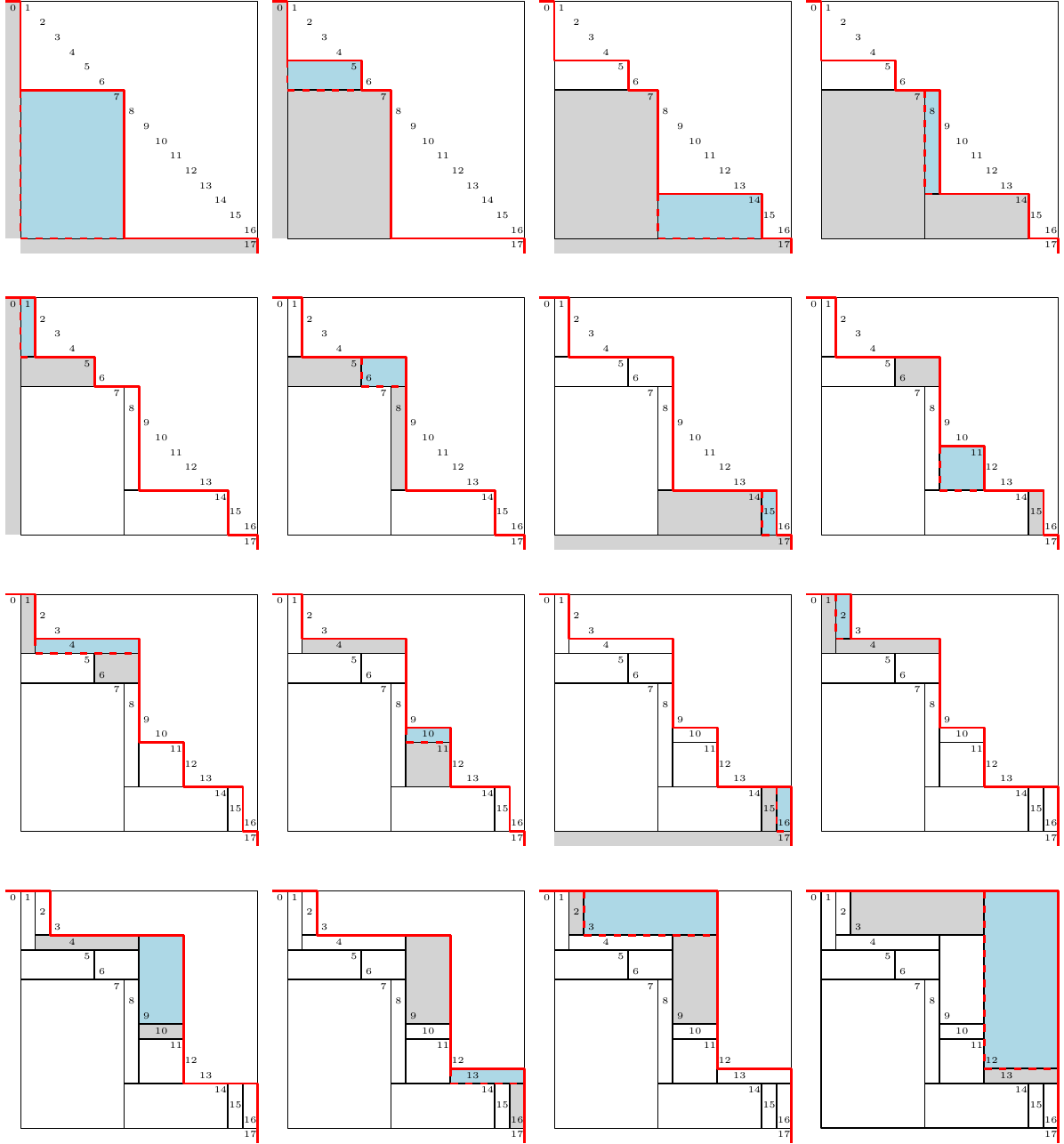} 
\end{center}
\caption{Constructing $\gamma_w(\pi)$ for $\pi =  7 \ \ 5 \ \ 14 \ \ 8 \ \ 1 \ \ 6 \ \ 15 \ \ 11 \ \ 4 \ \ 10 \ \ 16 \ \ 2 \ \ 9 \ \ 13 \ \ 3 \ \ 12$. 
At each step, the inserted rectangle is blue,
and the rectangles incident to the adjacent peaks are grey.}
\label{fig:gamma_w}
\end{figure}

\subsection{Fibers}\label{sec:fibers_weak}

The mapping $\gamma_w$ is surjective but not injective.
Given a weak rec\-tan\-gu\-la\-ti\-on $\rc$ of size $n$, one can recover all permutations $\pi\in S_n$ such that $\gamma_w(\pi)=\rc$ by applying 
the following \textit{backward algorithm} which in fact reverses Algorithm~WF.
Here, a rectangle $r_j$ of a partial rec\-tan\-gu\-la\-ti\-on $\tilde{R}$ is \textit{available}
if it is not blocked from top or from right by some other rectangle of $\tilde{R}$ ---
that is, if there is no rectangle $r_k$ of $\tilde{R}$ such that $j \triangleleft k$.

\begin{mdframed}\label{alg:wb}
\textbf{Algorithm WB (weak backward): Weak rec\-tan\-gu\-la\-ti\-ons to permutations.} 

\noindent Input: Weak rec\-tan\-gu\-la\-ti\-on $\rc \in \mathsf{WR}_n$.

\noindent Output: A permutation $\pi \in S_n$ such that $\pi \in \gamma_w^{-1}(\rc)$.

\begin{enumerate}
\item Consider $\dc$, the diagonal representative $\rc$. 
\item Label the rectangles of $\dc$ by the NW--SE labeling.
\item For $i$ from $n$ to $1$:\\[1mm]
   Remove an available rectangle $r_{j}$. Set $\pi_{i}={j}$.
\end{enumerate}
\end{mdframed}

Given a poset $P$, denote the set of its linear extensions by $\mathcal{L}(P)$. 
The following results are shown in~\cite[Section~6]{LR12}.
\begin{proposition}\label{prop:weakfiber}
Let $\rc$ be a weak rec\-tan\-gu\-la\-ti\-on.
\begin{enumerate}
\item At every step of Algorithm~WB there is at least one available rectangle.
\item The set of permutations that can be generated by  Algorithm~WB is precisely the fiber $\gamma_w^{-1}(\rc)$.
\item It is also the set of linear extensions of the weak poset of $\rc$:
\[\gamma_w^{-1}(\rc) = \mathcal{L}(P_w(\rc)).\]
\end{enumerate}
\end{proposition}

Figure~\ref{fig:gamma_w}, read backwards, demonstrates how $\pi$ is obtained by Algorithm~WB as one of the preimages of~$\dc_1$.
According to Proposition~\ref{prop:weakfiber}, the permutations that can be obtained in this way 
are precisely the linear extensions of the poset $P_w(\rc)$ from Figure~\ref{fig:w_poset}.

Finally, we remark that both algorithms can be performed from the opposite corner:
in the forward algorithm one can start inserting rectangles from the top-right corner,
and in the backward algorithm one can start removing rectangles from the bottom-left corner, with obvious adjustments of the rules. 
In both cases, the modified algorithms lead to the same results.

\subsection{Baxter, twisted Baxter, and co-twisted Baxter permutations}\label{sec:baxter}

By Proposition~\ref{prop:weakfiber}, there is a bijection between $\mathsf{WR}_n$ and a family of posets on $[n]$,
and the linear extensions of all these posets cover the entire $S_n$.
Then we have 
Theorem~\ref{thm:weak_distinguished} concerning distinguished elements of $\mathcal{L}(P_w(\rc))$,
and Theorem~\ref{thm:weak_bijections} concerning bijective restrictions of $\gamma_w$
to some permutation classes mentioned in Section~\ref{sec:families}.
These results were proven in several contributions, including~\cite{ABP06,CSS18,LR12,M19a}.

\begin{theorem}
\label{thm:weak_distinguished}
Let $\rc$ be a weak rec\-tan\-gu\-la\-ti\-on, with its rectangles labeled by the NW--SE labeling. 
Then:
\begin{enumerate}
\item $\mathcal{L}(P_w(\rc))$ contains a unique twisted Baxter permutation. It is the minimum element of $\mathcal{L}(P_w(\rc))$ with respect to the weak Bruhat order. 
\item $\mathcal{L}(P_w(\rc))$ contains a unique co-twisted Baxter permutation. It is the maximum element of $\mathcal{L}(P_w(\rc))$ with respect to the weak Bruhat order.
\item $\mathcal{L}(P_w(\rc))$ contains a unique Baxter permutation. It is obtained by reading the labels of the rectangles of $\rc$ in the SW--NE (anti-diagonal) order.
\end{enumerate}
\end{theorem}

\begin{theorem}
\label{thm:weak_bijections}
The mapping $\gamma_w$ restricts to three bijections between weak rec\-tan\-gu\-la\-ti\-ons and permutation classes: 
\begin{enumerate}
\item A bijection $\beta_\mathsf{TB}$ between weak rec\-tan\-gu\-la\-ti\-ons and twisted Baxter permutations;
\item A bijection $\beta_\mathsf{CTB}$ between weak rec\-tan\-gu\-la\-ti\-ons and co-twisted Baxter permutations;
\item A bijection $\beta_\mathsf{B}$ between weak rec\-tan\-gu\-la\-ti\-ons and Baxter permutations.
\end{enumerate}
Moreover, the bijection $\beta_\mathsf{B}$ restricts to a bijection $\beta_\mathsf{S}$ between weak guillotine rec\-tan\-gu\-la\-ti\-ons and separable permutations.

\end{theorem}

Note that, given Proposition~\ref{prop:weakfiber}, the three items of Theorem~\ref{thm:weak_distinguished} imply
the corresponding items of Theorem~\ref{thm:weak_bijections}. Hence we only have to care of Theorem~\ref{thm:weak_distinguished}.
Since $P_w(\rc)$ forms an interval in the weak Bruhat order,
the minimum (respectively maximum) of this interval can be obtained by iteratively choosing and deleting the leaf with the smallest (respectively largest) label.    
Since the leaves (current minima) have increasing labels from left to right in the ``embedded'' Hasse diagram of $P_w(\rc)$,
this corresponds to pruning the leftmost (respectively rightmost) leaf at every step.
Hence, we also refer to the minimum and the maximum elements of $\mathcal{L}(P_w(\rc))$,
with respect to the weak Bruhat order,
as the \emph{leftmost} and the \emph{rightmost} linear extensions of $P_w(\rc)$.
We will denote them by $\pi_L$ and $\pi_R$.
Then Theorem~\ref{thm:weak_distinguished}(1,2) says
that the twisted Baxter and the co-twisted Baxter representatives of $P_w(\rc)$
are precisely $\pi_L$ and $\pi_R$.
These two linear extensions are a realizer of the 2-dimensional poset $P_w(\rc)$, 
which (as mentioned in Section~\ref{sec:weak_poset}) is a~planar lattice. 
Figure~\ref{fig:w_permutations} shows the twisted Baxter, co-twisted Baxter, and Baxter representatives of $P_w(\rc_1)$. 

\begin{figure}[!h]
\begin{center}
\includegraphics[width=\textwidth]{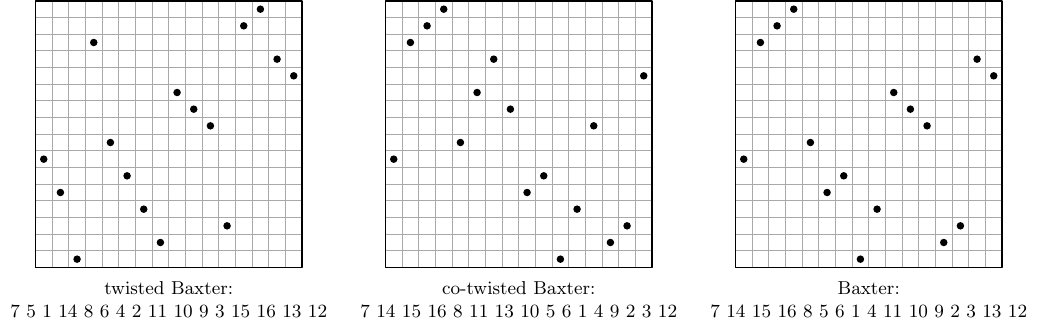} 
\end{center}
\caption{The twisted Baxter, co-twisted Baxter, and Baxter representatives of $P_w(\rc_1)$.}
\label{fig:w_permutations}
\end{figure}

\section{Strong rec\-tan\-gu\-la\-ti\-ons}
\label{sec:strong}

In this section, we consider strong rec\-tan\-gu\-la\-ti\-ons.
The set of strong rec\-tan\-gu\-la\-ti\-ons of size $n$ will be denoted by $\mathsf{SR}_n$. 
We first discuss their representation by posets and permutations, similarly to the weak case.
We define the strong poset of a rec\-tan\-gu\-la\-ti\-on, $P_s(\rc)$,
and a surjective mapping $\gamma_s$ from $S_n$ to~$\mathsf{SR}_n$.
The fibers of this mapping define equivalence classes of permutations, 
which are exactly the linear extensions of the strong poset.
In these fibers, we identify two particular representatives --- 
2-clumped and co-2-clumped permutations, both in bijection with strong rec\-tan\-gu\-la\-ti\-ons. 
 This part includes an alternative treatment of results from~\cite{R12}:
in particular, our descriptions of $P_s(\rc)$ and $\gamma_s$
lead to a simple geometric proof of the bijections.  The new proof makes the
correspondence between patterns in rectangulations and patterns in permutations
more transparent, additionally, it simplifies the identification of the flip graph
of strong rectangulations, and it is well suited to encode strong rectangulations by quadrant walks.

\subsection{The strong poset}

Let $\rc$ be a strong rec\-tan\-gu\-la\-ti\-on of size $n$.
Label the rectangles of $\rc$ with their NW--SE labeling. 
We set $a\btri b$ if one of the following four conditions hold:
\begin{enumerate}
  \item \emph{adjacency relations} (earlier denoted by $\triangleleft$ and also called blocking):
  \begin{enumerate}
\item \label{rel:b1} $r_a$ and $r_b$ are adjacent, and $r_a$ is on the left of $r_b$,
\item \label{rel:b2} $r_a$ and $r_b$ are adjacent, and $r_a$ is below $r_b$;
  \end{enumerate}
\item \emph{special relations} (see Figure~\ref{fig:special}):
  \begin{enumerate}
\item \label{rel:s1} the right side of $r_b$ lies on the same vertical segment as the left side of $r_a$, and the bottom-right corner of $r_b$ lies above the top-left corner of $r_a$ on this segment,
\item \label{rel:s2} the top side of $r_b$ lies on the same horizontal segment as the bottom side of $r_a$, and the top-left corner of $r_b$ lies on the right of the bottom-right corner of $r_a$ on this segment.
  \end{enumerate}
\end{enumerate}

\begin{figure}[!h]
  \begin{center}
    \includegraphics[page=4, scale=.5]{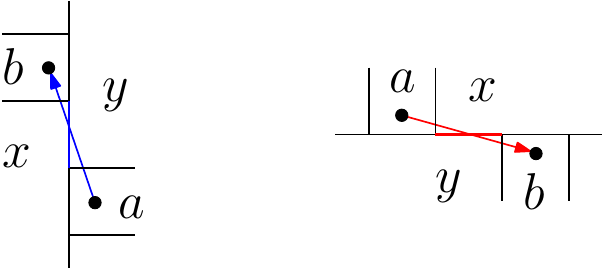}
  \end{center}
  \caption{The special relations in the definition of the strong poset. 
  In both cases, we have $a\btri b$.}
  \label{fig:special}
\end{figure}

As above, the adjacency poset $P_a(\rc)$ is the transitive closure of the adjacency relation $\triangleleft$.
Now, let $\prec_s$ be the transitive closure of $\btri$.
Note that the special relations \ref{rel:s1} and \ref{rel:s2} yield an extension of $P_a(\rc)$.
\begin{proposition}\label{prop:is_poset}
The relation  $\prec_s$ is a partial order on $[n]$.
\end{proposition}
\begin{proof}
To prove that $\prec_s$ is acyclic, we show that there is a linear
order~$\lambda$ on the rectangles of $\rc$ which respects the relations of $\prec_s$, and such that the union of
rectangles in any prefix of $\lambda$ is a staircase.  The order
$\lambda$ is constructed element by element. Consider the staircase
formed by the taken elements, and let $b_1,\ldots,b_k$ be the labels
of the rectangles of $\rc$ whose bottom-left corners correspond to 
a~valley of the staircase, listed in the left-to-right order.  Note that
the left side of $b_1$ is contained in the staircase.
If~$r_{b_1}$ is a
minimal non-taken element with respect to $\prec_s$ --- that is,
there is no other non-taken element $a$ such that $a \prec_s b_1$ ---
then we select $r_{b_1}$ as the next element for $\lambda$: after
that, the taken elements still form a staircase.  Otherwise, there is
an $a$ such that $a \triangleleft b_1$ or $a \btri b_1$ due to a special
relation.  If we have $a \triangleleft b_1$, then the bottom side of $r_{b_1}$ extends beyond the peak 
that separates the valleys for $r_{b_1}$ and $r_{b_2}$ in the staircase. 
If $a \btri b_1$ due to a special  relation, then the peak
that separates the valleys for $r_{b_1}$ and $r_{b_2}$ in the staircase  
is the bottom-right corner of $r_{b_1}$.
In both cases we see that $b_2 \prec_s b_1$, and the left side of $r_{b_2}$
  is contained in the staircase.  Iterating, we obtain a maximum
  length chain $b_\ell \prec_s b_{\ell-1} \prec_s \ldots \prec_s b_2
  \prec_s b_1$, and the left sides of all $r_{b_j}$ for $1\leq
  j\leq\ell$ are contained in the staircase. The bottom side of
  $r_{b_\ell}$ is contained in the staircase, since otherwise we have
  $b_{\ell +1} \prec_s b_\ell$. Hence, $r_{b_\ell}$ is a~minimal
  element which can be added to $\lambda$, such that the remaining
  elements form a staircase.
\end{proof}

 We refer to $\prec_s$ as the \emph{strong order}, and refer to the set $[n]$ partially
  ordered with respect to $\prec_s$ as the \emph{strong poset}
  $P_s(\rc)=([n],\prec_s)$ of~$\rc$. 
In Figure~\ref{fig:hasse} we show how the strong poset $P_s(\rc_1)$ is obtained in two
steps from the weak poset $P_w(\rc_1)$:
first, new adjacencies obtained by shuffling yield $P_a(\rc_1)$, the adjacency poset of $\rc_1$,
and then the special relations yield the strong poset $P_s(\rc_1)$.

\begin{figure}[!h]
  \begin{center}
 \includegraphics[scale=0.7]{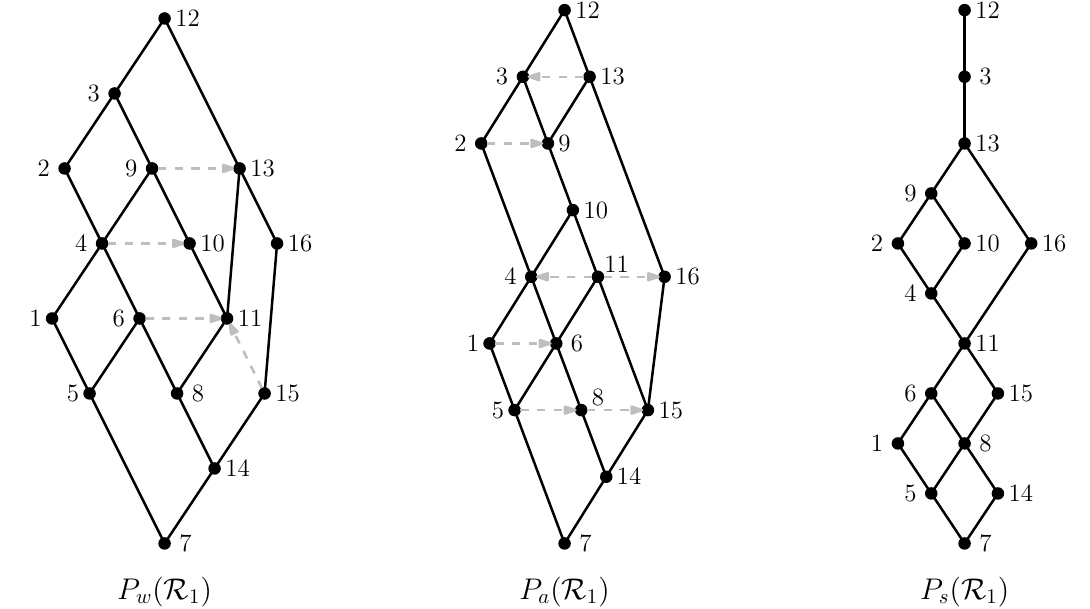} 
  \end{center}
  \caption{Hasse diagrams of three posets associated with $\rc_1$.
    Left: Solid black edges form $P_w(\rc_1)$, the weak poset of $\rc_1$
    (or: $P_a(\dc_1)$ the adjacency poset of $\dc_1$);
    dashed grey edges are new adjacencies contributed by shuffling $\dc_1$ into $\rc_1$.
    Middle: Solid black edges form $P_a(\rc_1)$, the adjacency poset of $\rc_1$;
    dashed grey edges are contributed by special relations.
    Right: $P_s(\rc_1)$, the strong poset of $\rc_1$.}
    \label{fig:hasse}
  \end{figure}

\begin{proposition}
  \label{prop:planar}
  The strong poset $P_s(\rc)$ is a planar two-dimensional lattice.
\end{proposition}
\begin{proof}
It is known that planar bounded posets are two-dimensional lattices (see for instance Baker, Fishburn, and Roberts~\cite{BFR72}). It therefore suffices to prove that $P_s(\rc)$ is planar. For this, we consider the planar drawing of the adjacency graph of the rectangles of $\rc$ obtained by choosing a point in each rectangle and connecting points in adjacent rectangles by an arc that intersect corresponding edges of the rec\-tan\-gu\-la\-ti\-on. 
When oriented from left to right and from bottom to top, these edges give all arcs 
that correspond to adjacency conditions~\ref{rel:b1} and~\ref{rel:b2}.
Next, we remove all the edges implied by transitivity and obtain the diagram of the adjacency poset.
It remains to show that the arcs corresponding to the special relations~\ref{rel:s1} and~\ref{rel:s2} can be added without creating crossings. 
  For this, we need two observations:
  \begin{itemize}
  \item Every covering special relation is associated with an edge of the rec\-tan\-gu\-la\-ti\-on:
a vertical edge that connects the top-left corner of $r_a$ and the bottom-right corner of $r_b$
or a~horizontal edge that connects the bottom-right corner of $r_a$ and the top-left corner  of $r_b$
(refer to Figure~\ref{fig:special}). 
Hence, if we need to draw an arc between two such rectangles $r_a$ and~$r_b$, 
then these two rectangles are separated by a single edge $s$. 
  \item The two rectangles $r_x$ and $r_y$ that share edge $s$ 
  are not in the covering relation of the adjacency poset: their adjacency order $x \prec y$ is implied by transitivity.
    Indeed, referring again to Figure~\ref{fig:special}, we have either that $r_x$ is below $r_b$ which is left of $r_y$ (special relation~\ref{rel:s1}), or $r_x$ is above $r_b$ which is right of $r_y$ (special relation~\ref{rel:s2}).
    Hence, we can draw the corresponding arc from $r_a$ to $r_b$ without crossing another arc.
  \end{itemize}
  These two observations allow us to draw the arcs corresponding to the special relations
  without creating crossing arcs.
  Together with the arcs corresponding to the adjacency conditions~\ref{rel:b1} and~\ref{rel:b2}, they yield a~planar drawing of the Hasse diagram of the strong poset.
See Figure~\ref{fig:strong_poset} for an example.
\end{proof}

\begin{figure}[!h]
  \begin{center}
\includegraphics[width=0.865\textwidth]{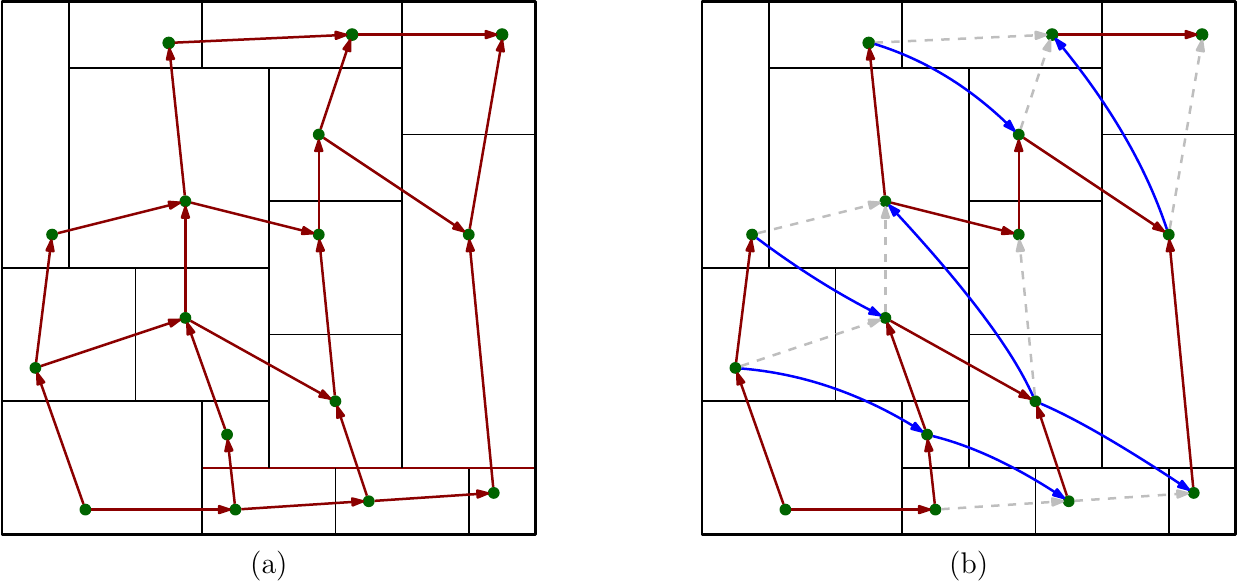} 
  \end{center}
\caption{Illustration to the proof of Proposition~\ref{prop:planar}.
(a) The Hasse diagram of $P_a(\rc_1)$, the adjacency poset of $\rc_1$.
(b) Solid arcs form the Hasse diagram of $\prec_s$
(the blue arcs are contributed by special relations).
The dashed grey arcs belong to $P_a(\rc_1)$,
but in $P_s(\rc_1)$ they are implied by transitivity.} 
\label{fig:strong_poset}  
\end{figure}

\subsection{Mapping $\gamma_s$ from permutations to strong rectangulations}

In~\cite{R12}, Reading defined a mapping $\gamma$ from $S_n$ to $\mathsf{SR}_n$, 
whose restriction yields a bijection between 2-clumped permutations and strong rec\-tan\-gu\-la\-ti\-ons. 
His construction of $\gamma(\pi)$ consists of two steps: first, one constructs the weak rec\-tan\-gu\-la\-ti\-on corresponding to $\pi$ --- what we denote by $\gamma_w(\pi)$. 
Then, one shuffles the neighbors of every segment $s$ according to the order in which 
the labels of rectangles adjacent to $s$ occur in $\pi$. 

In this section we offer an alternative description of $\gamma$ (which we denote by $\gamma_s$),
which consists of just one step and uses a modification of Algorithm~WF. 
Our description emphasizes both the parallelism and the difference between 
the weak and the strong cases, thus contributing to better understanding of both kinds of 
equivalence. It also leads to very transparent and descriptive proofs concerning the structure
of the strong posets and their linear extensions.

\smallskip

We define the mapping $\gamma_s\colon S_n\to \mathsf{SR}_n$
via a \textit{forward algorithm} that constructs the rec\-tan\-gu\-la\-ti\-on incrementally:
we read the permutation $\pi=\pi_1\pi_2\ldots\pi_n$ from left to right, 
and insert the rectangle with label $j=\pi_i$ successively, for $i=1,2,\ldots ,n$. 
The following invariants hold at every step.
\begin{enumerate}
\item The \textit{partial rec\-tan\-gu\-la\-ti\-on} 
is bounded by a horizontal segment from the bottom, 
a vertical segment from the left, 
and a monotonically decreasing \emph{staircase} from the top-right. Similarly to the weak case, we imagine fictitious thin rectangles $r_0$ and $r_{n+1}$ respectively to the left of the left boundary, and below the bottom boundary, and accordingly we extend the staircase horizontally at its 
top-left end and vertically at its bottom-right end.
We refer to the turning points of the staircase as \emph{peaks} \pe\ and \emph{valleys} \va.
\item The labels of the rectangles corresponding to the peaks are in increasing order from top-left to bottom-right, with labels of consecutive peaks differing by at least $2$. 
\end{enumerate}

\begin{mdframed}
\textbf{Algorithm SF (strong forward): Permutations to strong rec\-tan\-gu\-la\-ti\-ons.} 

\noindent Input: Permutation $\pi=\pi_1 \pi_2 \ldots \pi_n \in S_n$.

\noindent Output: Strong rec\-tan\-gu\-la\-ti\-on $\rc = \gamma_s(\pi)$.

\begin{enumerate}
\item Initialize the \emph{staircase} to be made of two peaks, of labels $0$ and $n+1$, and one valley. 
The set $P$ of peaks is thus $\{0, n+1\}$ initially.
\item For $i$ from $1$ to $n$, with $j=\pi_i$:\\[1mm]
Insert rectangle $r_{j}$ according to the following rules. 
\begin{itemize}
\item The bottom-left corner of $r_{j}$ is the valley delimited by the two consecutive peaks of $P$
 with labels $a$ and $b$ such that $a<{j}<b$.

\item If all rectangles $r_k$ with $a < k < {j}$ have already been inserted,
then the top side of $r_{j}$ aligns with the top side of $r_{a}$,
and $a$ is removed from $P$.
Otherwise, the top side of $r_{j}$ forms a~$\tr$ with the right side of $r_{a}$.

\item If all rectangles $r_k$ with ${j} < k < b$ have already been inserted, 
then the right side of $r_{j}$ aligns with the right side of $r_{b}$,
and $b$ is removed from $P$.
Otherwise, 
the right side of $r_{j}$ forms a~$\tu$ with the top side of $r_{b}$.

\item Update the staircase by replacing 
the union of left and bottom sides of $r_{j}$ 
with the union of its top and right sides.
Add ${j}$ to $P$.
\end{itemize}
\end{enumerate}
\end{mdframed}

It is straightforward to verify that the algorithm maintains the invariants above, 
that it produces a rec\-tan\-gu\-la\-ti\-on, and that the 
labeling of the rectangles is the NW--SE labeling. 
Algorithm~SF is schematically illustrated in Figure~\ref{fig:gamma},
and an example of its execution for a permutation of size $n=16$ is given in Figure~\ref{fig:map}.
The four cases of placements in the Algorithm correspond 
to the four cases considered 
by Takahashi, Fujimaki, and Inoue~\cite{TFI09} in their encoding procedure, 
and by Fran\c{c}on and Viennot~\cite{FV79} in the proof of their Theorem~2.2 (the current valleys of the partial rec\-tan\-gu\-la\-ti\-on correspond to the current gaps in their iterative encoding of $\pi^{-1}$).

\begin{figure}[!h]
\begin{center}
\includegraphics[scale=0.78]{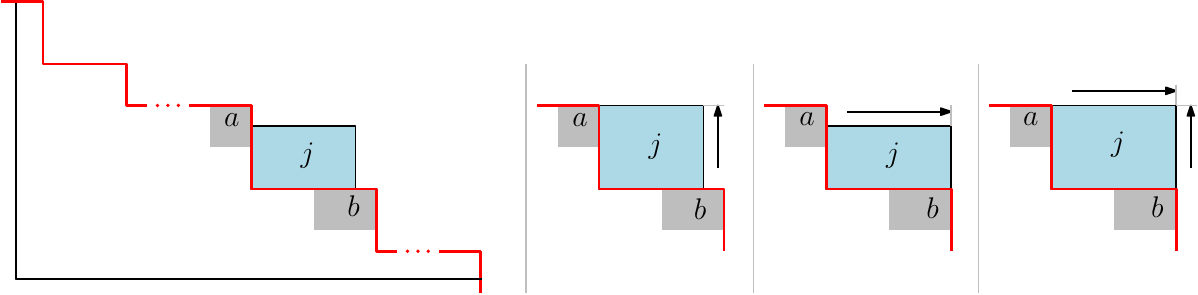}
\caption{Illustration of the four cases in Algorithm~SF.
The rectangle with label $j$ is inserted between the peaks of labels $a$ and $b$ such that $a<j<b$. 
The top (right) sides of $r_j$ are extended, respectively, 
upwards (to the right) to align with $r_a$ (with $r_b$) 
if all rectangles with intermediate labels have already been inserted. }
\label{fig:gamma}
\end{center}
\end{figure}

\begin{figure}[hp]
\begin{center}
\includegraphics[width=\textwidth]{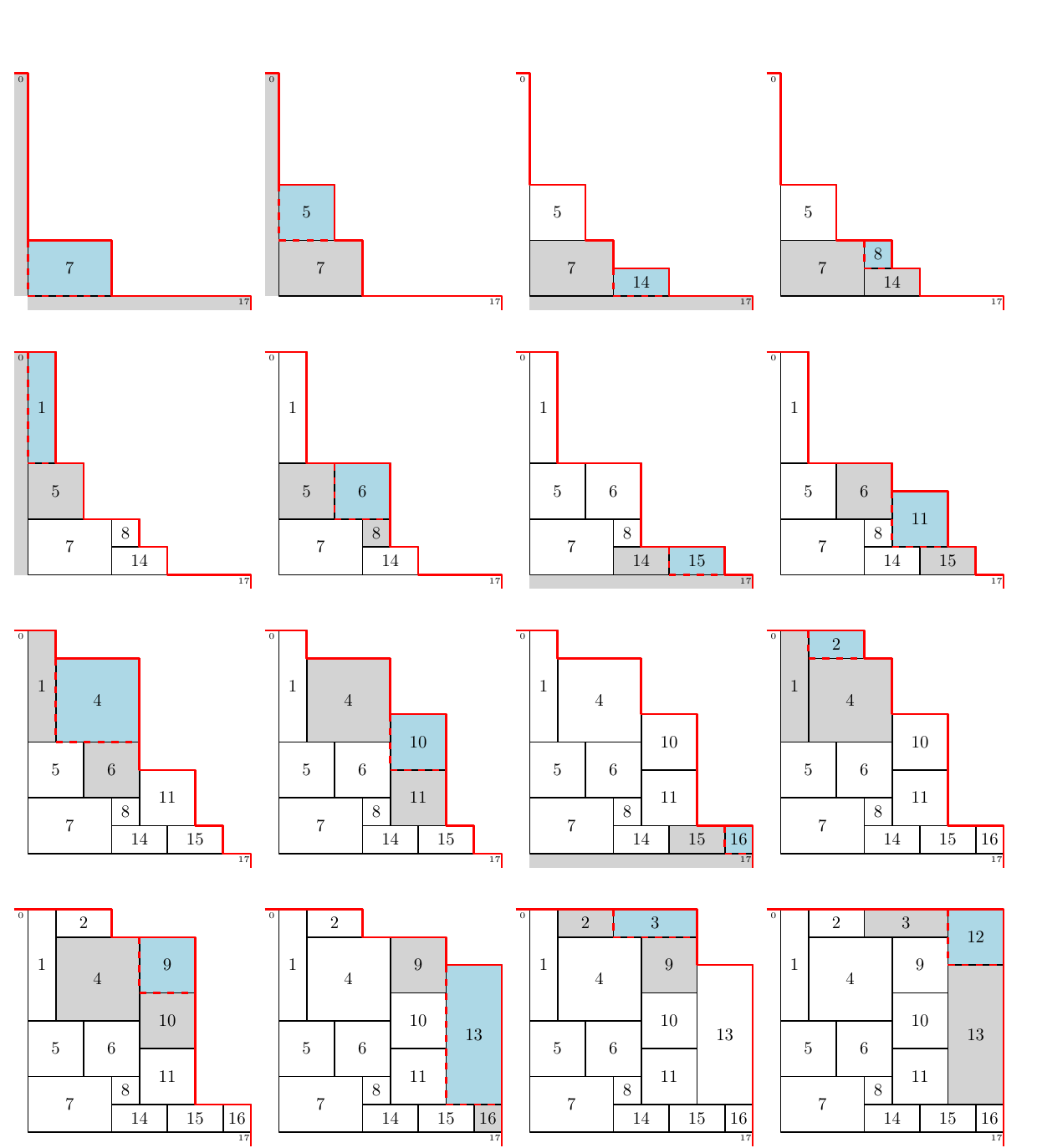} 
\end{center}
\caption{Constructing $\gamma_s(7 \ \ 5 \ \ 14 \ \ 8 \ \ 1 \ \ 6 \ \ 15 \ \ 11 \ \ 4 \ \ 10 \ \ 16 \ \ 2 \ \ 9 \ \ 13 \ \ 3 \ \ 12)$. 
At each step, the inserted rectangle is blue, and the rectangles incident to the adjacent peaks are grey.}
\label{fig:map}
\end{figure}

\subsection{Fibers}\label{sec:fibers_strong}

Given a strong rec\-tan\-gu\-la\-ti\-on $\rc$ of size $n$, we now describe a method to recover any permutation $\pi\in S_n$ such that $\gamma_s(\pi)=\rc$.
This method iteratively removes rectangles, starting from the top right rectangle, and ending with the bottom left rectangle, and 
constructs a permutation $\pi$ ``from right to left''.
As in Algorithm~WB, we first label the rectangles of $\rc$ by the NW--SE labeling.
However, in this case the definition of available rectangles is slightly more involved.
Let $\tilde{\rc}$ be the partial rec\-tan\-gu\-la\-ti\-on of the not yet taken rectangles, the choice of rectangles deleted at each step maintaining the property that its top-right boundary is a staircase. 
Precisely, a rectangle $r_\ell$ in $\tilde{\rc}$ is called \emph{available} if it satisfies the following conditions, where by convention the top-left corner of $\rc$ is a~$\td$, and the bottom-right corner is a~$\tl$:
\begin{itemize}
\item The top side and the right side of $r_\ell$ are entirely contained in the staircase.
\item The top-left corner of $r_{\ell}$ has the shape $\td$\hspace{1pt};
or it has the shape $\tr$, and the rectangle adjacent to this point from left contains the previous peak.
\item The bottom-right corner of $r_{\ell}$ has the shape $\tl$\hspace{1pt};
or it has the shape $\tu$, and the rectangle adjacent to this point from bottom contains the next peak.
\end{itemize}

\begin{mdframed}\label{alg:sb}
\textbf{Algorithm SB (strong backward): Strong rec\-tan\-gu\-la\-ti\-ons to permutations.} 

\noindent Input: Strong rec\-tan\-gu\-la\-ti\-on $\rc \in \mathsf{SR}_n$.

\noindent Output: A permutation $\pi \in S_n$ such that $\pi \in \gamma_s^{-1}(\rc)$.
\begin{enumerate}
\item Label the rectangles of $\rc$ by the NW--SE labeling.
\item For $i$ from $n$ to $1$:\\[1mm]
 Remove any available rectangle $r_{j}$, and set $\pi_{i}={j}$.
\end{enumerate}
\end{mdframed}

The next results show the validity of Algorithm~SB.

\begin{lemma}
  The set of permutations that can be constructed by Algorithm~SB is exactly $\gamma_s^{-1}(\rc)$.
\end{lemma}
\begin{proof}
  At every step of the execution of the forward algorithm, 
  the last rectangle that has been inserted is by definition available among the rectangles that have already been inserted.
  Conversely, an available rectangle removed by the backward algorithm is one that could have been inserted by the forward algorithm in the same situation.
  Hence any execution of the forward algorithm can be mirrored to yield a sequence of rectangles removed by the backward algorithm, and vice versa.
\end{proof}

Then, the following  determines precisely how the structure of Algorithms SF and SB is connected with the strong poset.
Recall that a subset $S$ of the poset $P_s(\rc)$ is a~\emph{downset} if it is closed for the relation $\prec_s$,
hence if $x\in S$ and $y\prec_s x$, then $y\in S$.
\begin{lemma}
  \label{lem:backward}
  At every step of Algorithm~SB, the set of labels of the remaining rectangles is a downset of $P_s(\rc)$, and a rectangle $r_{j}$ is available if and only if $j$ is maximal with respect to $\prec_s$ in that set. 
\end{lemma}
\begin{proof}
We first observe that $r_{j}$ is available if and only if 
none of the remaining rectangles $r_k$ satisfies $k\succ_s {j}$.
Indeed, if $r_{j}$ is available, then the remaining rectangles $r_k$ can touch neither the top nor the right side of $r_{j}$, and from the definition of availability, cannot be located 
 as $a$ in the special relations shown in Figure~\ref{fig:special}.  Conversely, if there is no rectangle $r_k$ such that $k\succ_s {j}$, then $r_{j}$ is available.
  Since the backward algorithm removes an available rectangle at each step, the set of remaining rectangles is always a downset of $P_s(\rc)$.
\end{proof}

Lemma~\ref{lem:backward} implies the following analogue of Proposition~\ref{prop:weakfiber}.

\begin{proposition}
  \label{prop:strongfiber}
  Let $\rc$ be a strong rec\-tan\-gu\-la\-ti\-on of size $n$. 
  Then the fiber $\gamma_s^{-1}(\rc)$ is exactly the set of linear extensions of $P_s(\rc)$:
  \[
  \gamma_s^{-1}(\rc) = \mathcal{L}(P_s(\rc)).
  \]
\end{proposition}

Recall that the skeleton graph ${\cal G}_n$ of the \emph{permutahedron} is  the graph on $S_n$ with edges corresponding to adjacent transpositions. This
  graph can also be viewed as the cover graph of the weak Bruhat order  on~$S_n$. From Proposition~\ref{prop:planar}, we know that $P_s(\rc)$ is a planar
  two-dimensional lattice. A realizer of size two of~$P_s(\rc)$ is given by  the pair $\{\pi_L,\pi_R\}$ 
  where $\pi_L$ is the leftmost and $\pi_R$ is the rightmost linear extension 
  (the definition of the leftmost and the rightmost linear extensions was given in Section~\ref{sec:baxter}). 
  This implies that the set
  $\mathcal{L}(P_s(\rc))$ of linear extensions is the convex set spanned by
  $\pi_L$ and $\pi_R$ in ${\cal G}_n$, i.e., the set of permutations that
  belong to shortest $\pi_L,\pi_R$ paths in ${\cal G}_n$ (see for instance
  Theorem 6.8 in Bj\"orner and Wachs~\cite{BW91}, or Felsner and
  Wernisch~\cite{FW97}). Due to the NW--SE labeling of $\rc$ we know that if
  $a$, $b$ is an incomparable pair then $a$ is left of $b$ in $P_s(\rc)$ if
  and only if $a < b$ in the labeling. Hence $\pi_L$ is the element of
  $\mathcal{L}(P_s(\rc))$ with minimal set of inversions and $\pi_R$ is the
  one with maximal set of inversions. This implies the following:

\begin{proposition}
  \label{prop:interval}
  Given a strong rec\-tan\-gu\-la\-ti\-on $\rc$, the set $\mathcal{L}(P_s(\rc))$ of linear extensions of its strong poset induces an interval in the weak Bruhat order on $S_n$.
\end{proposition}

In the next section, we describe the maximum and the minimum of these intervals.

\subsection{2-clumped and co-2-clumped permutations}\label{sec:2clumped}

Recall that the class of 2-clumped permutations is defined as 
$\mathsf{Av}(24\underline{51}3, 42\underline{51}3, 3\underline{51}24, 3\underline{51}42)$, and 
the class of co-2-clumped permutations is defined as 
$\mathsf{Av}(24\underline{15}3, 42\underline{15}3, 3\underline{15}24, 3\underline{15}42)$.
The following two theorems were proven by Reading in~\cite{R12}.

\begin{theorem}
  \label{thm:strong_distinguished}
Let $\rc$ be a weak rec\-tan\-gu\-la\-ti\-on, with its rectangles labeled by the NW--SE labeling. 
Then:
\begin{enumerate}
\item $\mathcal{L}(P_s(\rc))$ contains a unique 2-clumped permutation. 
It is $\pi_L$ --- the minimum (the ``leftmost'') element of $\mathcal{L}(P_s(\rc))$ with respect to the weak Bruhat order. 
\item $\mathcal{L}(P_s(\rc))$ contains a unique co-2-clumped permutation. 
It is $\pi_R$ --- the maximum (the ``rightmost'') element of $\mathcal{L}(P_s(\rc))$ with respect to the weak Bruhat order.
\end{enumerate}
\end{theorem}

\begin{theorem}
  \label{thm:strong_bijections}
The mapping $\gamma_s$ restricts to two bijections between strong rec\-tan\-gu\-la\-ti\-ons and permutation classes: 
\begin{enumerate}
\item Bijection $\beta_\mathsf{2C}$ between strong rec\-tan\-gu\-la\-ti\-ons to 2-clumped permutations;
\item Bijection $\beta_\mathsf{C2C}$ between strong rec\-tan\-gu\-la\-ti\-ons to co-2-clumped permutations;
\end{enumerate}
\end{theorem}

As in the weak case, given Proposition~\ref{prop:strongfiber}, the two items of Theorem~\ref{thm:strong_distinguished}
imply the corresponding items of Theorem~\ref{thm:strong_bijections}.
Also similarly to the weak case, one can easily obtain $\pi_L$ and  $\pi_R$ by 
repeated pruning the leftmost (respectively, the rightmost) leaf of the 
Hasse diagram of $P_s(\rc)$.
Figure~\ref{fig:s_permutations} shows $\pi_L$ --- the 2-clumped representative,
and $\pi_R$ --- the co-2-clumped representative of $P_s(\rc_1)$. 
\begin{figure}[!h]
\begin{center}
\includegraphics[scale=0.95]{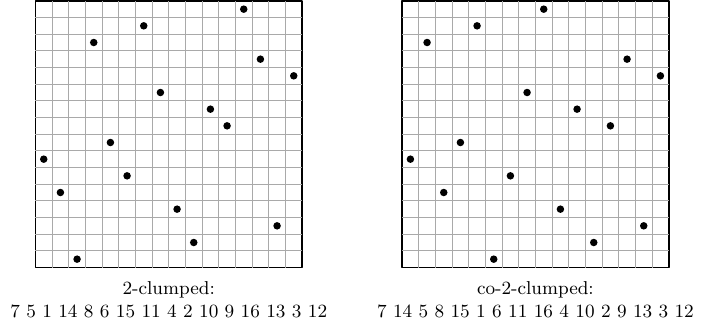} 
\end{center}
\caption{The 2-clumped and the co-2-clumped representatives of $P_s(\rc_1)$.}
\label{fig:s_permutations}
\end{figure}

Below, we provide an alternative proof of Theorems~\ref{thm:strong_distinguished}
and ~\ref{thm:strong_bijections}.
Our proof consists of a sequence of lemmas (\ref{lem:cross}, \ref{lem:2clumped}, \ref{lem:unique}), 
 and emphasizes the correspondence between patterns in rec\-tan\-gu\-la\-ti\-ons and patterns in permutations. 
In both cases we prove the part about 2-clumpled permutations;
the part about co-2-clumpled permutations then follows 
by symmetry (Observation~\ref{obs:mirror}) ---
which can also be seen from the characterization of the congruence classes 
associated to strong rec\-tan\-gu\-la\-ti\-ons in~\cite[Prop 2.2(2)]{R12}.

\begin{lemma}
  \label{lem:cross}
  Let $\pi_L$ be the leftmost linear extension of $P_s(\rc)$.
  Then the pattern $2\underline{41}3$ occurs in $\pi_L$ if and only if the pattern \zwall\ occurs in $\rc$.
\end{lemma}
\begin{proof}
  $(\Rightarrow)$ Suppose that the pattern appears in $\pi_L$ in the
  form $b\underline{da}c$, where $a<b<c<d$ and $d$ and $a$ appear
  consecutively. Since $\pi_L$ is the leftmost linear extension, we
  necessarily have $d\prec_s a$, otherwise~$a$ would occur in $\pi_L$
  earlier than $d$. Just after taking $r_d$, the rectangle $r_a$ is available:
  hence, there is a~segment~$s$ which contains 
  either the bottom side of $r_a$ and the top side of $r_d$ (a possible configuration is shown in Figure~\ref{fig:szwall}(a)),
  or the right side of $r_a$ and the left side of $r_d$,
  in which case the bottom-right corner of $r_a$ is higher then the top-left corner of $r_d$,
as shown in Figure~\ref{fig:szwall}(b). 

\begin{figure}[!h]
\begin{center}
  \includegraphics[scale=1, page=2]{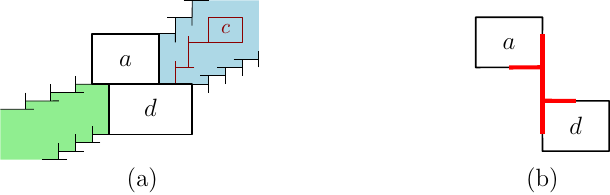}
  \caption{The two possibilities for a covering pair $d \prec_s a$ with $a < d$ in $\pi_L$.} 
  \label{fig:szwall}
\end{center}
\end{figure}

If $s$ is horizontal (case (a)),
then, by Observation~\ref{obs:lrab}(1) (refer also to Figure~\ref{fig:ablr_diag}),
all the rectangles $r_x$ with $a<x<d$ lie 
in the union of two regions (shown by green and blue in Figure~\ref{fig:szwall}(a)) delimited by the NE and SW alternating paths
of $r_a$ and $r_d$.
Since $\pi_L$ contains the pattern $bdac$,
the rectangle~$r_b$ lies in the green region, 
and the rectangle~$r_c$ in the blue region. 
However, considering the SW alternating path of such~$r_c$, 
we see that the entire green region is below $r_c$, 
and, hence, $c<b$, which is a contradiction.
Therefore, we necessarily have case (b), and this configuration contains the pattern \zwall.

$(\Leftarrow)$
Suppose that $\rc$ contains the pattern \zwall.
Denote by 
$s$ the vertical segment,
and by $r_a$ and $r_d$ two rectangles contributing to the pattern
as in Figure~\ref{fig:szwall} (b). 
Let $b$ be the lowest rectangle touching~$s$ from the left,
and $c$ the highest rectangle touching $s$ from the right, 
see Figure~\ref{fig:cross}(a).
We have $a<b<c<d$, and also $c=b+1$.
Then, in any linear extension $\pi$ of $P_s(\rc)$, we have the pattern $2413$ realized as $bdac$ with $c=b+1$.
It is well known that such a pattern implies an occurrence of $2\underline{41}3$:
to see that, note that~$\pi$ has two consecutive letters $d'a'$, with 
$d'$ weakly on the right of $d$, and $a'$ weakly on the left of $a$, 
such that $d'>c$ and $a'<b$, see Figure~\ref{fig:cross}(b).
Then $bd'a'c$ is an occurrence of $2\underline{41}3$.
\end{proof}
 
\noindent\textit{Remark.}
In the proof of $(\Leftarrow)$ we did not use the assumption that $\pi_L$ is the leftmost linear extension.
Therefore, in fact, a stronger result holds: If $\rc$ contains \zwall, then \textit{any} preimage of $\rc$ contains $2\underline{41}3$.
 
\begin{figure}[!h]
\begin{center}
  \includegraphics[scale=0.9]{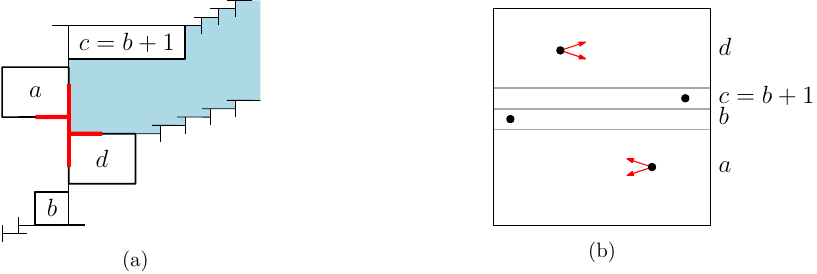}
  \caption{(a) The relative position of the rectangles $r_a, r_b, r_{b+1}, r_d$ in the proofs of Lemmas~\ref{lem:cross} and~\ref{lem:2clumped}. 
(b) An occurrence $bdac$ of $2{41}3$, where $c=b+1$, implies an occurrence of $2\underline{41}3$.}
  \label{fig:cross}
\end{center}
\end{figure}

\begin{lemma}
  \label{lem:2clumped}
  The leftmost linear extension $\pi_L$ of $P_s(\rc)$ is 2-clumped.
\end{lemma}

\begin{proof}
  Assume for the sake of contradiction that $\pi_L$ contains one of the four patterns 
  $24\underline{51}3$, $42\underline{51}3$, $3\underline{51}24$, $3\underline{51}42$
  forbidden in 2-clumped permutations.  
  Then, clearly, $\pi_L$ contains $2\underline{41}3$. 
  Then, similarly to the proof of $(\Rightarrow)$ in Lemma~\ref{lem:cross},  
  the pattern \zwall\ appears in~$\rc$, realized by four rectangles
  $r_a, r_b, r_c, r_d$ with labels $a<b<c<d$. 
  Using the argument symmetric to that shown in Figure~\ref{fig:cross}(b),
  we can assume that $c=b+1$, and the four rectangles are in the relative position
  as shown in Figure~\ref{fig:cross}(a).

  Consider the possible completions of this occurrence of
  $2\underline{41}3$ to one of the two patterns $24\underline{51}3$
  or $42\underline{51}3$. 
  It follows that there is a rectangle $r_x$ that yields an occurrence
  of $bx\underline{da}c$ or $xb\underline{da}c$ in $\pi_L$, with
  $a,b,c,d$ fixed as above and $c<x<d$.  
  Then, the condition $c<x<d$ implies that $r_x$ lies in the region (shown by blue)
  delimited by the rectangles $r_c$, $r_d$, the segment $s$, and the NE alternating paths of $r_c$ and $r_d$
  (note that the SW alternating path of $r_d$ is included in that of $r_c$).
However, $r_x$ is inserted earlier than $r_d$, and must lie
below the staircase obtained just before inserting $r_d$.
Since the blue region is above such a staircase,
it is not possible to place $r_x$ so that 
an occurrence of $42\underline{51}3$ or $24\underline{51}3$ will be created.
One can show with a~symmetric argument that the occurrence of the
pattern $2\underline{41}3$ can neither be completed to an occurrence
of $3\underline{51}42$ nor $3\underline{51}24$. Therefore,~$\pi_L$
must be a 2-clumped permutation. 
\end{proof}

In order to show that the fibers of $\gamma_s$, hence the strong rec\-tan\-gu\-la\-ti\-ons, 
are in bijection with 2-clumped permutations, 
we need to prove that the leftmost linear extension is the unique 2-clumped one.

\begin{figure}[!h]
\begin{center}
  \includegraphics[scale=.4, page=3]{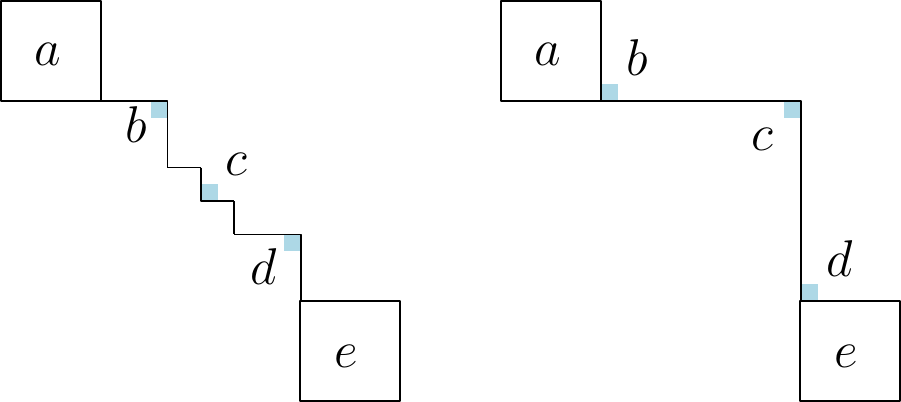}
  \caption{The two cases in the proof of Lemma~\ref{lem:unique}.}
  \label{fig:unique}
\end{center}
\end{figure}

\begin{lemma}
  \label{lem:unique}
  If a linear extension $\pi$ of $P_s(\rc)$ is not the leftmost linear extension, then it is not 2-clumped.
\end{lemma}
\begin{proof}
  Since $\pi\not=\pi_L$,
  there are two indices $i$ and $j$ with $i<j$ such that $\pi_i$ and $\pi_j$
  are both minima of the poset induced by $\pi_i,\ldots,\pi_n$ and $\pi_j<\pi_i$.
  By looking at the elements between $\pi_i$ and $\pi_j$ we
  find an index~${\ell}$ with $i\leq {\ell} <j$ such that $\pi_{{\ell}+1}<\pi_{\ell}$
  and such that $\pi_{\ell}$ and $\pi_{{\ell}+1}$ are both minima of the poset induced by $\pi_{\ell},\ldots,\pi_n$. 
  
Let $a=\pi_{{\ell}+1}$ and $e=\pi_{\ell}$. Note that $a$ and $e$ are incomparable in $P_s(\rc)$. 
  Consider the rectangles of $\rc$ with labels in the prefix
  $\pi_1\ldots \pi_{{\ell}-1}$ of $\pi$: 
  They form a staircase such that the rectangles $r_a$ and $r_e$ are in two valleys, the one for $r_a$ before the one for~$r_e$ along the staircase.
  We consider this staircase and distinguish two cases, see Figure~\ref{fig:unique}.

  First, if there is a valley between the valleys occupied by $r_a$ and $r_e$,
  then the bottom side of $r_a$ and the left side of $r_e$ belong to two
  segments forming two peaks belonging to two different rectangles $r_b$ and $r_d$
  with $a<b<d<e$. Consider the next rectangle $r_c$ to be inserted in a valley between
  those occupied by $r_a$ and $r_e$. Due to the NW--SE labeling, we have
  $b<c<d$, and in $\pi$ we first have $b$ and $d$ in any order, then
  consecutively $ea$, and finally $c$. This yields an occurrence of one
  of the two forbidden patterns $24\underline{51}3$ and $42\underline{51}3$.

  Now suppose that the rectangles $r_a$ and $r_e$ are inserted in consecutive valleys.
  Let $r_c$ be the rectangle forming the peak between the valleys of $r_a$ and
  $r_e$. We claim that neither $r_a$ nor $r_e$ extend to the top-right corner of $r_c$.
  Indeed, if $r_e$ extends to the top-right corner of $c$, then we have
  $a\prec_s e$ due to the special relations;  
  and similarly, if $r_a$ extends to the top-right corner of $r_c$, then $e\prec_s a$. Both are impossible since $a$ and $e$ are incomparable  in $P_s(\rc)$.   
  Hence, adding $r_a$
  and $r_e$ to the staircase makes two valleys, one on each side of $r_c$.
  Let $r_b$ and $r_d$ be the rectangles filling these valleys. From the order along the
  staircase we obtain $a<b<c<d<e$, while in $\pi$ we first have $c$ then
  consecutively $ea$, and finally $b$ and $d$ in any order. This
  yields one of the other forbidden patterns
  $3\underline{51}24$ and $3\underline{51}42$.
\end{proof}

Lemmas~\ref{lem:cross}, \ref{lem:2clumped} and \ref{lem:unique} together imply Theorems~\ref{thm:strong_distinguished}(1) and~\ref{thm:strong_bijections}(1).

\medskip

Given a permutation $\pi$ in $S_n$, 
its \emph{complement} $\bar{\pi}$ is
the permutation whose $i$th component is $\bar{\pi}_i = n+1-\pi_i$.
Note that the forbidden patterns of co-2-clumped permutations are the
complements of the forbidden patterns of 2-clumped permutations.

The following fact is a direct consequence of the 
symmetry of the forward algorithm.

\begin{observation}\label{obs:mirror}
The rec\-tan\-gu\-la\-ti\-on $\bar{\rc}=\gamma_s(\bar{\pi})$ is symmetric to $\rc=\gamma_s(\pi)$ with respect to the SW--NE diagonal.
\end{observation}

Since the rightmost linear extension of $P_s(\rc)$ is the complement of the leftmost linear extension of $P_s(\bar{\rc})$, it must forbid the complements of the forbidden patterns for the 2-clumped permutations.
Therefore, the rightmost linear extension $\pi_R$ of $P_s(\rc)$ is co-2-clumped, and if a linear extension $\pi$ of $P_s(\rc)$ is not the rightmost linear extension, then it is not co-2-clumped.
Hence co-2-clumped permutations are exactly the maxima, in the weak Bruhat order, 
of the intervals $\gamma_s^{-1}(\rc)$, and they are bijective to strong rec\-tan\-gu\-la\-ti\-ons as well.
This completes the proof of Theorem~\ref{thm:strong_distinguished} and~\ref{thm:strong_bijections}.

\medskip

\noindent\textit{Remark.}
Combining Theorem~\ref{thm:strong_bijections} with Lemma~\ref{lem:cross}, we observe that $\gamma_s$ specializes into a bijection 
from permutations that avoid $2\underline{41}3$ --- the \textit{semi-Baxter permutations} --- to rec\-tan\-gu\-la\-ti\-ons that avoid $\zwall$.
By duality~\cite[Sec.2.4]{FNS21}, rec\-tan\-gu\-la\-ti\-ons of size $n$ avoiding $\zwall$ are in bijection with plane bipolar posets with $n+2$ vertices, which are in a simple bijection~\cite[Sec.5]{FNS21} with permutations of size $n$ that avoid $2\underline{14}3$ (\textit{plane permutations}).
Plane permutations are shown in~\cite{BouvelGRR18} to be in bijection with semi-Baxter permutations, but the bijection is recursive (it proceeds via generating trees).
Via rec\-tan\-gu\-la\-ti\-ons avoiding $\zwall$ we have a more geometric bijection between these two permutation classes.

Moreover, by symmetry, $\gamma_s$ specializes into a bijection from permutations avoiding $3\underline{14}2$ to rec\-tan\-gu\-la\-ti\-ons avoiding $\iswall$.
So $\gamma_s$ specializes into a bijection from Baxter permutations to rec\-tan\-gu\-la\-ti\-ons avoiding $\zwall$ and $\iswall$, which identify with weak rec\-tan\-gu\-la\-ti\-ons  (and are realized by the anti-diagonal representation), and we recover the bijection from~\cite{ABP06} between Baxter permutations and weak rec\-tan\-gu\-la\-ti\-ons.

\subsection{The flip graph on strong rectangulations}\label{sec:flip}

We briefly recall the notion of lattice congruence, and refer to Reading~\cite{R04} for a specific treatment of congruences of the weak Bruhat order.
An equivalence relation $\equiv$ on the set of elements of a lattice $(L,\wedge,\vee)$ is said to be a~\emph{lattice congruence} if it behaves consistently with respects to joins and meets, hence if $x\equiv x'$ and $y\equiv y'$, then $x\wedge y\equiv x'\wedge y'$, and $x\vee y\equiv x'\vee y'$.
In that case, one can define the quotient of the lattice on the congruence classes, such that the order as well as the meet and the join of two classes is defined respectively by the order, the meet, and the join in $L$ of any two representatives of the classes. The lattice quotient can also be shown to be isomorphic to the lattice induced in $L$ by the minimal element of each congruence class.
It is known that the equivalence classes of permutations defined by the fibers of $\gamma_s$ form a lattice congruence.

\begin{theorem}[Reading~\cite{R12}]\label{thm:reading}
  Consider the equivalence relation $\equiv$ on $S_n$ defined by
  \[
  \pi\equiv\sigma\Leftrightarrow \gamma_s(\pi)=\gamma_s(\sigma).
  \]
  Then $\equiv$ is a lattice congruence of the weak Bruhat order $\prec$ on $S_n$.
\end{theorem}

In particular, the partial order induced by the weak Bruhat order on these equivalence classes is a~lattice.
The cover graph of this lattice is a graph with vertex set $\mathsf{SR}_n$.
Meehan~\cite{M19a} described the edges of this graph as local operations on the rec\-tan\-gu\-la\-ti\-ons, so that two rec\-tan\-gu\-la\-ti\-ons are adjacent
if and only if they differ by such an operation. These operations are called \emph{flips}, and the cover graph of the lattice on
$\mathsf{SR}_n$ is called the \emph{flip graph} on $\mathsf{SR}_n$. This is in perfect analogy with the well-known flip graph on triangulations of a convex polygon
defined by the Sylvester congruence~\cite{T51}, and the flip graph on diagonal rec\-tan\-gu\-la\-ti\-ons defined by the Baxter congruence~\cite{LR12}.
These flip graphs happen to be skeletons of polytopes: Flip graphs on triangulations, for instance, are skeletons of associahedra~\cite{STT88,P14}.
A~remarkable result by Pilaud and Santos allow us to make the same statement for the flip graph on strong rec\-tan\-gu\-la\-ti\-ons.

\begin{theorem}[Pilaud and Santos~\cite{PS19}]
  For any lattice congruence $\equiv$ of the weak Bruhat order on $S_n$,
the cover graph of the quotient of the weak Bruhat order by $\equiv$ is the skeleton of a polytope.
\end{theorem}

Our algorithm describing the mapping $\gamma_s$ allows us to identify the flip operations defining the graph of this strong rec\-tan\-gu\-la\-ti\-on polytope.

\begin{figure}[!h]
\begin{center}
  \includegraphics[scale=.4, page=2]{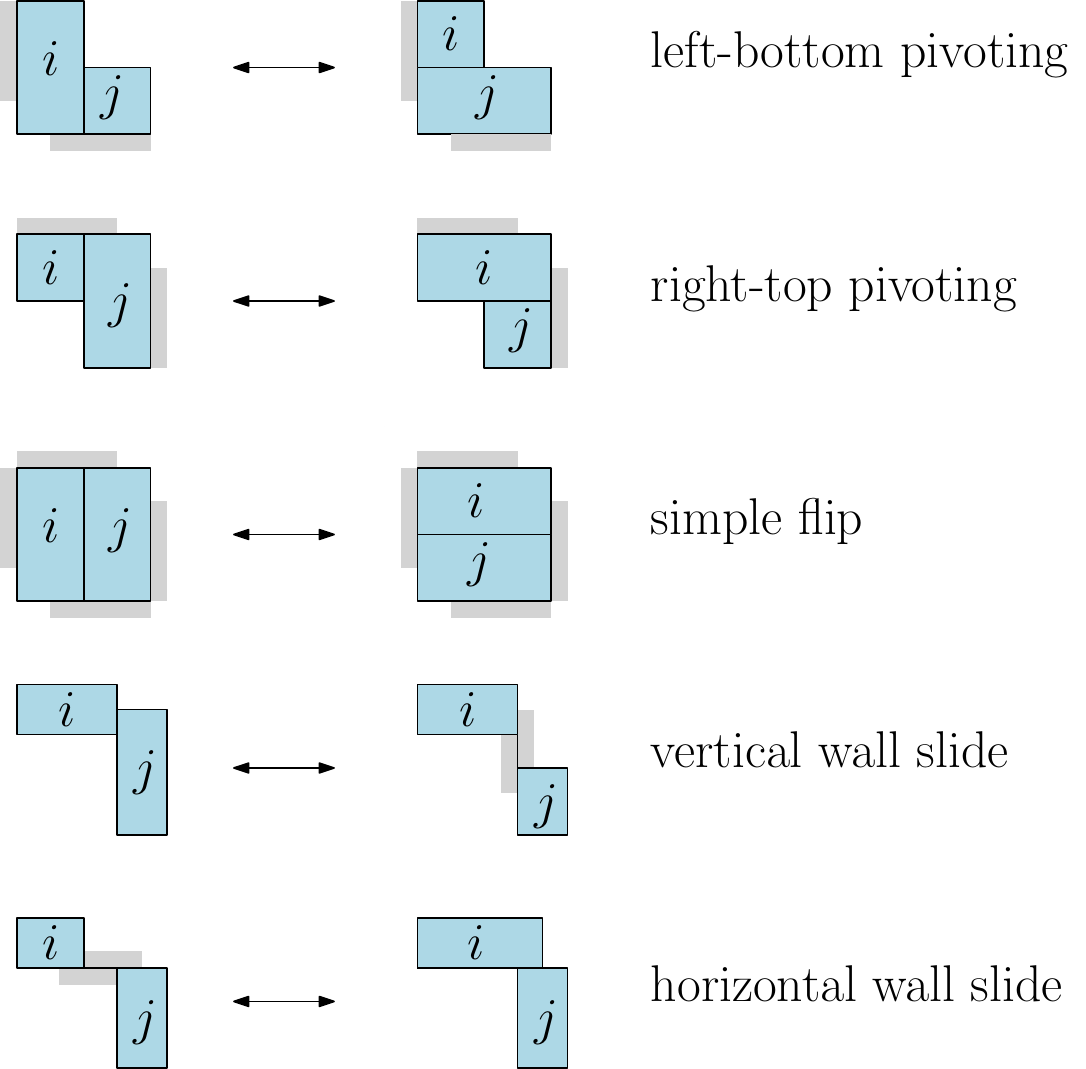}
  \caption{Flips in strong rec\-tan\-gu\-la\-ti\-ons that correspond to cover relation in the lattice of strong rec\-tan\-gu\-la\-ti\-ons. These are obtained by considering the changes that occur in the rec\-tan\-gu\-la\-ti\-ons when the forward algorithm is applied on two permutations that differ by the adjacent transposition of $i$ and $j$. In all five situations, the shaded regions cannot intersect any edge of the rec\-tan\-gu\-la\-ti\-on, by the definition of the mapping $\gamma_s$.}
  \label{fig:flips}
\end{center}
\end{figure}

\begin{theorem}
  The flip graph on the set of strong rec\-tan\-gu\-la\-ti\-ons $\mathsf{SR}_n$ is described by the flip operations of Figure~\ref{fig:flips}.
\end{theorem}
\begin{proof}
  Observe that if we consider two permutations $\pi$ and $\pi'$ differing by one adjacent transposition, then either $\pi\equiv\pi'$ or, by definition of a lattice congruence, $\gamma(\pi)$ and $\gamma(\pi')$ are in a cover relation in the lattice quotient, hence $\gamma(\pi)$ and $\gamma(\pi')$ differ by a flip.
  It therefore suffices to inspect all changes in the rec\-tan\-gu\-la\-ti\-on $\gamma(\pi)$ that can occur when two adjacent entries of $\pi$ are transposed.
  
  Recall that the forward algorithm reads the input permutation $\pi=\pi_1\pi_2\ldots \pi_n$ from
  left to right, and at each step $t$, inserts the rectangle $r_{\pi_t}$, with label $\pi_t$. Consider two successive steps of the algorithm, involving $\pi_t=i$ and $\pi_{t+1}=j$. Suppose, without loss of generality, that $i<j$. There are five possible ways that the rectangles change placement after the transposition of $i$ and $j$, which are illustrated in Figure~\ref{fig:flips}. Note that in this figure, the grey regions cannot intersect any edge of the rec\-tan\-gu\-la\-ti\-on. This follows from the way that the rectangles $r_i$ and $r_j$, as well as any rectangle processed later, are inserted by the forward algorithm. 
  In all five cases, the transformation in $\rc$ is of one of three types of flips: pivoting flips, simple flips, and wall slides.

  Conversely, if such an operation is possible in a rec\-tan\-gu\-la\-ti\-on $\rc$, then there is an execution of the forward algorithm such that the two rectangles involved are inserted consecutively. One can, for instance, consider the downset $S_{i,j}\subset [n]$ of labels $\ell$ such that either $\ell\prec_s i$, or $\ell\prec_s j$, and consider any linear extension of $(S_{i,j}, \prec_s)$. By Proposition~\ref{prop:strongfiber}, running the forward algorithm on this prefix of a permutation leads to a situation in which we can insert either $r_i$ or $r_j$, and the flip can be implemented.   
\end{proof}

\subsection{Quadrant walk encoding and enumeration}

From the definition of the forward algorithm, we can now establish bijections between families of strong rec\-tan\-gu\-la\-ti\-ons and families of quadrant walks (we also discuss
how the method adapts in the weak case).

For a point $p=(x,y)$ in the quadrant $\mathbb{N}^2$, the \emph{level} of $p$ is $h(p):=x+y$. 
We define a~\emph{history quadrant walk} as a sequence of points $(x,y)$ in the quadrant $\mathbb{N}^2$, each point having a color in $\{$black, red, green, white$\}$, such that for any two consecutive points $p,p'$ of the sequence:
\begin{itemize}
\item
if $p$ is black, then $h(p')=h(p)+1$,
\item
if $p$ is red or green, then $h(p')=h(p)$,
\item
if $p$ is white, then $h(p')=h(p)-1$.
\end{itemize}
Such a walk is called \emph{closed} if the final point is at the origin and is white; it is called an \emph{excursion} if it is closed and starts at the origin.  

\begin{figure}[!h]
\begin{center}
\includegraphics[width=0.8\linewidth]{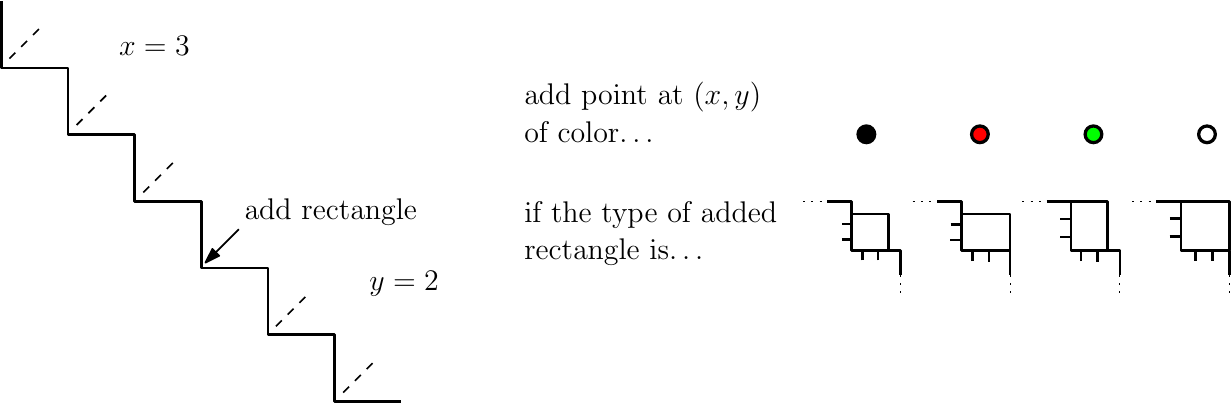}
\end{center}
\caption{Rule for inserting a colored point in the quadrant, corresponding to inserting a rectangle by the forward algorithm.}
\label{fig:add_rectangle_strong}
\end{figure}

\begin{figure}[!h]
\begin{center}
\includegraphics[width=\linewidth]{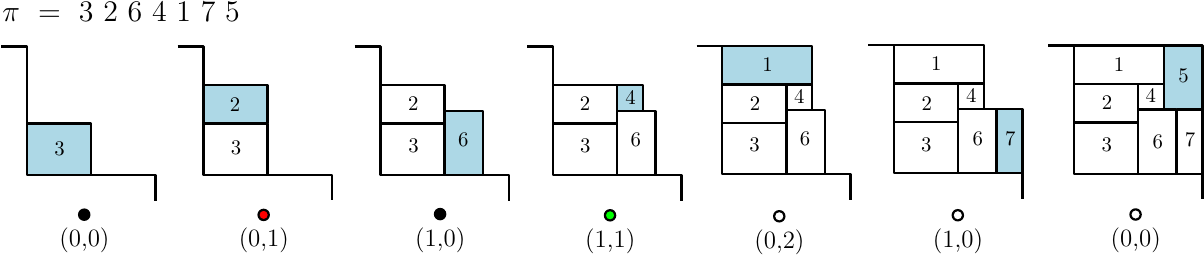}
\end{center}
\caption{A permutation $\pi$, and the associated quadrant history $\sigma$, which is built jointly with the rec\-tan\-gu\-la\-ti\-on $\rc$ produced by the forward algorithm (note that $\pi$ is not needed to build $\rc$ from~$\sigma$).}
\label{fig:example_history}
\end{figure}

For $\pi$ a permutation of size $n$, with $\rc=\gamma_s(\pi)$ the rec\-tan\-gu\-la\-ti\-on produced from $\pi$ by the forward algorithm, the corresponding \emph{quadrant history} is the history quadrant excursion (with $n$ points) where each rectangle addition yields a point as shown in Figure~\ref{fig:add_rectangle_strong}.
See also Figure~\ref{fig:example_history} for a complete example.  

\medskip

\noindent\textit{Remark.}
A~\emph{bicolored Motzkin excursion} is a Motzkin excursion (walk with steps in $\{(1,1),(1,0),(1,-1)\}$, starting at the origin, staying in $\{y\geq 0\}$, and ending on the line $\{y=0\}$)  where each horizontal step is colored either red or green, it is \emph{decorated} if 
each point of the excursion is assigned an integer $x$ between~$0$ and its height. 
For $\pi$ a permutation of size $n$, the quadrant history $\sigma$ of $(\pi,\rc=\gamma(\pi))$ can be encoded by a decorated bicolored Motzkin excursion (of length $n-1$), 
where the successive heights in the Motzkin excursion are given by the sequence of levels of points in $\sigma$, the horizontal steps are colored as the initial point of the corresponding step in $\sigma$, and the assigned integers are given by the abscissas of points in $\sigma$.  
One can check that this decorated bicolored Motzkin excursion is the one associated to the 
permutation~$\pi^{-1}$ by the Fran\c{c}on-Viennot bijection~\cite{FV79}.

\medskip
\medskip

A history quadrant walk is called \emph{leftmost} if, for any two consecutive points $p,p'$:
\begin{itemize}
\item
if the color of $p$ is in $\{$black,red$\}$ and the color of $p'$ is in $\{$black,green$\}$, then $x(p')\geq x(p$),
\item
otherwise, $x(p')\geq x(p)-1$. 
\end{itemize}

\begin{figure}[!h]
\begin{center}
\includegraphics[width=\linewidth]{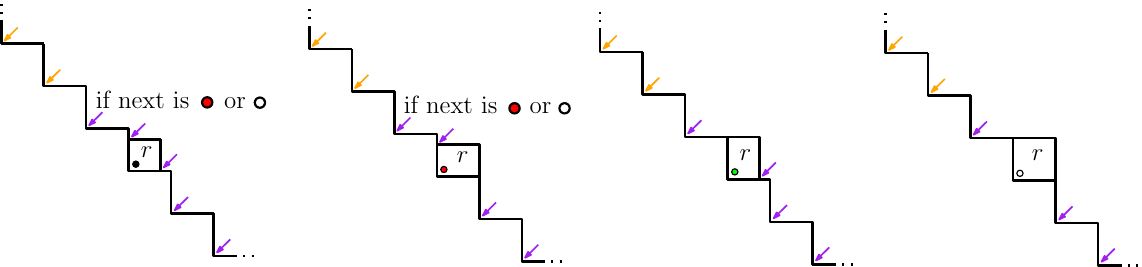}
\end{center}
\caption{If the last inserted rectangle $r$ is the current rightmost available rectangle, the figure indicates for each valley whether the insertion of a rectangle $r'$ at that valley makes $r'$ the new rightmost available rectangle (purple) or not (orange). As shown, when $r$ is of black or red type, there is a mixed valley to the left of $r$, where $r'$ is allowed to be inserted only if it is red or white (indeed, in that case, $r$ is not available anymore after inserting $r'$).}
\label{fig:leftmost_staircase}
\end{figure}

As shown in~\cite{ITF09} (in different but equivalent terms) and illustrated in Figure~\ref{fig:leftmost_staircase}, the quadrant history of a pair $(\pi,\rc=\gamma(\pi))$
is leftmost if and only if $\pi$ is the leftmost linear extension of $P_s(\rc)$, that is, at any step, the last added rectangle is the rightmost available rectangle. We therefore obtain the following.

\begin{proposition}
  \label{prop:qw}
  Leftmost history quadrant excursions of length $n-1$ (hence having $n$ points) are in bijection with rec\-tan\-gu\-la\-ti\-ons of size $n$, and with 2-clumped permutations of size $n$. 
  \end{proposition}

The above characterization can be turned  into a recurrence for counting these walks, and gives an efficient procedure for counting rec\-tan\-gu\-la\-ti\-ons~\cite{ITF09} 
(other polynomial-time counting methods have been given respectively in~\cite{conant2014number} via inclusion-exclusion, and in~\cite{FNS21} by a different quadrant walk encoding, via some decorated  plane bipolar orientations). The sequence starts with $1, 2, 6, 24, 116, 642, 3938, 26194, 186042, 1395008,\ldots$ (\href{https://oeis.org/A342141}{OEIS A342141}).  
As shown in~\cite{FNS21} (and in~\cite{DBLP:conf/iscas/FIT09,TFI09} for the upper bound), 
its exponential growth rate is $27/2$. 

Symmetrically, a history quadrant walk is called \emph{rightmost} if, for any two consecutive points $p,p'$:
\begin{itemize}
\item
if the color of $p$ is in $\{$black,green$\}$ and the color of $p'$ is in $\{$black,red$\}$, then $y(p')\geq y(p$),
\item
otherwise, $y(p')\geq y(p)-1$. 
\end{itemize}
These correspond to pairs $(\pi,\rc=\gamma(\pi))$ such that $\pi$ is the rightmost linear extension of $P_s(\rc)$ 
(at any step, the last added rectangle is the leftmost available one), which occurs if and only if $\pi$ is co-2-clumped. 

\begin{figure}[!h]
\begin{center}
\includegraphics[width=0.8\linewidth]{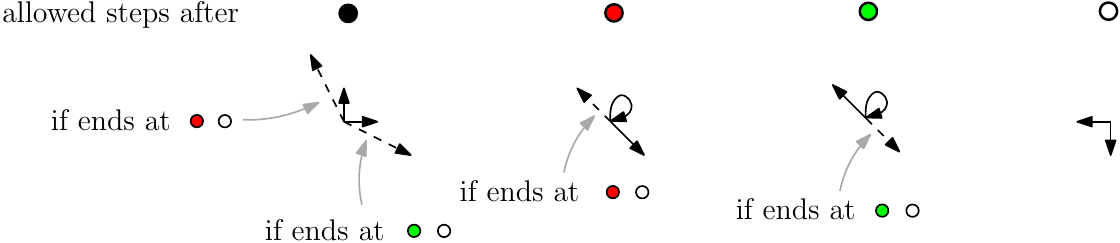}
\end{center}
\caption{The allowed steps in leftright history walks (special steps are shown dashed).}
\label{fig:leftright}
\end{figure}

A history quadrant walk is called \emph{leftright} if it is both leftmost and rightmost. Equivalently, it is a~history walk (here better formulated 
in terms of allowed steps) such that,  for any two consecutive points $p,p'$ (see Figure~\ref{fig:leftright}):
\begin{itemize}
\item
If $p$ is black, then from $p$ to $p'$ the steps $(0,1)$ and $(1,0)$ are allowed. Furthermore, if the color of $p'$ is in $\{$red,white$\}$ then
the step $(-1,2)$ is allowed, and if the color of $p'$ is in $\{$green,white$\}$ then
the step $(2,-1)$ is allowed, such steps being called special. 
\item
If $p$ is red, then from $p$ to $p'$ the steps $(0,0)$ and $(1,-1)$ are allowed, and furthermore the step $(-1,1)$, called a special step, is 
allowed if the color of  $p'$ is in $\{$red,white$\}$. 
\item
If $p$ is green, then from $p$ to $p'$ the steps $(0,0)$ and $(-1,1)$ are allowed, and furthermore the step $(1,-1)$, called a special step, is 
allowed if the color of  $p'$ is in $\{$green,white$\}$.
\item
If $p$ is white, then from $p$ to $p'$ the allowed steps are $(-1,0)$ and $(0,-1)$.
\end{itemize}

Leftright history quadrant excursions thus correspond to rec\-tan\-gu\-la\-ti\-ons $\rc$ such that $P_s(\rc)$ is a total order, i.e., the fiber has size $1$ (indeed, at any step, the last
added rectangle is both the leftmost and rightmost available rectangle, hence is the unique available rectangle), a superfamily of rec\-tan\-gu\-la\-ti\-ons avoiding $\zwall$ and $\iswall$ 
(those represented by anti-diagonal rec\-tan\-gu\-la\-ti\-ons). 
These also correspond to permutations that are 2-clumped and co-2-clumped (and to equivalence classes of size $1$ for the congruence in Theorem~\ref{thm:reading}), 
a superfamily of Baxter permutations. 

The above characterization of leftright walks can be turned into a recurrence as follows. For $\cG$ a set of history quadrant walks, and for $n\geq 1,i,j\geq 0$, 
we let $G_{n,i,j}$ be the number of closed walks of length $n$ in $\cG$ and starting at $(i,j)$. Then, with $\cA$ the set of leftright history quadrant walks, and with 
$\cB$ (resp. $\cR,\cG,\cW$) the subset of those starting at a black (resp. red,green,white) point, 
a classical decomposition by first-step removal yields, for $n\geq 1$ and $i,j\geq 0$,

\begin{equation}\label{eq:Un_rec}
\left\{
\begin{array}{rcl}
B_{n,i,j}&=& A_{\nm;i+1,j} + A_{\nm,i,j+1} + R_{\nm,i-1,j+2}+ W_{\nm,i-1,j+2}\\
&&+\ G_{\nm,i+2,j-1}+ W_{\nm,i+2,j-1}\\
R_{n,i,j}&=& A_{\nm,i+1,j-1}+A_{\nm,i,j}+R_{\nm,i-1,j+1}+W_{\nm,i-1,j+1},\\
G_{n,i,j}&=& A_{\nm,i-1,j+1}+A_{\nm,i,j}+G_{\nm,i+1,j-1}+W_{\nm,i+1,j-1},\\
W_{n,i,j}&=& A_{\nm,i-1,j} + A_{\nm,i,j-1},\\
A_{n,i,j}&=&B_{n,i,j}+R_{n,i,j}+G_{n,i,j}+W_{n,i,j},\\
\end{array}
\right.
\end{equation} 
with boundary conditions $L_{n;i,j}=0$ for $n\leq 0$ or $i<0$ or $j<0$ (for $L\in\{A,B,R,G,W\}$), except for $W_{0,0,0}=A_{0,0,0}=1$. 
 
Note that, by $\{x,y\}$-symmetry of the walk specification, we have $R_{n,i,j}=G_{n,j,i}$ (and the coefficients 
$A_{n,i,j},B_{n,i,j},W_{n,i,j}$ are symmetric in $i$ and $j$).
The sequence $U_n=A_{n-1,0,0}$ gives the number of permutations of size $n$ that are 2-clumped and co-2-clumped\footnote{With the strong poset characterization it is not difficult to show that it also counts weak rectangulations of size $n$ where every 2-sided segment (segment with at least one neighbor on each side) is given weight 2.}, it starts with 
$1, 2, 6, 24, 112, 582, 3272, 19550, 122628, 800392,\ldots$,
and, to our knowledge, it has not been considered before. 

\begin{proposition}\label{prop:upper_bound_Un}
The exponential growth rate of $U_n$ is 
bounded from above
by $\Gamma:=\frac1{2}(9+\sqrt{113})\approx 9.815$. 
\end{proposition}
\begin{proof}
Let 
\[
\mathbb{A}=\left(
\begin{matrix}
2&3&3&4\\
2&3&2&3\\
2&2&3&3\\
2&2&2&2
\end{matrix}
\right)
\]
and let $\mathbb{I}=(1,1,1,1)$. Then obviously the number of leftright walks of length $n$ (starting at the origin) with no constraint on domain nor on endpoint  
 is equal to $\mathbb{I}\cdot \mathbb{A}^n\cdot \mathbb{I}^T$; and $\Gamma$ is the spectral radius of $\mathbb{A}$. 
 \end{proof}
 
\noindent\textit{Remark.}
From the table of initial coefficients, 
the ratio $U_n/U_{n-1}$ seems to converge to $\Gamma$ (this is even more visible when applying acceleration of convergence techniques, see e.g.~\cite[Sec.6]{guitter1999hamiltonian}).  By similar calculations as~\cite[Conjecture~25]{FNS23} (details omitted), letting $\xi=(-93+9\sqrt{113})/4$, one can conjecture (up to a plausible extension of~\cite{denisov2015random}) the asymptotic estimate $U_n\sim c\, \Gamma^n n^{-\alpha}$, with $c>0$ and $\alpha=1+\pi/\arccos(\xi)\approx 4.742$. By a criterion in~\cite{bostan2014non} (ensuring that $\alpha\notin\mathbb{Q}$), this would imply that the generating 
function of $U_n$ is not D-finite.

\bigskip

We now discuss the specialization to anti-diagonal rec\-tan\-gu\-la\-ti\-ons and Baxter permutations. We refer here to an \emph{anti-diagonal rec\-tan\-gu\-la\-ti\-on} as a rec\-tan\-gu\-la\-ti\-on avoiding 
the patterns $\zwall$ and $\iswall$. Each weak class of rec\-tan\-gu\-la\-ti\-ons has a unique such representative, we see them here as a subclass of strong rec\-tan\-gu\-la\-ti\-ons and  
do not insist on considering the specific anti-diagonal representation on the $n\times n$ grid. Any anti-diagonal rec\-tan\-gu\-la\-ti\-on has fiber of size $1$, so that the corresponding history quadrant excursion is leftright. We also recall from the remark 
at the end of Section~\ref{sec:2clumped} 
that the mapping $\gamma_s$ specializes into a bijection between Baxter permutations and anti-diagonal rec\-tan\-gu\-la\-ti\-ons.

\begin{proposition}
The history quadrant excursions of length $n-1$ that encode Baxter permutations and anti-diagonal rec\-tan\-gu\-la\-ti\-ons of size $n$ are in bijection with the set $\mathsf{NIT}_n$ of non-intersecting triples of lattice walks (with steps up or right), starting respectively at $(-1,1),(0,0),(1,-1)$, and 
ending at $(n-k-1,k),(n-k,k-1),(n-k+1,k-2)$ for some $1\leq k\leq n$. 
\end{proposition}
\begin{proof}
Let $\sigma$ be a history quadrant excursion, with $\rc$ the rec\-tan\-gu\-la\-ti\-on built from $\sigma$. The following properties are easy to check:
\begin{itemize}
\item
If $\sigma$ is leftmost, then each occurrence of $\zwall$ in $\rc$ 
corresponds to a transition from a black or red point $p=(x,y)$ to a red or white point $p'=(x',y')$ such that $x'=x-1$ (this corresponds to an insertion in a mixed valley in Figure~\ref{fig:leftmost_staircase}).  
\item
Symmetrically, if $\sigma$ is rightmost, then each occurrence 
of $\iswall$ in $\rc$ corresponds to a transition from a black or green point $p=(x,y)$ to a green or white point $p'=(x',y')$ such that $y'=y-1$. 
\end{itemize}  
Hence, if $\sigma$ is leftright, each occurrence of $\zwall$ in $\rc$ 
corresponds to an occurrence of a special step $(-1,2)$ or $(-1,1)$, while each occurrence of $\iswall$ in $R$  corresponds to an occurrence of a special step $(2,-1)$ or $(1,-1)$, so that $\rc$ is anti-diagonal if and only if $\sigma$ has no special step. 

Note that a leftright quadrant excursion $\sigma$ with no special step identifies to a quadrant walk of same length and with no colors on points, starting and ending at the origin, whose step-set is $\{2\times (0,0), (0,1),(0,-1),(1,0),(-1,0),(-1,1),(1,-1)\}$, with two kinds of stay-steps to account for the color of the initial point of each stay-step in $\sigma$. 
There is a simple bijection~\cite[Prop.20]{BurrillCFMM16} between such walks of length $n-1$ and $\mathsf{NIT}_n$. 
\end{proof}
Thus, we recover --- via rec\-tan\-gu\-la\-ti\-ons --- the fact that the Fran\c{c}on-Viennot encoding specialized to Baxter permutations yields a bijection
with non-intersecting triples of lattice walks~\cite{Vi81}. By the Gessel-Viennot Lemma, these are counted by the Baxter numbers $B_n$ (whose exponential 
growth rate is $8$). 

\bigskip

\begin{figure}[!h]
\begin{center}
\includegraphics[width=0.8\linewidth]{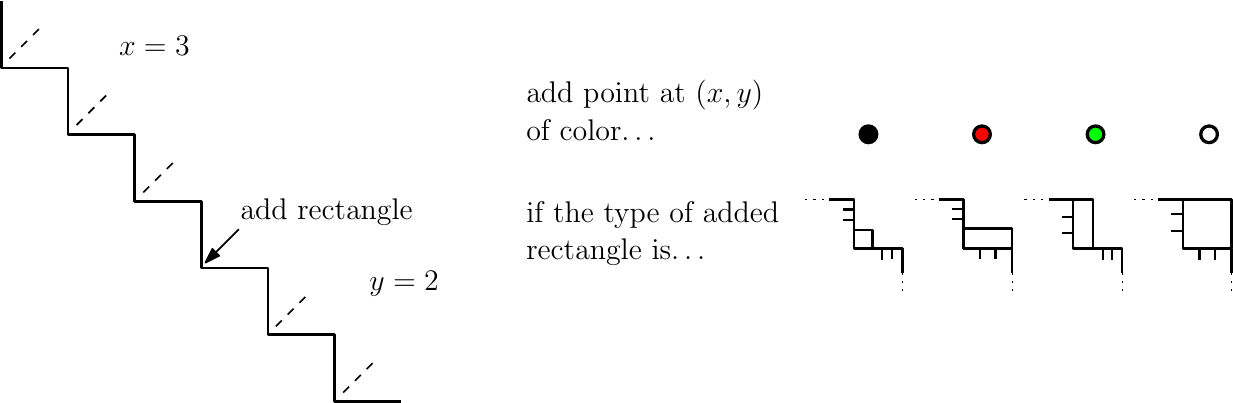}
\end{center}
\caption{Correspondence between the insertion of a colored point in the quadrant and the insertion of a rectangle for weak rec\-tan\-gu\-la\-ti\-ons}
\label{fig:add_rectangle_weak}
\end{figure}

To conclude the section, we briefly explain that a very similar study can be performed in the context of weak rec\-tan\-gu\-la\-ti\-ons.
A weak rec\-tan\-gu\-la\-ti\-on $\rc$ endowed with a linear extension of its weak poset $P_w(\rc)$
can again be bijectively encoded by a history quadrant excursion, 
where the addition of a rectangle is now done in the ``innermost'' way in a valley for the situations without alignment of sides,  
see Figure~\ref{fig:add_rectangle_weak} compared to Figure~\ref{fig:add_rectangle_strong}. Using the innermost convention yields the weak rec\-tan\-gu\-la\-ti\-on in the 
form of its strong representative with no $\izwall$ nor $\swall$, the one for which the diagonal representation exists. 

For the backward direction, in the current staircase shape, a rectangle is \emph{available} if and only if all its adjacent rectangles are to its left or below. 
It is then easy to characterize the leftmost (resp. rightmost) history quadrant excursions in this context, i.e., those corresponding to weak rec\-tan\-gu\-la\-ti\-ons endowed with the 
leftmost (resp. rightmost) linear extension of their weak poset, equivalently at each step the last added rectangle is the rightmost (resp. leftmost) available one.
Quite nicely, these have the same specification as the leftmost (resp. rightmost) walks in the strong case, upon replacing ``and'' by ``or'' in the first item.
These quadrant walks of length $n-1$ also encode twisted (resp. co-twisted) Baxter permutations of size $n$. They are thus counted by $B_n$, even if a direct bijection to $\mathsf{NIT}_n$ does not seem easy to find.

As in the strong case, we can then consider the history quadrant walks that are leftmost and rightmost, called \emph{leftright}. 
Leftright history quadrant excursions encode weak rec\-tan\-gu\-la\-ti\-ons whose weak poset is totally ordered.  
These are known to be the \emph{one-sided} rec\-tan\-gu\-la\-ti\-ons, i.e., weak rec\-tan\-gu\-la\-ti\-ons such that for each segment at least one side has no contact, which are also the weak rec\-tan\-gu\-la\-ti\-ons with a unique strong representative. And they correspond via $\gamma_w$ to the  permutations in  $\mathsf{Av}(2\underline{41}3,2\underline{14}3,3\underline{41}2,3\underline{14}2)$, i.e.,
twisted and co-twisted Baxter permutations.
By intersecting the step-sets for leftmost and rightmost walks, the specification of the step-set for leftright walks is as shown in Figure~\ref{fig:leftright_weak}. 

\begin{figure}[!h]
\begin{center}
\includegraphics[width=0.85\linewidth]{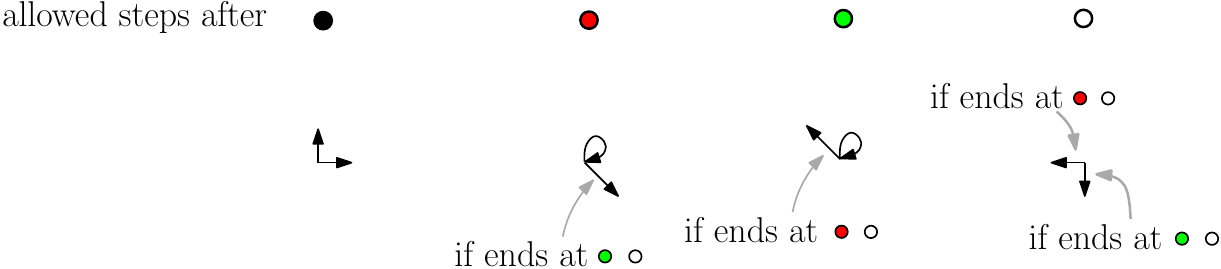}
\end{center}
\caption{The allowed steps in leftright history quadrant walks, in the context of weak rec\-tan\-gu\-la\-ti\-ons.}
\label{fig:leftright_weak}
\end{figure}

Letting $O_n$ be the number of one-sided rec\-tan\-gu\-la\-ti\-ons of size $n$, and number of leftright history quadrant excursions of length $n-1$, a recurrence for $O_n$ analogue to the recurrence~\eqref{eq:Un_rec} for $U_n$ can then be obtained, by first-step removal in closed leftright history quadrant walks. 
Another counting method for $O_n$, also in polynomial time, has been given in~\cite{bouvel19} (pages 162-175) by describing a generating tree for the permutation class, 
the sequence starts with $1, 2, 6, 20, 72, 274, 1088, 4470, 18884, 81652,\ldots$ and is \href{https://oeis.org/A348351}{OEIS A348351}.

\begin{proposition}\label{prop:up_bound_one_sided}
The exponential growth rate of $O_n$ is 
bounded from above
by $\Gamma':=\frac1{2}(7+\sqrt{17})\approx 5.562$. 
\end{proposition}
\begin{proof}
Let 
\[
\mathbb{A}'=\left(
\begin{matrix}
2&2&2&2\\
1&1&2&2\\
1&2&1&2\\
0&1&1&2
\end{matrix}
\right)
\]
and let $\mathbb{I}=(1,1,1,1)$. The number of leftright walks of length $n$ (starting at the origin) with no constraint on domain nor on endpoint  
 is equal to $\mathbb{I}\cdot \mathbb{A}'\,^n\cdot \mathbb{I}^T$; and $\Gamma'$ is the spectral radius of $\mathbb{A}'$.
\end{proof}

Again, by similar calculations as~\cite[Conjecture~25]{FNS23}, letting $\xi'=(-29+7\sqrt{17})/4$, one can conjecture the asymptotic estimate $O_n\sim c'\, \Gamma'\,^n n^{-\alpha'}$, with $c'>0$ and $\alpha'=1+\pi/\arccos(\xi')\approx 2.957$, which would imply that the generating function of $O_n$ is not D-finite.

An interesting consequence of Proposition~\ref{prop:up_bound_one_sided} is that the growth rate of one-sided rec\-tan\-gu\-la\-ti\-ons, which are also the rec\-tan\-gu\-la\-ti\-ons that are \emph{area-universal}~\cite{EppsteinMSV12}, is smaller than the known~\cite{Tu62,Fusy09}  growth rate $27/4=6.75$ of  triangulations of the 4-gon that are irreducible (no separating triangle).  Thus the irreducible triangulations of the 4-gon admitting a dual representation as an area-universal rec\-tan\-gu\-la\-ti\-on are exponentially rare, their growth rate being at most  $\Gamma'$. 

\section{Guillotine rectangulations}
\label{sec:guillotine}

In this section we deal with guillotine rec\-tan\-gu\-la\-ti\-ons, introduced in Section~\ref{sec:guil}.
While weak guillotine rec\-tan\-gu\-la\-ti\-ons are well understood
(see Propositions~\ref{prop:guillotinechar} and~\ref{prop:guillotineenum}),
we are not aware of any results concerning strong guillotine rec\-tan\-gu\-la\-ti\-ons.
In this section we provide a uniform treatment of guillotine rec\-tan\-gu\-la\-ti\-ons
by characterizing those permutations that correspond to guillotine partitions under both permutation-to-rec\-tan\-gu\-la\-ti\-on mappings $\gamma_w$ and $\gamma_s$, by means of pattern avoidance.
As a result, we can restrict all the bijections between (both weak and strong)
rec\-tan\-gu\-la\-ti\-ons that were mentioned above, to the guillotine case. In particular, we find
a permutation class bijective to \emph{strong guillotine rec\-tan\-gu\-la\-ti\-ons}. 

\subsection{Characterization by mesh patterns}\label{sec:mesh}

Consider the following mesh patterns\footnote{\ These mesh patterns were 
proposed by Merino and Mütze~\cite{MMpc}, see remark after Corollary~\ref{prop:guil_cor_S}.} (depicted in Figure~\ref{fig:mesh}):
\[
\begin{array}{l}
p_1 = (25314, \ \ \{(0,3), \, (0,4), \, (1,3), \, (4,2), \, (5,1), \, (5,2)\}), \\ 
p_2 = (41352, \ \ \{(0,1), \, (0,2), \, (1,2), \, (4,3), \, (5,3), \, (5,4)\}).
\end{array}
\]

\begin{figure}[!h]
\begin{center}
\includegraphics[scale=0.9]{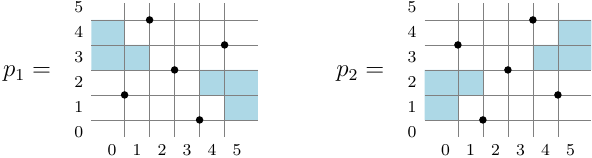}
\end{center}
\caption{Two mesh patterns whose avoidance characterizes guillotine permutations.}
\label{fig:mesh}
\end{figure}

\begin{theorem}\label{thm:guil_main} 
Let $\pi \in S_n$. Then the following conditions are equivalent:
\begin{enumerate}
\item[(1)] The weak rec\-tan\-gu\-la\-ti\-on $\gamma_w(\pi)$ is guillotine,
\item[(2)] The strong rec\-tan\-gu\-la\-ti\-on $\gamma_s(\pi)$ is guillotine,
\item[(3)] $\pi$ avoids both mesh patterns $p_1$ and $p_2$.
\end{enumerate}
\end{theorem}

The equivalence of (1) and (2) is clear, since being guillotine is invariant under shuffling.    
Hence, it suffices to prove the equivalence of (1) and (3). 
Recall from Proposition~\ref{prop:guillotinechar} that a rec\-tan\-gu\-la\-ti\-ons is guillotine if and only if
it avoids two ``windmills'' \wma\ and \wmb.
Theorem~\ref{thm:guil_main} follows directly from the following lemma, which also precisely points out the correspondence between both mesh patterns and both kinds of windmills.

\begin{lemma}\label{lem:guil_main_lemma} 
Let $\pi \in S_n$. 
\begin{itemize}
\item[(a)] $\gamma_w(\pi)$ contains \wma\ if and only if $\pi$ contains $p_1$,
\item[(b)] $\gamma_w(\pi)$ contains \wmb\ if and only if $\pi$ contains $p_2$.
\end{itemize}
\end{lemma}

\begin{proof}
We provide the proof for (a) (then (b) follows from (a) by reflection via Observation~\ref{obs:mirror}). 

At the first step we modify the pattern $p_1$ in a way that simplifies some technical details.
Consider the mesh pattern
\[q_1=(25314, \ \{0,1\}\times\{2,3,4\} \  \cup \  \{4,5\}\times\{1, 2,3\}).\]
We show that a permutation $\pi$ contains $p_1$ if and only if it contains $q_1$, refer to Figure~\ref{fig:p2q}. 
Assume that $\pi$ contains $p_1$, and the pattern $25314$ of $p_1$ is realized as $becad$ where $a<b<c<d<e$.
Then, in the plot, we can replace the point $e$ by a left-minimum point in the cell $(1,4)$,
then $b$ by the top-most point in $(0,2)\cup(1,2)$,
then $a$ by a right-maximum point in $(4,1)$,
and finally $d$ by the bottom-most point in $(4,3)\cup(5,3)$. 
This shows that if a permutation contains $p_1$ then it contains $q_1$. 
The converse implication is trivial.

\begin{figure}[!h]
\begin{center}
\includegraphics[scale=0.9]{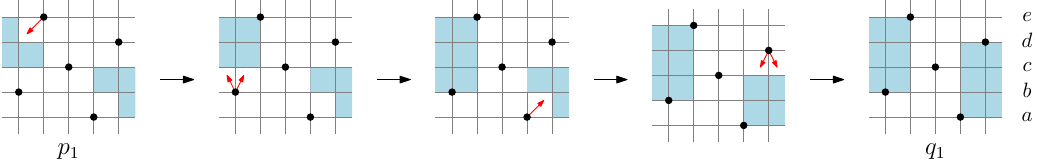}
\end{center}
\caption{An occurrence of $p_1$ implies an occurrence of $q_1$.}
\label{fig:p2q}
\end{figure}

We now prove that $\gamma_w(\pi)$ contains \wma\ if and only if $\pi$ contains $q_1$.

$(\Leftarrow)$ Let $\pi$ be a permutation that contains $q_1$, 
and consider the diagonal representative of $\gamma_w(\pi)$. 
Assume, as above, that the pattern $25314$ of $p_1$ is realized as $becad$.
The shaded regions of $q_1$ imply that no rectangle~$r_x$ with label $x$ such that $b < x < e$ is inserted earlier than~$e$, 
and that no rectangle~$r_x$ with $a < x < d$ is inserted later than~$a$. 
It follows that just after inserting the rectangle $r_e$, the staircase contains 
a horizontal segment $s_e$ just above~$r_e$ and a vertical segment~$s_b$ just to the right from~$r_b$, and these segments (shown by red in Figure~\ref{fig:guil_proof_d1}) meet in a~$\tr$ joint. 
Similarly, just before inserting the rectangle~$r_a$, the staircase contains 
a horizontal segment~$s_a$ just under~$r_a$ and a vertical segment~$s_d$ just to the left from~$r_d$,
and these segments (shown by green in Figure~\ref{fig:guil_proof_d1}) meet in a~$\tl$ joint. 
Due to the presence of~$r_c$, we know that~$s_d$ does not coincide with~$s_b$.

Now we show that $\gamma_w(\pi)$ contains a windmill \wma.
Traverse the segment $s_b$ from below to above. 
Due to $s_a$, the segment $s_b$ can not reach the upper side of $\rr$, as it is blocked by a horizontal segment $s_{a'}$ 
(which is possibly $s_a$). Now we traverse $s_{a'}$ to the right. Due to the existence of $s_d$, the segment $s_{a'}$ cannot reach the right side of $\rr$, as it is be blocked by a vertical segment $s_{d'}$ (which is possibly $s_d$). We continue traversing segments in this way 
and due to $s_a$, $s_d$, $s_e$, $s_b$ we never reach the boundary of $\rr$.
Since the process is finite, a windmill \wma\ will eventually be obtained.

\begin{figure}[!h]
\begin{center}
\includegraphics[scale=1]{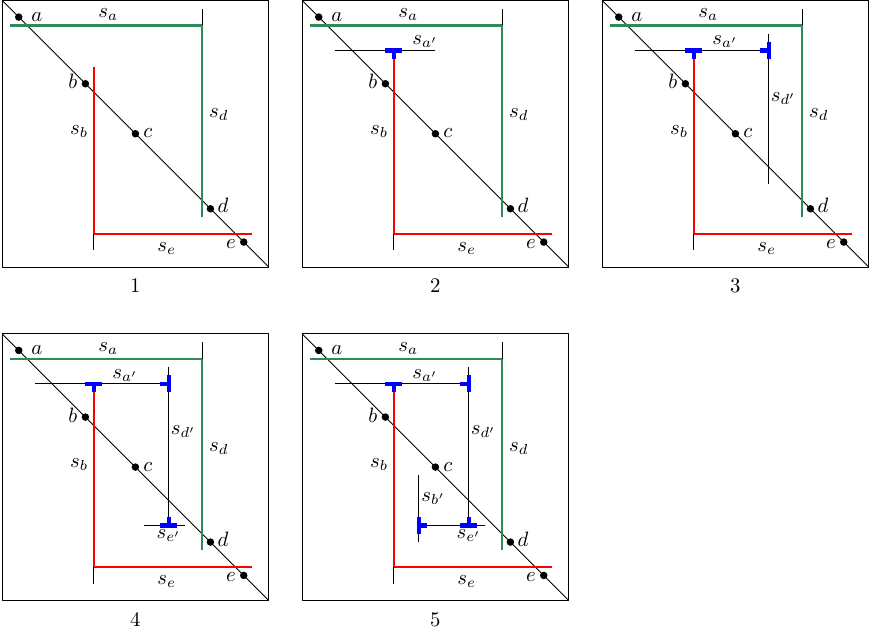} 
\end{center}
\caption{Illustration for $(\Leftarrow)$ in the proof of Theorem~\ref{thm:guil_main}: $q_1$ implies $\wma$.}
\label{fig:guil_proof_d1}
\end{figure}

$(\Rightarrow)$ Let $\rc$ be a rec\-tan\-gu\-la\-ti\-on containing \wma.
Label by $r_a, \, r_b, \, r_d, \, r_e$ the rectangles as shown in Figure~\ref{fig:guil_proof_d2}:
$r_a$ is the rectangle whose bottom-right corner is the top-right corner of the windmill,
$r_d$ is the upmost rectangle whose left side is included in the right vertical segment of the windmill, and similarly for $r_e$ and $r_b$.
Finally, let $r_c$ be any rectangle in the region bounded by the windmill. Then we have
$a<b<c<d<e$ in the diagonal ordering. On the other hand, in $P_w(\rc)$ we have
$b \prec_w e \prec_w c\prec_w a\prec_w d$, which gives the pattern $25314$ in any linear extension 
$\pi$ of this poset.
It remains to show that there are no points in the shaded regions from the plot of $q_1$.
Suppose there is a point in the region $\{0,1\} \times \{2,3,4\}$.
Then there exists a rectangle $r_x$ such that $b<x<e$, which is inserted earlier than $e$.
By Observation~\ref{obs:lrab}(1),
all rectangles $r_x$ such that $b<x<e$ are contained in the region shown in grey in Figure~\ref{fig:guil_proof_d2}. 
However, all rectangles $r_x$ included in this region satisfy $e \prec_w x$, 
hence $r_x$ can not be inserted earlier than $r_e$. 
\end{proof}

\begin{figure}[!h]
\begin{center}
\includegraphics[scale=1]{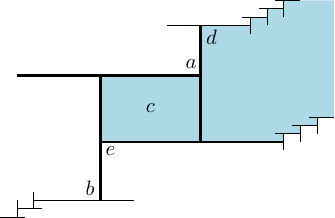}
\end{center}
\caption{Illustration for $(\Rightarrow)$ in the proof of Theorem~\ref{thm:guil_main}: $\wma$ implies $q_1$.}
\label{fig:guil_proof_d2}
\end{figure}

\subsection{New bijections and permutation classes for guillotine rectangulations}\label{sec:windmill_bij}

In this section we discuss specializations of the bijections 
$\beta_{\mathsf{TB}}$, $\beta_{\mathsf{CTB}}$, $\beta_{\mathsf{B}}$, 
$\beta_{\mathsf{2C}}$, $\beta_{\mathsf{C2C}}$ 
from Theorems~\ref{thm:weak_bijections} and~\ref{thm:strong_bijections} to the case of guillotine rec\-tan\-gu\-la\-ti\-ons.

\medskip

\noindent\textbf{Weak guillotine rec\-tan\-gu\-la\-ti\-ons.} Basically, the respective permutation classes are obtained by adding $p_1$ and $p_2$ to the forbidden patterns. 
The next lemma makes it possible to describe some of them by fewer patterns.

\begin{lemma}\label{lem:less_pattern}
The following identities between permutation classes hold:

\begin{center}
\begin{tabular}{lll}
\textit{1.} $\mathsf{Av}(2\underline{41}3,  \, 3\underline{41}2, \,  p_1) = \mathsf{Av}(2413,\, 3\underline{41}2)$,
&  \ \ \ \ \
&
\textit{2.} $\mathsf{Av}(2\underline{41}3,  \, 3\underline{14}2, \,  p_1) = \mathsf{Av}(2413,\, 3\underline{14}2)$, 
\vspace{6pt} \\ 
\textit{3.} $\mathsf{Av}(3\underline{14}2,  \, 2\underline{14}3, \,  p_2) = \mathsf{Av}(3142,\, 2\underline{14}3)$,
& \ \ \ \ \
&
\textit{4.} $\mathsf{Av}(3\underline{14}2,  \, 2\underline{41}3, \,  p_2) = \mathsf{Av}(3142,\, 2\underline{41}3)$.
\end{tabular}
\end{center}
\end{lemma}

\begin{proof} We provide a detailed proof of (1). The proof of (2) is similar, and the proofs of~(3) and~(4) are obtained by taking complements.

In~(1), the implication $\supseteq$ is obvious. 
To prove $\subseteq$, we need to show
that if $\pi$ contains $2413$, then it contains $2\underline{41}3$, $3\underline{41}2$, or $p_1$.

Let $bead$, where $a<b<d<e$, be a~\emph{(vertically) shortest} occurrence of $2413$, that is,
an occurrence with the smallest possible $e - a$. 
If it is not a part of $25314$, then we 
can replace the point~$e$ by the rightmost point in the region
$(2,3) \cup (2,4)$; this yields an
occurrence of $2\underline{41}3$
(refer to the left part of Figure~\ref{fig:ttp2ts}).

Now assume that our occurrence of $2413$ is a part of $25314$ (refer to the right part of Figure~\ref{fig:ttp2ts}).
The regions $(1,4)$, $(2,4)$, $(3,4)$, $(0,3)$, 
$(1,3)$, $(4,2)$, $(5,2)$, $(2,1)$, $(3,1)$, $(4,1)$ are empty because otherwise we have a shorter occurrence of the pattern $2413$.
If the region $(0,4)$ is not empty, then --- applying the same argument as above and using the fact that the region $(2,4)$ is empty --- we obtain an occurrence of $3\underline{41}2$. 
Similarly, we obtain $3\underline{41}2$ if $(5,1)$ is not empty.
And if both regions $(0,4)$ and $(5,1)$ are empty,
then our assumed $25314$ is an occurrence of~$p_1$.
\end{proof}

\begin{figure}[!h]
\begin{center}
\includegraphics[scale=0.9]{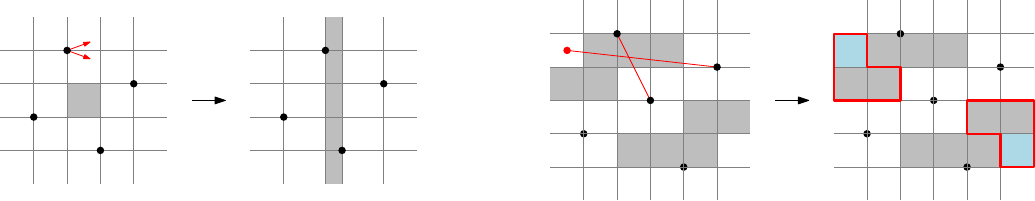} 
\end{center}
\caption{Illustration to the proof of Lemma~\ref{lem:less_pattern}.}
\label{fig:ttp2ts}
\end{figure}

\noindent\textit{Remark.} 
Note that it is not possible to ``cancel'' the patterns that occur on both sides of these identities. 
For example, $\mathsf{Av}(2\underline{41}3, \,  p_1) = \mathsf{Av}(2413)$ is false: $526314$ is a counterexample.

\medskip

Combining Lemma~\ref{lem:less_pattern} with Lemma~\ref{lem:guil_main_lemma} we obtain:
\begin{proposition}\label{prop:nonsep}
  The following families of permutations are in   bijection   
  (respectively via $\beta_{\mathsf{TB}}$, $\beta_{\mathsf{CTB}}$, and~$\beta_{\mathsf{B}}$) with weak rec\-tan\-gu\-la\-ti\-ons that avoid $\wma$:
\begin{enumerate}
\item
$\mathsf{Av}(2413,\,3\underline{41}2)$,
\item
$\mathsf{Av}(2\underline{14}3,\, 3\underline{14}2,\, p_1)$,
\item
$\mathsf{Av}(2413,\,3\underline{14}2)$.
\end{enumerate}
\end{proposition}
Similarly, via $\beta_{\mathsf{TB}}$, $\beta_{\mathsf{CTB}}$, and $\beta_{\mathsf{B}}$,  weak rec\-tan\-gu\-la\-ti\-ons that avoid $\wmb$ correspond respectively to 
the families $\mathsf{Av}(2\underline{41}3,\,3\underline{41}2,\, p_2)$, 
$\mathsf{Av}(2\underline{14}3,3142)$, 
and $\mathsf{Av}(2\underline{41}3,\,3142)$.  

\medskip

Proposition~\ref{prop:nonsep} sheds new light on some previously known results. 
Weak rec\-tan\-gu\-la\-ti\-ons of size $n$ are in a simple bijection with 2-orientations~\cite{FOR95}  
 on rooted simple quadrangulations (embedded on the sphere) with $n+1$ faces, i.e., orientations of the edges not incident to the root-face such that vertices not incident (resp. incident) to the 
 root-face have outdegree $2$ (resp. $0$).  
Moreover, a weak rec\-tan\-gu\-la\-ti\-on has a~$\wma$ if and only if the 2-orientation has a clockwise cycle (more precisely, the occurrence of a~$\wma$ corresponds
to the occurrence of a clockwise 4-cycle, and the presence of a clockwise cycle implies the presence of a clockwise 4-cycle). 
Since any rooted simple quadrangulation has a unique 2-orientation with no clockwise cycle~\cite{Pr93,Fe04}, weak rec\-tan\-gu\-la\-ti\-ons of size $n$ with no $\wma$ are thus in bijection
with rooted simple quadrangulations with $n+1$ faces, which are themselves in bijection with rooted non-separable maps with $n+1$ edges, whose
 counting coefficients are $a_n=\frac{2 \, (3n)!}{(n+1)! \, (2n+1)!}$ as shown in~\cite{Tu63,Sc98}. 
  
The fact that $\mathsf{Av}(2413,\, 3\underline{14}2)$ is in  bijection with rooted non-separable maps was already proved in~\cite{DGW96} (via isomorphic
generating trees) and in~\cite{BBF10}, by specializing a bijection between Baxter permutations and plane bipolar orientations: this bijection has the property that Baxter
permutations with no $2413$ correspond to 
plane bipolar orientations with no ROP (right-oriented piece), and moreover any  rooted non-separable map has a unique plane bipolar orientation with no ROP. Our bijection
 between $\mathsf{Av}(2413,\, 3\underline{14}2)$ and weak rec\-tan\-gu\-la\-ti\-ons with no $\wma$  is very analogous, since plane bipolar orientations are in a simple bijection~\cite{FOR95} with 2-orientations, such that the occurrence of a ROP corresponds to the occurrence of a~counterclockwise 4-cycle (or the occurrence of a~clockwise 4-cycle, upon reflection). 

Let us also mention that the coefficients $a_n$ are well-known to count 2-stack sortable permutations~\cite{Z92}. 
In~\cite{DGG98}, a correspondence with $\mathsf{Av}(2413,\, 3\underline{14}2)$ has been obtained 
via a chain of several bijections 
relating permutation classes (and relying on isomorphic generating trees).  Along that chain after $\mathsf{Av}(2413,\, 3\underline{14}2)$ is the class $\mathsf{Av}(2413,\, 3\underline{41}2)$ (see~\cite[Fig.3]{DGG98});  this 
corresponds to the link between the first and second item in Proposition~\ref{prop:nonsep}.  Recently a more direct bijective link between 2-stack sortable permutations and rooted non-separable maps has been established via \textit{fighting fish}~\cite{Fa18,DuHe22}. 

\bigskip

We now further specialize the bijections (for weak classes) to the guillotine case:

\begin{proposition}\label{prop:guill}
The following families of permutations are in  
bijection (respectively via $\beta_{\mathsf{TB}}$, $\beta_{\mathsf{CTB}}$, and $\beta_{\mathsf{B}}$) with weak guillotine rec\-tan\-gu\-la\-ti\-ons:
\begin{enumerate}
\item
$\mathsf{Av}(2413,\,3\underline{41}2,\,p_2)$,
\item
$\mathsf{Av}(2\underline{14}3,  \, 3142,\, p_1)$,
\item
$\mathsf{Av}(2413,\,3142)$.
\end{enumerate}
\end{proposition}
\begin{proof}
It follows from Theorems~\ref{thm:weak_bijections} and~\ref{thm:guil_main} that weak guillotine rec\-tan\-gu\-la\-ti\-ons are in bijection (respectively via $\beta_{\mathsf{TB}}$, $\beta_{\mathsf{CTB}}$, and $\beta_{\mathsf{B}}$) with
$\mathsf{Av}(2\underline{41}3, 3\underline{41}2, p_1, p_2)$, 
$\mathsf{Av}(2\underline{14}3, 3\underline{14}2, p_1, p_2)$, and
$\mathsf{Av}(2\underline{41}3, 3\underline{14}2, p_1, p_2)$. 
By Lemma~\ref{lem:less_pattern} we have 

\begin{center}
\begin{tabular}{lll}
$\mathsf{Av}(2\underline{41}3, 3\underline{41}2, p_1, p_2)=\mathsf{Av}(2413, 3\underline{41}2, p_2)$,
& \ \ \ &
$\mathsf{Av}(2\underline{41}3, 3\underline{41}2, p_1, p_2)=\mathsf{Av}(2413, 3\underline{41}2, p_2)$, 
\vspace{6pt} \\
$\mathsf{Av}(2\underline{41}3, 3\underline{14}2, p_1, p_2)=\mathsf{Av}(2\underline{41}3, 3142, p_1)$,
& \ \ \ & 
$\mathsf{Av}(2\underline{14}3, 3\underline{14}2, p_1, p_2)=\mathsf{Av}(2\underline{14}3, 3142, p_1)$.
\end{tabular}
\end{center} 
The combination of the two identities for $\mathsf{Av}(2\underline{41}3, 3\underline{14}2, p_1, p_2)$ implies that this class is equal to $\mathsf{Av}(2413,\,3142)$.
\end{proof}

Part 3 of Proposition~\ref{prop:guill} recovers the bijection $\beta_{\mathsf{S}}$ from Theorem~\ref{thm:weak_bijections}(4);
parts 1 and 2 are new results.  

\medskip

\noindent\textbf{Strong guillotine rec\-tan\-gu\-la\-ti\-ons.} 
Here we just add $p_1$ and $p_2$ to the patterns that define 
2-clumped permutations and co-2-clumped permutations, and we did not 
find a way to describe these classes with fewer patterns.
However, \textbf{this is the first known representation of strong guillotine rec\-tan\-gu\-la\-ti\-ons by permutation classes}.

\begin{proposition}\label{prop:guil_cor_S}
The following families of permutations are in  
bijection (respectively via $\beta_{\mathsf{2C}}$ and $\beta_{\mathsf{C2C}}$) with strong guillotine rec\-tan\-gu\-la\-ti\-ons:
\begin{enumerate}
\item the $\{p_1, p_2\}$-avoiding 2-clumped permutations,
\item the $\{p_1, p_2\}$-avoiding co-2-clumped permutations.
\end{enumerate}
\end{proposition}

\medskip

\noindent\textbf{Summary.} Proposition~\ref{prop:guil_cor_S}(1) was conjectured and communicated to us by Merino and Mütze~\cite{MMpc}.
They found the mesh patterns $p_1$ and $p_2$ experimentally, as a part of their study of patterns in rec\-tan\-gu\-la\-ti\-ons.
This conjecture was our starting point for the study presented in this section.
As our results --- mainly Theorem~\ref{thm:guil_main} and Lemma~\ref{lem:guil_main_lemma} --- show, these two patterns 
not just define a permutation class in bijection with with strong guillotine rec\-tan\-gu\-la\-ti\-ons,
but they generally ``encode the windmills in the language of permutations''. As such, 
these results belong to the study of representing patterns in rec\-tan\-gu\-la\-ti\-ons by patterns in permutations, 
which was suggested in~\cite[Section~11]{MM23} as an open question.
Our Lemma~\ref{lem:cross} is another instance of correspondence between these kinds of patterns,
see also~\cite{AsinowskiBanderier23} for more results of this kind.

\subsection{Enumeration of strong guillotine rectangulations}

\noindent\textbf{{Generating the enumerating sequence.}}
A straightforward way to generate the enumerating sequence
for strong guillotine rec\-tan\-gu\-la\-ti\-ons is counting multiplicities.
A~\emph{multiplicity} of a weak rec\-tan\-gu\-la\-ti\-on $\rc$
is the number of strong rec\-tan\-gu\-la\-ti\-ons whose union constitutes $\rc$.
Every segment $s$  contributes $\binom{a+b}{a}$ to the multiplicity of $\rc$,  
where $a$~and~$b$ are the numbers of neighbors of $s$ from both sides.
The total multiplicity of a rec\-tan\-gu\-la\-ti\-on $\rc$ is the product of such binomial coefficients taken over all its segments. 

For strong guillotine rec\-tan\-gu\-la\-ti\-ons, we can use the same argument as in our proof of 
Proposition~\ref{prop:guillotineenum}, but taking into account the multiplicities.
Let $\rc$ be a vertical guillotine rec\-tan\-gu\-la\-ti\-on of size $>1$, and let~$s$ be its leftmost cut.
Denote by $\rc_L$ and $\rc_R$ the left and the right subrec\-tan\-gu\-la\-ti\-ons
separated by~$s$.
If the multiplicity of $\rc_L$ is $m_1$,
and that of $\rc_R$ is $m_2$,
then the multiplicity of $\rc$ is $m_1 m_2 \binom{a+b}{a}$,
where~$a$ and~$b$ are the numbers of left and right neighbors of $s$. 

Therefore, we have to keep track of the numbers of segments that have an endpoint
on the sides of~$R$. This leads to a recurrence in five variables.
Denote by $S(n,\ell,t,r,b)$ the number of strong guillotine rec\-tan\-gu\-la\-ti\-ons of size $n$
with $\ell$, $t$, $r$, $b$ endpoints of segments on the left, top, right, bottom side.
Further, denote by $S_V(n,\ell,t,r,b)$ and $S_H(n,\ell,t,r,b)$
the numbers of vertical and, respectively, horizontal  strong guillotine rec\-tan\-gu\-la\-ti\-ons with these parameters.
(To keep the expressions more compact, here we regard the rec\-tan\-gu\-la\-ti\-on of size $1$ as both vertical and horizontal.)
Then we have the following recurrence.

For $n=1$:
\[S(1,\ell,t,r,b) = S_V(1,0,0,0,0)=S_H(1,0,0,0,0)= 
\left\{
\begin{array}{lll}
1, & & (\ell,t,r,b) = (0,0,0,0), \\
0, & & (\ell,t,r,b) \neq (0,0,0,0).
\end{array}
\right.\]

For $n>1$:
\[
S_V(n,\ell,t,r,b) = \displaystyle{\sum S_H(n',\ell,t',r',b') \cdot S(n-n',\ell',t-1-t',r,b-1-b') \cdot \binom{r'+\ell'}{r'}},
\]
where the sum is taken over $1 \leq n' \leq n-1$
and $0 \leq t', r', b', \ell' \leq  n$.

This rule is illustrated in Figure~\ref{fig:str}.

\begin{figure}[!h]
\begin{center}
\includegraphics[height=.15\textheight]{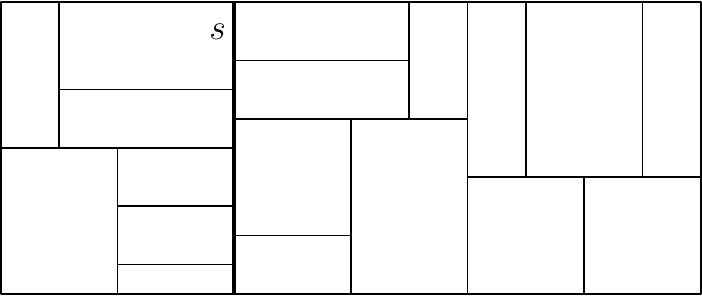}
\caption{Illustration of the recurrence for counting strong guillotine rec\-tan\-gu\-la\-ti\-ons.
The multiplicity of this rec\-tan\-gu\-la\-ti\-on 
is the product of 
the multiplicity of the left part,
the multiplicity of the right part,
and the binomial coefficient $\binom{7}{4}$.}
\label{fig:str}
\end{center}
\end{figure}

For $S_H(n,\ell,t,r,b)$ we have a similar expression, but for computations we can use 
\[
\begin{array}{c}
S_V(n,\ell,t,r,b) = S_V(n,r, t, \ell,b)= S_V(n,\ell,b,r,t) = S_V(n,r,b,\ell,t)= \\
=S_H(n,t, \ell,b,r) = S_H(n,b, \ell,t,r) = S_H(n,t, r,b,\ell) = S_H(n,b, r,t,\ell).
\end{array}
\]

Finally, for $n>1$ we have
\[S(n,\ell,t,r,b) = S_H(n,\ell,t,r,b)+S_V(n,\ell,t,r,b)\]

We implemented the recurrence in Maple, 
and obtained the first numbers in the enumerating sequence 
of strong guillotine rec\-tan\-gu\-la\-ti\-ons of size $n$:
\begin{center}
{\begin{tabular}{|r||r||r||r|}
\hline
$n=1\ldots 8$ & $n=9\ldots 16$ & $n=17\ldots 24$& $n=25\ldots 32$\\
\hline
1 & 138100 & 1143606856808 & 23673987861077379184 \\
2 & 926008 & 9072734766636 & 201493429381831155064 \\
6 & 6418576 & 72827462660824 & 1725380127954612191928 \\
24 & 45755516 & 590852491725920 & 14858311852609658166276 \\
114 & 334117246 & 4840436813758832 & 128634723318443875261706  \\
606 & 2491317430 & 40009072880216344 & 1119203662581349129800254 \\
3494 & 18919957430 & 333419662183186932 & 9783477314800654941937182  \\
21434 & 146034939362 & 2799687668599080296 & 85899976772035554402923170 \\ \hline
\end{tabular}
}
\end{center}

This sequence has no OEIS entry at the time of writing.
 
\medskip

\noindent\textbf{{Asymptotic bounds.}} We now would like to show that guillotine rec\-tan\-gu\-la\-ti\-ons are rare among strong rec\-tan\-gu\-la\-ti\-ons of size $n$, as $n$ gets large. Precisely, we will show that the exponential growth rate of strong guillotine rec\-tan\-gu\-la\-ti\-ons is bounded from above by a constant $\approx 13.081$, hence is strictly smaller than the exponential
growth rate of all strong rec\-tan\-gu\-la\-ti\-ons, which, as mentioned above, is known to be~$27/2$~\cite{FNS21}. 

We will make use of asymptotic results in~\cite{FNS21} (to be slightly extended below in Lemma~\ref{lem:rho_extended}) on so-called \emph{arbitrary rec\-tan\-gu\-la\-ti\-ons}, which are rec\-tan\-gu\-la\-ti\-ons allowing for points where 4 rectangles meet, called \emph{special points}. These are considered in the strong equivalence sense.  
Let $a_{n,k}$ be the number of arbitrary rec\-tan\-gu\-la\-ti\-ons of size $n$ with $k$ special points, let $a_n(v)=\sum_ka_{n,k}v^k$, and let $A(z,v)=\sum_n a_n(v)z^n$ be the
associated counting series. For fixed $v$, let $\rho(v)$ be the radius of convergence of $z\to A(z,v)$, i.e., $1/\rho(v)$ is the exponential growth rate of $a_n(v)$. 
It has been shown in~\cite[Thm.~4.3]{FNS21} that, for $v\geq 0$, 
\begin{equation}\label{eq:rho}
\rho(v)=\frac{2(2+v)}{2v^2+18v+27+(9+4v)^{3/2}}.
\end{equation}
\begin{lemma}\label{lem:rho_extended}
There exist non-negative coefficients $\tilde{a}_{n,k}$ such that, with $\tilde{a}_n(v):=\sum_k\tilde{a}_{n,k}v^k$, we have $a_n(v)=\tilde{a}_n(v+2)$, so that 
$a_n(v)\geq 0$ for $v\geq -2$. Moreover, for $v\in(-2,0)$, $\rho(v)$ is still given by~\eqref{eq:rho}.
\end{lemma}
\begin{proof}
An arbitrary rec\-tan\-gu\-la\-ti\-on is called \emph{reduced} if it 
avoids both $\izwall$ and $\iswall$.   
If we let $\tilde{a}_{n,k}$ be the number of reduced arbitrary rec\-tan\-gu\-la\-ti\-ons of size $n$ with $k$ special points, and let $\tilde{a}_n(v)=\sum_k\tilde{a}_{n,k}v^k$, then we have 
\[a_n(v)=\tilde{a}_n(v+2).\]
Indeed, an arbitrary rec\-tan\-gu\-la\-ti\-on yields a reduced one by contracting the inner segment of each $\izwall$ or $\iswall$, turning it into a~$\cross$. 
Conversely, a reduced arbitrary rec\-tan\-gu\-la\-ti\-on lifts to a set of arbitrary rec\-tan\-gu\-la\-ti\-ons by choosing, for each special point $\cross$, whether
it stays unchanged, is expanded into~$\izwall$, or is expanded into $\iswall$.   

The encoding of arbitrary rec\-tan\-gu\-la\-ti\-ons by weighted quadrant walks that is obtained in~\cite[Sec.2.4]{FNS21} (relying on a bijection in~\cite{kenyon2019bipolar}) can be specialized to reduced arbitrary rec\-tan\-gu\-la\-ti\-ons: with the ter\-mi\-nol\-o\-gy of~\cite{FNS21} (where the study is done in the dual setting of transversal structures), forbidding $\iswall$ amounts to forbidding consecutive face-steps, and forbidding $\izwall$ can easily be encoded in the weight affected to face-steps. All calculations done as in~\cite[Sec.4]{FNS21} 
(details omitted) one finds that, 
if $\tilde{\rho}(v)$ denotes the radius of convergence of $v\to\sum_n\tilde{a}_n(v)z^n$ for $v> 0$ (which equals $\rho(v-2)$ since $\tilde{a}_n(v)=a_n(v-2)$), then the obtained expression of $\tilde{\rho}(v)$ matches the right-hand side of~\eqref{eq:rho} where $v$ is substituted by $v-2$. 
\end{proof}

In a~strong rec\-tan\-gu\-la\-ti\-on the \emph{enclosing 4-gon} of a windmill is the 4-gon extracted from the union of the 4 constituting segments. The windmill is called \emph{simple} 
if there is no segment leaving a point on a~side of the enclosing 4-gon towards the exterior of the 4-gon. 
A~\emph{small windmill} is a~simple windmill with a~single rectangular region
inside the enclosing 4-gon. 
Let $b_n$ be the number of rec\-tan\-gu\-la\-ti\-ons of size $n$ with no small windmill. Obviously, $b_n$ is an upper bound on the number of guillotine rec\-tan\-gu\-la\-ti\-ons of size $n$.

\begin{proposition}\label{prop:simple_windmill}
The exponential growth rate of $b_n$ is bounded from above
by the unique positive root $x_0\approx 13.155$ of the polynomial $2x^5 - 29x^4 + 36x^3 - 8x^2 - 8$. 
\end{proposition}
\begin{proof} 
Let $\hat{a}_{n,k}$ be the number of rec\-tan\-gu\-la\-ti\-ons of size $n$ with $k$ small windmills; and let $\hat{A}(z,v)=\sum_{n,k}\hat{a}_{n,k}z^nv^k$ be the associated
counting series. Then we have $A(z,2vz)=\hat{A}(z,1+v)$. Indeed, starting from an arbitrary rec\-tan\-gu\-la\-ti\-on, each special point $\cross$ can be expanded into either $\wma$ or $\wmb$ (here these
symbols are to be understood as small windmills);  then 
 these form an arbitrary subset of all small windmills in the obtained rec\-tan\-gu\-la\-ti\-on.  
 Hence, letting $B(z)=\sum_nb_nz^n$, we have
\[
B(z)=\hat{A}(z,0)=A(z,-2z).
\]
For $0\leq z\leq 1$ such that $z<\rho(-2z)$, the function $A(.,.)$ is analytic at $(z,-2z)$ (this follows from the fact that, by continuity of $\rho(.)$, 
there exist $\epsilon,\eta>0$ such that $\sum_{n,k}\tilde{a}_{n,k}(z+\epsilon)^n(2-2z+\eta)^k<+\infty$). 
Hence $B(.)$ is analytic at $z$. Thus, letting $z_0$ be the 
smallest positive root of the equation $z=\rho(-2z)$, the function~$B(.)$ is analytic at $z$ for $0\leq z<z_0$. By Pringsheim's theorem, the radius 
of convergence of $B(z)$ is at least~$z_0$, hence the exponential growth rate of $b_n$ is at most $1/z_0$.  From the above expression of $\rho(v)$, 
 we find that $1/z_0\approx 13.155$ is the unique positive root of the polynomial $2x^5 - 29x^4 + 36x^3 - 8x^2 - 8$. 
 \end{proof}

 Now let $\ob_n$ be the number of rec\-tan\-gu\-la\-ti\-ons of size $n$ with no simple windmill, which again is an upper bound on the number of guillotine rec\-tan\-gu\-la\-ti\-ons of size $n$. 
 \begin{proposition}
 The exponential growth rate of $\ob_n$ is bounded from above
 by $13.081$. 
 \end{proposition}
 \begin{proof}
 For any fixed $k\geq 1$, a simple windmill is called \emph{$k$-small} if the rec\-tan\-gu\-la\-ti\-on inside the enclosing 4-gon is a~guillotine rec\-tan\-gu\-la\-ti\-on of size at most $k$. 
 Let $b_n^{(k)}$ be the number of rec\-tan\-gu\-la\-ti\-ons of size $n$ with no $k$-small windmill (in particular $b_n=b_n^{(1)}$, and $\ob_n\leq b_n{(k)}$ for all $k\geq 1$); let $B^{(k)}(z)=\sum_nb_n^{(k)}z^n$. With $g_n:=S(n)$ the number of guillotine rec\-tan\-gu\-la\-ti\-ons of size $n$, the 
 argument in Proposition~\ref{prop:simple_windmill}  extends to give the equation
  \[
B^{(k)}(z)=A\Big(z,-2\sum_{i=1}^kg_iz^i\Big).
\]
 Letting $z_0^{(k)}$ be the smallest positive solution of the equation $z=\rho(-2\sum_{i=1}^kg_iz^i)$, by the same argument as in Proposition~\ref{prop:simple_windmill}, 
 the radius of convergence of $B^{(k)}(z)$ is at least $z_0^{(k)}$, hence $1/z_0^{(k)}$ is an upper bound on the exponential growth rate of $b_n^{(k)}$, and also of $\ob_n\leq b_n^{(k)}$.  We find 
 that as $k$ increases, $1/z_0^{(k)}$ (which decreases) rapidly approaches a constant $\approx 13.081$ (upper approximation).
 \end{proof}

 \medskip
 
\noindent\textit{Remark.}
For fixed $n$, $b_n^{(k)}$ weakly decreases with $k$ and stabilizes to $\ob_n$.  
We expect that (for any fixed $k\geq 1$) $1/z_0^{(k)}$ is the exponential growth rate of $b_n^{(k)}$, and that, as $k\to\infty$, it converges to the exponential growth rate of $\ob_n$, which thus should be $\approx 13.081$.   
However, forbidding only simple windmills does not seem to give a close upper bound on the exponential growth rate of guillotine rec\-tan\-gu\-la\-ti\-ons. 
 Indeed, from the table of the initial counting coefficients $g_n$, and applying acceleration of convergence (see, e.g.,~\cite[Sec.6]{guitter1999hamiltonian}) to the ratio $g_{n+1}/g_n$, the exponential growth rate of $g_n$ seems to be $ \approx 10.24$. 
 
\medskip

Finally, we discuss lower bounds. 
By Proposition~\ref{prop:guillotineenum}, the number of weak guillotine rec\-tan\-gu\-la\-ti\-ons of size $n$ is the $(n-1)$th Schr\"oder number.
Therefore, the exponential growth rate of Schröder numbers, $3+2\sqrt{2}\approx 5.828$,
is a~``trivial'' lower bound on the exponential growth rate of strong guillotine rec\-tan\-gu\-la\-ti\-ons. In order to give a better bound,
we consider weak guillotine rec\-tan\-gu\-la\-ti\-ons 
where every 2-sided segment (that is, a segment with at least one neighbor on each side)
has weight $2$.
This will give a lower bound, since the neighbors of every 2-sided segment
can be shuffled in at least two ways. We
adapt the decomposition
from our proof of Proposition~\ref{prop:guillotineenum} as follows.
Let $G=G(x,y)$ be the generating function for weak guillotine rec\-tan\-gu\-la\-ti\-ons,
where the variable $x$ is for the size, and $y$ for the number of 2-sided segments.
Further, let $G_0=G_0(x,y)$ be the generating function for weak guillotine rec\-tan\-gu\-la\-ti\-ons
that have \textit{no segment} with an endpoint on the left side of $\rr$,
and let $G_1=G_1(x,y)$ be the generating function for weak guillotine rec\-tan\-gu\-la\-ti\-ons
that have \textit{at least one segment} with an endpoint on the left side of~$\rr$.
Finally, let $H=H(x,y)$ and $V=V(x,y)$ be the generating functions for horizontal and, respectively, vertical weak guillotine rec\-tan\-gu\-la\-ti\-ons.
Then we have $H=V$, $G=x+H+V$, $G_0 = xG+x$, and $G_1=(1-x)G-x$, and the decomposition 
of a vertical guillotine rec\-tan\-gu\-la\-ti\-on by its leftmost cut
leads to the following weighted version of equation~\eqref{eq:HV}:
\[
V = xG + H\Big(G_0 + yG_1\Big).
\]
The solution of this system yields
\[
G(x,2)=\frac{1+x-x^2-\sqrt{1-6x-5x^2+2x^3+x^4}}{2(2-x)},
\]
and its dominant singularity gives us the following lower bound.

\begin{proposition}
The exponential growth rate of the number of strong guillotine rec\-tan\-gu\-la\-ti\-ons
is bounded from below by $\frac{1}{2}\Big(1+\sqrt{13-8\sqrt{2}}\Big)\Big(3+2\sqrt{2}\Big) \approx 6.699$.
\end{proposition}

This strategy can be pushed further: for a fixed threshold value $t\geq 1$, every segment with $i$ neighbors on one side and $j$ neighbors on the other side is weighted by $\binom{i'+j'}{i'}$, where $i'=\mathrm{min}(i,t)$ and $j'=\mathrm{min}(j,t)$. 
The exponential growth rate for any fixed $t$ should be computable (by the approach of~\cite[VII.6.3]{FlaSe}), 
and grow with $t$, giving better and better lower bounds. 
The complexity of the decomposition and of computations, however, will rapidly explode as $t$ grows.

\section*{Acknowledgments}
The work on this paper began during the \textit{Order \& Geometry Workshop}, Ciążeń, Poland, 13--18 September 2022.
The research of Andrei Asinowski is partially supported by FWF --- The Austrian Science Fund, grant P32731. 
The research of Stefan Felsner is partially supported by grant DFG FE 340/13-1. 
The research of \'Eric Fusy is partially supported by the projects ANR-20-CE48-0018 (3DMaps) and ANR-19-CE48-0011 (COMBIN\'E).

\bibliographystyle{plain}
\bibliography{rectangulations}

\end{document}